\documentclass[11pt,letterpaper]{article}
\usepackage{amsfonts, amsmath, amssymb, amscd, amsthm, color, graphicx, mathrsfs, wasysym, setspace, mdwlist, calc, float, mathtools, tikz-cd, calc, multirow, makecell, array, caption}
\newcolumntype{?}{!{\vrule width 1pt}}

\numberwithin{equation}{section}

\usepackage{colortbl}
\usepackage{hyperref}
\usepackage{tocloft}
\definecolor{DarkBlue}{RGB}{20,20,200}
\definecolor{DarkGreen}{RGB}{20,120,20}

\usepackage{xcolor}

\usepackage{hyperref}
\hypersetup{linktocpage}

\let\OLDthebibliography\thebibliography
\renewcommand\thebibliography[1]{
  \OLDthebibliography{#1}
  \setlength{\parskip}{1.5pt}
  \setlength{\itemsep}{1.5pt plus 0.3ex}
}

\hypersetup{colorlinks,
    linkcolor={red!50!black},
    citecolor={blue!80!black},
    urlcolor={blue!80!black}}
\usepackage{float}

\renewcommand{\kappa}{\varkappa}
\renewcommand{\phi}{\varphi}
\newcommand{\e}{\varepsilon}

\newcommand{\NN}{\mathbb{N}}
\newcommand{\ZZ}{\mathbb{Z}}
\newcommand{\QQ}{\mathbb{Q}}
\newcommand{\RR}{\mathbb{R}}
\newcommand{\CC}{\mathbb{C}}
\renewcommand{\L}{\mathcal{L}}
\newcommand{\Lw}{\L _{\omega_1, \omega}}

\newcommand{\Mod}{{\rm Mod}}

\renewcommand{\d}{{\rm d}}
\renewcommand{\ll }{\langle\hspace{-.7mm}\langle }
\newcommand{\rr }{\rangle\hspace{-.7mm}\rangle }

\newcommand{\diam}{\operatorname{diam}}
\renewcommand{\H}{\mathcal H}

\newcommand{\act}{\curvearrowright}
\newcommand{\la}{\langle}
\newcommand{\ra}{\rangle}
\newcommand{\lA}{\preccurlyeq_{\scriptscriptstyle A}}
\newcommand{\eA}{\sim_{\scriptscriptstyle A}}
\newcommand{\lPL}{\preccurlyeq_{\scriptscriptstyle PL}}
\newcommand{\ePL}{\sim_{\scriptscriptstyle PL}}
\newcommand{\sGS}{\sim_{G\act S}}

\newtheorem{thm}{Theorem}[section]
\newtheorem*{thm*}{Theorem}
\newtheorem{cor}[thm]{Corollary}
\newtheorem{lem}[thm]{Lemma}
\newtheorem{prop}[thm]{Proposition}
\newtheorem{prob}[thm]{Problem}

\theoremstyle{definition}
\newtheorem{defn}[thm]{Definition}

\theoremstyle{remark}
\newtheorem{rem}[thm]{Remark}

\newtheorem{ex}[thm]{Example}

\newfont{\eufm}{eufm10}

\begin{document}
\title{\vspace*{-5mm}Classifying group actions on hyperbolic spaces}
\author{D. Osin, K. Oyakawa}
\date{\vspace*{-3mm}\it With an appendix by D. Osin and A. Rapinchuk}

\maketitle
\vspace*{-6mm}

\begin{abstract}
For a given group $G$, it is natural to ask whether one can classify all isometric $G$-actions on Gromov hyperbolic spaces. We propose a formalization of this problem utilizing the complexity theory of Borel equivalence relations. In this paper, we focus on actions of general type, i.e., non-elementary actions without fixed points at infinity. Our main result is the following dichotomy: for every countable group $G$, either all general type actions of $G$ on hyperbolic spaces can be classified by an explicit invariant ranging in an infinite-dimensional projective space, or they are unclassifiable in a very strong sense. In terms of Borel complexity theory, we show that the equivalence relation associated with the classification problem is either smooth or $K_\sigma$ complete. Special linear groups $SL_2(F)$, where $F$ is a countable field of characteristic $0$, satisfy the former alternative, while non-elementary hyperbolic (and, more generally, acylindrically hyperbolic) groups satisfy the latter. In the course of proving our main theorem, we also obtain results of independent interest that offer new
insights into algebraic and geometric properties of groups admitting general type actions on hyperbolic spaces.
\end{abstract}

%\tableofcontents

%%%%%%%%%%%%%%%%%%%%%%%%%%%%%%%%%%%%%%%%%%%%%%%%%%%%%%%%%%%%%%%%%%%%%%%%%%%%%%%%%%%

\section{Introduction}

%%%%%%%%%%%%%%%%%%%%%%%%%%%%%%%%%%%%%%%%%%%%%%%%%%%%%%%%%%%%%%%%%%%%%%%%%%%%%%%%%%%

In geometric group theory, groups are often studied through their actions on metric spaces. This approach is particularly effective when the action is ``sufficiently non-trivial" and the metric space exhibits some form of negative curvature. The origins of these ideas trace back to Klein and Poincar\'e, but the modern theory began to take shape in the 1980s when Gromov introduced the concept of an abstract hyperbolic space \cite{Gro}.

Much of the early work on groups acting on hyperbolic spaces was focused on proper cocompact actions, which gave rise to the rich theory of hyperbolic groups. However, many general ideas are applicable beyond these settings, and considering actions of a given group $G$ on different hyperbolic spaces becomes vital in order to recover various features of $G$ (see, for example, \cite{B,F,Osi18}). A fundamental question in this approach is the following.

\begin{prob}\label{prob}
Given a group $G$, classify all isometric actions of $G$ on hyperbolic spaces.
\end{prob}

Our main goal is to formalize this problem and study its complexity using the language of Borel equivalence relations. Over the past three decades, the study of definable equivalence relations has been an active area of research in descriptive set theory. It provides a foundation for analyzing the complexity of classification problems arising in various areas of mathematics. A comprehensive discussion of recent developments can be found in the monographs \cite{Kan} and \cite{Kec24}; for results directly related to group theory, we refer to the survey \cite{ST} and papers \cite{CMMS,Tho13,Tho15}.

In this paper, we focus on general type group actions on hyperbolic spaces, as these are the most useful from the geometric group theory point of view. 

\begin{defn} Recall that an isometric action of a group $G$ on a hyperbolic space $S$ is said to be of \emph{general type} if the limit set of $G$ is infinite and $G$ does not fix any point of the Gromov boundary $\partial S$ (equivalently, $G$ contains two loxodromic isometries with disjoint pairs of fixed points on $\partial S$). 
    Following the standard terminology (see, for example, \cite{MT}), we call a group $G$ \emph{weakly hyperbolic} if it admits a general type action on a hyperbolic space.
\end{defn} 
By a result of Sisto, a group $G$ is weakly hyperbolic if and only if  there exists a generating set $X$ of $G$ such that the Cayley graph $Cay (G,X)$ is hyperbolic and the standard action of $G$ on $Cay(G,X)$ is of general type (see Corollary \ref{Cor:Sisto} for details).

Informally, our results can be summarized as follows. With every countable group $G$, we associate an equivalence relation on a standard Borel space that serves as a formalization of Problem \ref{prob}. In fact, our approach can be used to formalize classification problems for group actions on other types of spaces whose geometry is definable in a certain infinitary logic (see Proposition \ref{Prop:LE}).

Our main result -- Theorem \ref{Thm:main} -- gives a complete classification of possible complexity levels of Problem \ref{prob} for general type actions up to Borel bi-reducibility. More precisely, we prove that every countable weakly hyperbolic group falls into one of the following two classes.

\begin{itemize}
\item {\bf Isotropic groups.} For every group in this category, general type actions on hyperbolic spaces can be classified by an explicit invariant taking values in an infinite dimensional projective space. In terms of Borel complexity theory, this means that the natural equivalence relation associated with the classification problem is smooth. Furthermore, we show that all smooth Borel equivalence relations can be realized in this way up to Borel bi-reducibility. The simplest examples of isotropic weakly hyperbolic groups are $SL_2(F)$, where $F$ is a countable subfield of $\CC$.

\item {\bf Anisotropic groups.} General type actions of these groups on hyperbolic spaces are unclassifiable in a very strong sense. Namely, the equivalence relation associated with the classification problem is $K_\sigma$ complete; in particular, it cannot be induced by a Polish group action. This class includes all non-elementary hyperbolic (and, more generally, all acylindrically hyperbolic) groups as well as many other examples.
\end{itemize}

To justify our terminology, we mention that for any cobounded general type action of an isotropic group $G$ on a hyperbolic space $S$, all directions in $S$ are ``coarsely equivalent" modulo the $G$-action. In contrast, every anisotropic group $G$ admits a hyperbolic Cayley graph that can be $G$-equivariantly compressed in infinitely many ``independent directions" resulting in a plethora of general type $G$-actions on hyperbolic spaces; this property is precisely what accounts for the extreme complexity of Problem \ref{prob} in the anisotropic case.

In the course of proving our main theorem, we also obtain several results that appear to be of independent interest from the group theoretic perspective. These include rigidity of group actions on hyperbolic spaces (see Theorems \ref{Thm:LoxRigid} and \ref{Thm:LoxRigAct}) reminiscent of the work of Alperin–Bass and Culler–Morgan on groups acting on $\RR$-trees \cite{AB,CM}, a characterization of isotropic group actions on hyperbolic spaces in terms of their boundary dynamics (Theorem \ref{Thm:IsoEquiv}), and various conditions of algebraic and analytic nature that allow us to classify weakly hyperbolic groups into one of the two categories described above (see, for example, Propositions \ref{Prop:UP} and \ref{Prop:qm-intr}).

Finally, we would like to mention that our work is closely related to the paper \cite{ABO} and its successors \cite{ABR23, ABR24, Bal, BFGS}, where the authors made another attempt to formalize Problem \ref{prob} by introducing the \emph{poset of hyperbolic structures} on a group. Notably, for all the groups studied in these papers, the poset of general type hyperbolic structures was either remarkably simple or highly complex. We show that this is not a mere accident; in fact, this dichotomy holds for all countable groups.

In the next section, we review the necessary background and provide a more detailed discussion of our main results.

\paragraph{Acknowledgments.} The first author has been supported by the NSF grant DMS-2405032 and the Simons Fellowship in Mathematics MP-SFM-00005907.

%%%%%%%%%%%%%%%%%%%%%%%%%%%%%%%%%%%%%%%%%%%%%%%%%%%%%%%%%%%%

\section{Background and main results}

%%%%%%%%%%%%%%%%%%%%%%%%%%%%%%%%%%%%%%%%%%%%%%%%%%%%%%%%%%%%

\paragraph{2.1. A Borel approach to classification problems.} The study of orbit equivalence relations induced by group actions has a long history in the context of dynamical systems and operator algebras. In the 1990s, foundational work by Becker, Dougherty, Harrington, Hjorth, Jackson, Kechris, Louveau and others spurred the development of a more abstract theory of definable equivalence relations at the interface of topology and logic. Over the past three decades, this theory has seen rapid development and now serves as a standard tool for analyzing the complexity of various classification problems arising in algebra, geometry, topology, functional analysis, and dynamical systems. We begin with a brief review of the relevant background.

In the framework of Borel complexity theory, \emph{classification problems} are modelled by pairs $(X, E)$, where $X$ is a topological (usually Polish) space and $E$ is a definable (usually Borel or analytic) equivalence relation on $X$; to compare their difficulty, one introduces the following.

\begin{defn}
Let  $E$ and $F$ be two binary relations on topological spaces $X$ and $Y$, respectively. One says that the relation $E$ is \emph{Borel reducible} to $F$ and writes $E\preccurlyeq _B F$ if there exists a Borel map $f\colon X\to Y$ such that
$$
aEb \;\; \Longleftrightarrow f(a)Ff(b) \;\;\; \forall\, a,b\in X.
$$
Further, $E$ and $F$ are \emph{Borel bi-reducible} (denoted by $E\sim _B F$) if $E\preccurlyeq_B F$ and $F\preccurlyeq_B E$.
\end{defn}

The idea behind this definition is that ``explicitly definable" (in any reasonable sense) maps  between topological spaces are typically Borel, and vice versa. Therefore, when $E$ and $F$ are equivalence relations, the map $f$ provides an ``explicit" reduction of the classification problem $(X, E)$ to the (possibly more complicated) classification problem $(Y, F)$.

Recall that a topological space $X$ is \emph{Polish} if it is separable and completely metrizable. A \emph{Borel isomorphism} between Polish spaces $X$ and $Y$ is a bijective Borel map $f\colon X\to Y$ (in this case, $f^{-1}$ is also Borel \cite[Theorem 14.12]{Kec}). By Kuratowski's theorem \cite[Theorem 15.6]{Kec}, every Polish space is Borel isomorphic to either $\mathbb R$ or a discrete space.

A binary relation $R$ on a Polish space $X$ is \emph{Borel} (respectively, \emph{analytic}) if $R$ is a Borel (respectively, analytic) subset of $X\times X$. We are especially interested in Borel equivalence relations, abbreviated BERs. 

\begin{defn}\label{def:smooth}
A BER on a topological space is said to be \emph{smooth} (or \emph{concretely classifiable}) if it is Borel reducible to the equality relation on a Polish space.
\end{defn}

Kuratowski's theorem mentioned above and Silver's dichotomy \cite[Theorem 35.20]{Kec} imply that every smooth BER on a Polish space is Borel bi-reducible with the equality relation on a discrete space of cardinality $n\in \NN$, a countably infinite discrete space, or $\RR$ (denoted by $=_n$, $=_\NN$, and $=_\RR$, respectively); furthermore, these relations form an initial segment of the poset of BERs considered up to Borel bi-reducibility.

\begin{figure}
 \centering \hspace*{7mm} \small\def\svgscale{.85}{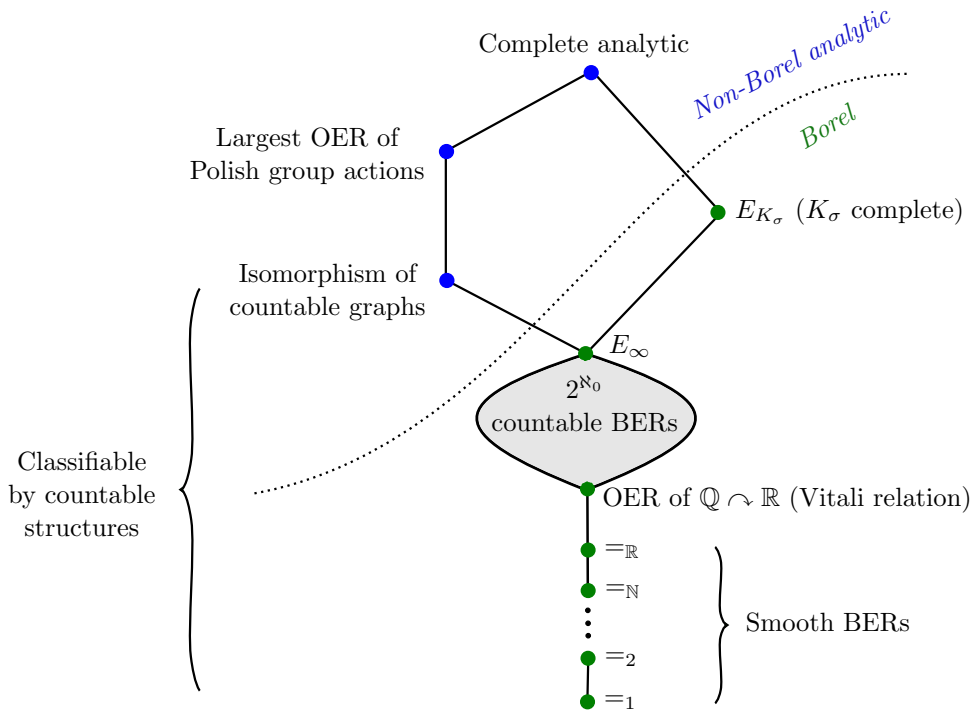}\\
  \caption{Complexity degrees of some classification problems}\label{Fig1}
\end{figure}

In contrast, the universe of non-smooth BERs has a complex and intricate structure. The simplest example of a non-smooth BER is the orbit equivalence relation (OER) of the natural action $\QQ \curvearrowright \RR$, called the \emph{Vitali relation}. It is also the smallest non-smooth BER with respect to $\preccurlyeq_B$ (see \cite{HKL}). Above it lies a very rich set of BERs with countable equivalence classes known as \emph{countable BERs}. Adams and Kechris \cite{AK} showed that the poset of Borel bi-reducibility classes of countable BERs has $2^{\aleph_0}$ elements and contains the largest element (usually denoted by $E_\infty$)  represented by the OER of the natural action of the free group $F_2$ of rank $2$   on $2^{F_2}$.

By the Feldman--Moore theorem \cite{FM}, every countable BER on a Polish space $X$ arises as an OER of an action of a countable group by Borel automorphisms of $X$. Passing to more general \emph{Polish group actions} (i.e., Borel actions of Polish groups on Polish spaces) yields a rich class of analytic OERs, which are not necessarily Borel. 

From the classification point of view, of special interest are the isomorphism relations on the spaces of countable structures of a fixed signature, which occur as the OERs of continuous actions of the Polish group $Sym(\NN)$, called the \emph{logic action} (for details, see \cite[Section 2.5]{BK}).

\begin{defn}
An equivalence relation $E$ on a topological space is \emph{classifiable by countable structures} if $E$ is Borel reducible to the isomorphism relation on the spaces of countable structures of a certain countable signature.
\end{defn}

All countable BERs and, more generally, all OERs arising from Borel actions of closed subgroups of $Sym(\NN)$ on Polish spaces belong to this class \cite[Theorem 12.3.3]{Kan}. In fact, all OERs of Borel actions of $Sym(\NN)$ on Polish spaces are Borel reducible to the isomorphism relation on the space of countable graphs \cite[Theorem 12.4.1]{Kan}. Other notable classes admitting maximal (also called \emph{complete}) elements with respect to Borel reducibility are the OERs induced by Polish group action \cite[Corollary 2.6.8]{BK} and the class of all analytic relations on Polish spaces (see \cite{FLR} and references therein).

Finally, we recall the definition of two relations introduced by Rosendal \cite{Ros}, which will play an important role in our paper.

\begin{defn}\label{Def:EKs}
Let $\Pi$ denote the set of sequences $\prod_{n\in \NN}\{ 1, \ldots, n\}$ endowed with the product topology. By $Q_{K_\sigma}$ we denote the relation on $\Pi$ defined by the rule
$$
rQ_{K_\sigma} s \;\; \Longleftrightarrow\;\; \sup\limits_{n\in \NN} (r(n)-s(n))<\infty.
$$
Clearly $Q_{K_\sigma}$ is a quasi-order. By $E_{K_\sigma}$ we denote the induced equivalence relation on $\Pi$; thus,  $$r E_{K_\sigma} s \;\; \Longleftrightarrow\;\; \sup_{n\in \NN} |r(n)-s(n)|<\infty.$$
\end{defn}

Recall that a subset of a topological space is said to be $K_\sigma$ if it is a countable union of compact subsets. It is straightforward to verify that both $Q_{K_\sigma}$ and $E_{K_\sigma}$ are $K_\sigma $ subsets of $\Pi\times \Pi$. We will need the following theorem. The first claim is due to Rosendal \cite{Ros}; the second claim follows from the first one and \cite[Theorem 4.2]{KL}.

\begin{thm}[Rosendal, Kechris--Louveau]\label{Thm:RKL}

\begin{enumerate}
\item[(a)] Every $K_\sigma$ quasi-order (respectively, equivalence relation) on a Polish space is Borel reducible to $Q_{K_\sigma}$ (respectively, $E_{K_\sigma}$).

\item[(b)]  $E_{K_\sigma}$ cannot be Borel reduced to an OER of a Polish group action; in particular, $E_{K_\sigma}$ is not classifiable by countable structures.
\end{enumerate}
\end{thm}

Our discussion of Borel complexity degrees is summarized on Figure \ref{Fig1}.

\paragraph{2.2. Classifying geometric structures on groups.}
There are two standard ways to introduce a geometric structure on a group. First, we can choose a generating set and consider the associated word metric. Second, we can fix an isometric group action on a metric space and consider the pseudometric on the group induced by the orbit map. These two approaches are linked by the well-known Svarc–Milnor lemma, which shows that classifying cobounded actions of a given group $G$ on geodesic metric spaces up to coarsely $G$-equivariant quasi-isometry is essentially the same as classifying metrics on $G$ arising from generating sets up to bi-Lipschitz equivalence. In this section, we focus on the latter classification problem, as it is somewhat easier to formalize. The more general case of arbitrary group actions on metric spaces is discussed in Section \ref{Sec:CGA}.

Let $G$ be a group and let $$Gen(G)=\{ X\subseteq G \mid \la X\ra = G\}$$ denote the set of all generating sets of $G$. We think of $Gen(G)$ as a subspace of the product space $2^G$ and endow it with the induced topology. 
For a generating set $X\in Gen(G)$, we denote by $ |\cdot|_X$ and $\d_X$ the corresponding word length and word metric on $G$, respectively. Further, by $Cay(G,X)$ we denote the Cayley graph of $G$ equipped with the standard metric, also denoted by $\d_X$, obtained by identifying every edge with $[0,1]$. 

The following definition was proposed in \cite{ABO}. 

\begin{defn}\label{Def:ABO}
Let $X, Y\in Gen(G)$. We say that $X$ is \emph{dominated} by $Y$ and write $X\preccurlyeq Y$ if the identity map $(G, \d_Y) \to (G, \d_X)$ is Lipschitz (equivalently, $\sup_{y\in Y} |y|_X < \infty$.)
Further, we say that $X$ and $Y$ are \emph{equivalent} and write  $X\sim Y$ if $X\preccurlyeq Y$ and $Y\preccurlyeq X$. \end{defn}

Clearly, $\preccurlyeq $ is a quasi-order and $\sim$ is an equivalence relation on $Gen(G)$. Informally, we order elements $X\in Gen(G)$ according to the ``size" of the metric spaces $(G, \d_X)$. 

\begin{ex} \label{Ex:Gen}  For any group $G$, the generating set $X=G$ is dominated by any other generating set. If $G$ is finitely generated, all finite generating sets of $G$ are equivalent and dominate all other generating sets. 
\end{ex}

Recall that any isometric action of a group $G$ on a hyperbolic space $S$ induces an action of $G$ on the Gromov boundary $\partial S$ by homeomorphisms. We say that a hyperbolic Cayley graph $Cay(G,X)$ is of \emph{general type} if $|\partial Cay(G,X)|=\infty$ and $G$ does not fix any point of $\partial Cay(G,X)$. Along with the space $Gen(G)$, we consider its subspaces
$$
Hyp(G)=\{ X\in Gen(G) \mid Cay(G,X) {\rm \; is \; hyperbolic}\},
$$
and 
$$
Hyp_{gt}(G)=\{ X\in Hyp(G) \mid Cay(G,X) {\rm \; is\; of\; general\; type}\}.
$$

In the study of Borel equivalence relations up to Borel reducibility, it is customary to forget the topology of the ambient space and work in the category of Borel spaces. Recall that a \emph{Borel space} is a pair $(B,\Sigma)$, where $B$ is a set and $\Sigma$ is a $\sigma$-algebra of subsets of $B$. A Borel space $(B,\Sigma)$ is \emph{standard} if it admits a \emph{Polishing}; that is, there exists a Polish topology on $B$ such that $\Sigma$ is precisely the set of all Borel subsets of $B$.

Clearly, the Borel bi-reducibility class of a given equivalence relation on a standard Borel space $B$ is independent of the choice of a particular Polishing of $B$. Recall also that for every Borel subset $B$ of a Polish space $P$, the collection $\Sigma =\{ A\subseteq B\mid A {\rm \; is \; Borel\; in \;} P\}$ defines the structure of a standard Borel space on $B$ \cite[Corollary 13.4]{Kec}. Therefore, the analysis of Borel complexity discussed above can be expanded to equivalence relations on Borel subsets of Polish spaces. The result below provides a formalization of Problem \ref{prob} in these settings.

\begin{prop} \label{Prop:Hyp}
Let $G$ be a countable group.
\begin{enumerate}
    \item[(a)] The space $Gen(G)$ is Polish.
    \item[(b)] The relations $\preccurlyeq $ and $\sim$ on $Gen(G)$ are Borel.
    \item[(c)] $Hyp(G)$ and $Hyp_{gt}(G)$ are Borel subsets of $Gen(G)$. 
\end{enumerate}
\end{prop}

We are now ready to state our main theorem. For a more general result encompassing all general type group actions on hyperbolic spaces, see Theorem \ref{Thm:main-full}.

\begin{thm}\label{Thm:main}
For any countable weakly hyperbolic group $G$, the restriction of $\sim$ to $Hyp_{gt}(G)$ is Borel bi-reducible to one of the relations $=_1$, $=_2$, $\ldots$, $=_\NN$, $=_\RR$, or $E_{K_\sigma}$, and all possibilities can be realized. In particular, the restriction of $\sim$ to $Hyp_{gt}(G)$ is either smooth or not classifiable by countable structures.
\end{thm}

The proof of this result relies on a dichotomy for weakly hyperbolic groups, which appears to be of independent interest; we discuss it in more detail below.

\paragraph{2.3. Isotropic vs. anisotropic weakly hyperbolic groups.}
We begin with a general definition, which makes sense for any action of a group $G$ on a metric space $S$. Informally, it captures the idea that all directions in $S$, represented by equidistant pairs of points,  are ``coarsely the same" modulo the $G$-action (see Fig.~\ref{Fig2}).

\begin{defn}\label{Def:Iso}
An action of a group $G$ on a metric space $(S, \d_S)$ is \emph{isotropic} if there exists $D\ge 0$ (called the \textit{isotropy constant} of the action) such that, for any points $x,y,x^\prime, y^\prime \in S$ satisfying the equality $\d_S(x,y)=\d_S(x^\prime, y^\prime)$, there is $g\in G$ such that
$$
\max\{\d_S(gx,x^\prime),\, \d_S(gy,y^\prime)\}\le D.
$$
\end{defn}

\begin{figure}
 \centering 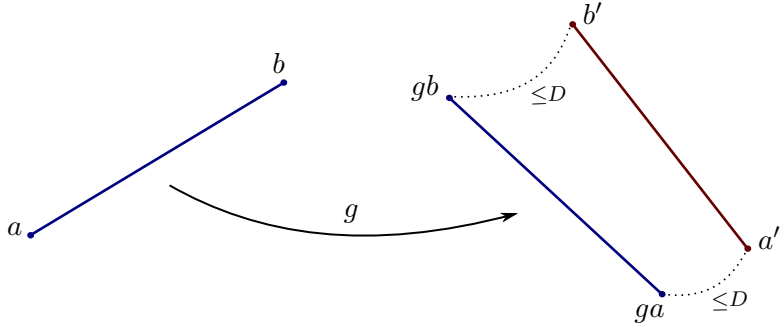\\
  \caption{Isotropic group action}\label{Fig2}
\end{figure}
\begin{ex}
The natural action of the infinite dihedral group on $\RR$ is isotropic.
\end{ex}

Recall that an action of a group $G$ on a topological space $T$ is \emph{minimal} if the $G$-orbit of every point is dense in $T$.  Isotropic group actions on hyperbolic spaces can be alternatively characterized as follows.

\begin{thm}\label{Thm:IsoEquiv}
For any general type action of a group $G$ on a hyperbolic space $S$, the following conditions are equivalent.
\begin{enumerate}
\item[(a)] The action is isotropic.
\item[(b)] The action is cobounded and the induced action of $G$ on $\{ (x,y)\in \partial S \times \partial S \mid x\ne y\}$ is minimal.
\item[(c)] There exist a minimal element $X$ of $Hyp_{gt}(G)$ and a $G$-equivariant quasi-isometry $(G, \d_X)\to S$. 
\end{enumerate}
\end{thm}

We are now ready to define isotropic weakly hyperbolic groups.

\begin{defn}
 A weakly hyperbolic group $G$ is called \emph{isotropic} if for every generating set $X\in Hyp_{gt}(G)$,
the natural action $G\act Cay(G,X)$ is isotropic. 
\end{defn}

A crucial observation in the proof of Theorem \ref{Thm:main} is that general type actions of isotropic groups on hyperbolic spaces can be classified by their translation length functions. This is close in spirit to the Alperin--Bass--Culler--Morgan theorem proved independently in \cite{AB} and \cite{CM}, which asserts that general type actions of a given group on $\mathbb R$-trees are determined by the corresponding translation length functions up to equivariant isometry. We discuss the particular case of actions on Cayley graphs here and refer to Theorem \ref{Thm:LoxRigAct} for the more general result. 

Given a group $G$, a generating set $X\in Gen(G)$, and an element $g\in G$, we define the \textit{translation length} of $g$ with respect to $X$ to be
$$
\tau_X(g)=\liminf_{n\to \infty} \frac{|g^n|_X}n.
$$
For a function $\tau \colon G\to \RR$, we let $[\tau]=\big\{c\tau \mid c\in \RR\setminus \{ 0\} \big\}.$ If $\tau \colon G\to \RR$ is non-zero, we can think of $[\tau]$ as an element of the projective space $\mathbf{P}(\RR^G)$. 

\begin{thm}\label{Thm:LoxRigid}
Let $G$ be an isotropic weakly hyperbolic group. For any $X,Y\in Hyp_{gt}(G)$, we have $X\sim Y$ if and only if $[\tau_X]= [\tau_Y]$. 
\end{thm}

If $G$ is countable, $\RR^G$ is a Polish space. Further, it is easy to show that the proportionality relation on $\RR^G$ is smooth, and the map $Gen(G)\to \RR^G$ given by $X\mapsto \tau_X$ is Borel. Thus, we obtain the following. 

\begin{cor}\label{Cor:Iso}
For any countable isotropic weakly hyperbolic group, the restriction of $\sim $ to $Hyp_{gt}(G)$ is smooth.
\end{cor}

Following tradition, we use the term ``anisotropic" in lieu of ``not isotropic". Thus, a weakly hyperbolic group is \emph{anisotropic} if it admits at least one anisotropic general type hyperbolic Cayley graph. For groups in this class, we prove the following.

\begin{thm}\label{Thm:Aniso}
For any countable anisotropic weakly hyperbolic group $G$, there exists a Borel map $f\colon \Pi \to Hyp_{gt}(G)$ such that $$f(r)\preccurlyeq f(s) \;\;\; \Longleftrightarrow\;\;\; r \,Q_{K_\sigma}\, s \;\;\forall\,  r,s\in \Pi.$$ In particular, $E_{K_\sigma}$ is Borel reducible to the restriction of $\sim$ to $Hyp_{gt}(G)$.
\end{thm}

On the other hand, it is not difficult to show that the equivalence relation $\sim$ on $Gen(G)$ is Borel reducible to $E_{K_\sigma}$ whenever $G$ is countable. Therefore, we obtain the following. 

\begin{cor} \label{Cor:Aniso}
For any countable anisotropic weakly hyperbolic group $G$, the restriction of $\sim$ to $Hyp_{gt}(G)$ is Borel bi-reducible with $E_{K_\sigma}$.
\end{cor}

\paragraph{2.4. An application to hyperbolic structures.}
The set of equivalence classes 
$$
\mathcal H_{gt}(G)=Hyp_{gt}(G)/{\sim} 
$$
endowed with the induced ordering is called the \emph{poset of general type hyperbolic structures}. In \cite{ABO} the authors showed that for every acylindrically hyperbolic group $G$, $\mathcal H_{gt}(G)$ is very rich and difficult to describe; on the other hand, there are groups for which $\mathcal H_{gt}(G)$ has a very simple structure. As a corollary of our result, we show that this dichotomy holds for any countable group $G$. 

\begin{cor}\label{Cor:HS}
For any countable weakly hyperbolic group $G$, exactly one of the following conditions holds.
\begin{enumerate}
\item[(a)]  $\mathcal H_{gt}(G)$ is an antichain of cardinality $1, 2,\ldots, \aleph_0$, or $2^{\aleph_0}$. %(and all possibilities can be realized).
\item[(b)] $\mathcal H_{gt}(G)$ contains chains and antichains of cardinality $2^{\aleph_0}$, and every poset of cardinality at most $\aleph_1$ can be embedded in $\mathcal H_{gt}(G)$. %In particular, $\mathcal H_{gt}(G)$ is a universal poset of cardinality $2^{\aleph_0}$ assuming the continuum hypothesis.  
\end{enumerate}
\end{cor}

\paragraph{2.5. Examples.} 
Corollaries \ref{Cor:Iso} and \ref{Cor:Aniso} prove Theorem \ref{Thm:main}, except for the assertion that all possibilities are realized. We now discuss some results confirming the latter claim.

\begin{ex}[{\cite[Theorem 4.28]{ABO}}]\label{Ex:I}
There exists a finitely generated group $I$ such that $|\mathcal H_{gt}(I)|=1$. Further, for every $n\in \NN\cup\{\aleph_0\}$,  we have $|\mathcal H_{gt}(I^n)|=n$, where $I^n$ is the direct sum  of $n$ copies of $I$ (see Proposition \ref{Prop:DS}). In particular, the group $I^n$ is isotropic, and the restriction of $\sim $ to $Hyp_{gt}(I^n)$ is Borel bi-reducible with $=_n$ (if $n\in \NN$) or $=_\NN$ (if $n=\aleph_0$). 
\end{ex}

In \cite{ABO}, the group $I$ was constructed using the theory of almost prescribed local actions on trees developed in \cite{LB}. Examples of different nature were found in \cite{BCFS}. The following theorem is an immediate consequence of their main result. 

\begin{thm}[Bader--Caprace--Furman--Sisto, {\cite{BCFS}}]
For any $n\in \NN$, every irreducible lattice in a product of simple, centerless, Lie groups with exactly $n$ factors of rank $1$ has exactly $n$ general type hyperbolic structures. 
\end{thm}

At the other end of the spectrum lie acylindrically hyperbolic groups, which are easily seen to be anisotropic. Recall that the class of acylindrically hyperbolic groups includes all non-elementary hyperbolic and relatively hyperbolic groups, mapping class groups of closed surfaces of genus at least $1$, $Aut(F_n)$ and $Out(F_n)$ (where $F_n$ is the free group of rank $n\ge 2$), and many other examples (see \cite{Osi16,Osi18}). 

In Section \hyperref[Sec:Ex]{5.4}, we provide further examples of isotropic and anisotropic weakly hyperbolic groups. Our findings are summarized in the table below. 

\vspace{1mm}
{\begin{center}
\begin{tabular}{?c?c|l|}
\Xhline{2\arrayrulewidth}
&&\\
\textbf{Type} & \textbf{Complexity}  &  \hspace{30mm}\textbf{Examples}  \\  
&& \\
\Xhline{2\arrayrulewidth}
\multirow{3}{*}
               &&\\
               & &  $\bullet$ $I^n$, where $I$ is the group from Example \ref{Ex:I} \\
               &$=_n$, $n\in \NN$ &\\
               & & $\bullet$  Irreducible lattices in products of simple, centerless\\ 
               && \hspace{2mm} Lie groups with exactly $n$ factors of rank $1$ \\&&\\
               \cline{2-3}
               {\textbf{Isotropic}} &&\\
               &$=_\mathbb N$&  $G^{\aleph_0}$ for any group $G$ as in the cell above\\
               &&\\ \cline{2-3}
               &&\\
               & $=_\mathbb R$ & $SL_2$ over any countable transcendental extension of $\QQ$  \\
               &&\\ \hline
&&\\
&& $\bullet$ Acylindrically hyperbolic (e.g., non-elementary \\ 
&& \hspace{2mm} hyperbolic and relatively hyperbolic) groups\\  &&\\
\textbf{Anisotropic} & $E_{K_\sigma}$ & $\bullet$ HNN-extensions $A\ast_{B^t=C}$, where $B\ne A\ne C$; e.g.,\\
&& \hspace{2mm}  $BS(m,n)=\langle a,b\mid b^{-1}a^mb=a^n\rangle$, $|m|, |n|\ne 1$\\  &&\\
& & $\bullet$ Amalgams $A\ast_C B$, where $A\ne C$ and $|C\backslash B/C|\ge 3$\\
&& \\ \hline
\end{tabular}
\end{center}}
{\captionof{table}{Complexity of the classification problem $(Hyp_{gt}(G), \sim)$ for countable groups.}\label{table}}
\bigskip

We mention a couple of results that are instrumental in determining whether a given weakly hyperbolic group is isotropic. Recall that a group $G$ is \emph{uniformly perfect} if there exists $k\in \NN$ such that every element of $G$ is a product of at most $k$ commutators. Further, a group $G$ is \textit{boundedly generated}  if there exist $k\in \NN$ and elements $a_1, \ldots a_k\in G$ such that every $g\in G$ can be written as 
$g=a_1^{\alpha _1}\cdots a_k^{\alpha_k}$ for some $\alpha_1, \ldots, \alpha_k \in \ZZ$.

\begin{prop}\label{Prop:UP}
If a weakly hyperbolic group is uniformly perfect or boundedly generated, then it is isotropic. 
\end{prop}

Note that the group $SL_2(R)$ is weakly hyperbolic for any subring $R\subseteq \CC$, as evidenced by the standard action on $\mathbb H^3$. Using Proposition \ref{Prop:UP} and a generalization of the bounded generation result of Morgan, Rapinchuk, and Sury \cite{MRS} for rings of arithmetic integers (see Appendix), we will prove the following.

\begin{cor}\label{Cor:SL2}
\begin{enumerate}
    \item[(a)] The group $SL_2(F)$ is isotropic for any subfield $F\subseteq \CC$.
    \item[(b)] Let $R$ be a finitely generated subring of the algebraic closure of $\QQ$. The group $SL_2(R)$ is anisotropic if and only if $R$ is contained in the ring of integers of $\QQ(\sqrt{d})$ for some negative $d\in \ZZ$. 
\end{enumerate}
\end{cor}

Recall also that a \textit{quasi-morphism} on a group $G$ is a map $q\colon G\to \RR$ such that 
\begin{equation}\label{Eq:qm}
\sup_{g,h\in G}|q(gh)-q(g)-q(h)|<\infty.
\end{equation}

\begin{defn}\label{Def:subord}
A quasi-morphism $q\colon G\to \RR$ is \textit{bounded} if $\sup_{g\in G} |q(g)|<\infty$. Further, we say that $q$ is \emph{subordinate} to a $G$-action on a metric space $S$ if there exist $s\in S$ and $M\ge 0$ such that
\begin{equation}\label{Eq:subord}
|q(g)| \le M\d_S(s, gs)+M\;\;\; \forall\, g\in G.
\end{equation}
\end{defn}

For example, if $G$ is an HNN-extension of the form $A\ast_{B^t=C}$, then the natural homomorphism $G\to \langle t\rangle$ is subordinate to the action of $G$ on the associated Bass-Serre tree. The last two examples in the anisotropic section of Table \ref{table} are obtained using the following.

\begin{prop}\label{Prop:qm-intr}
Suppose that a group $G$ admits a general type action on a hyperbolic space $S$ and there exists an unbounded quasi-morphism $G\to \RR$ subordinate to $G\curvearrowright S$. Then $G$ is anisotropic.
\end{prop}

\paragraph{2.6. Structure of the paper.} In the next section, we discuss the equivalence of group actions introduced in \cite{ABO} and provide a formalization of the associated classification problem using the space of pseudo-length functions on a group.  We also prove parts (a) and (b) of   Proposition \ref{Prop:Hyp} (see Proposition \ref{Prop:PLGen}).

In Sections \ref{Sec:GAHS} and \ref{Sec:ModGA}, we obtain some technical results about loxodromic isometries and use them to construct new general type group actions on hyperbolic spaces starting from a single anisotropic action. This construction plays a major role in the proof of Theorem~\ref{Thm:Aniso}.

Section \ref{Sec:IWHG} primarily focuses on the in-depth study of isotropic weakly hyperbolic groups. In particular, we prove generalizations of Theorems \ref{Thm:IsoEquiv} and \ref{Thm:LoxRigid} there (see Theorems \ref{Thm:Iso} and \ref{Thm:LoxRigAct}, respectively). We also prove Propositions \ref{Prop:UP} and \ref{Prop:qm-intr} there and obtain further results that enable us to provide examples of isotropic and anisotropic weakly hyperbolic groups, including those listed in the table above.

The proofs of our main results are contained in Section \ref{Sec:DV}. Specifically, part (c) of Proposition \ref{Prop:Hyp} is proved in Subsection~\hyperref[Sec: BSPL]{7.1} and the proofs of generalized versions of Theorem~\ref{Thm:main}, Theorem \ref{Thm:Aniso}, and Corollary~\ref{Cor:HS} are given in Subsections~\hyperref[Sec:PMT]{7.2} and \hyperref[Sec:AtoHS]{7.3}.

\paragraph{Acknowledgments.} The authors are grateful to Andrei Rapinchuk for useful discussions on generalizations of the results in \cite{MRS} to finitely generated rings of algebraic numbers, and to Talia Fernos for helpful comments. During the preparation of this paper, the first author was supported by the NSF grant DMS-2405032 and the Simons Fellowship in Mathematics MP-SFM-00005907.
%%%%%%%%%%%%%%%%%%%%%%%%%%%%%%%%%%%%%%%%%%%%%%%%%%%%%%%%%%%%%%%%%

\section{Classifying group actions on metric spaces}\label{Sec:CGA}

%%%%%%%%%%%%%%%%%%%%%%%%%%%%%%%%%%%%%%%%%%%%%%%%%%%%%%%%%%%%%%%%%

\paragraph{3.1. Comparing group actions.}\label{Sec:CompGA} 
Recall that a \emph{quasi-order} on a set is any reflexive and transitive relation. Every quasi-order $\preceq _{\scriptscriptstyle P}$ on a set $P$ induces an equivalence relation $\simeq _{\scriptscriptstyle P}$ by the rule 
$$
x \simeq_{\scriptscriptstyle P} y \;\;\; \Longleftrightarrow \;\;\; x\preceq_{\scriptscriptstyle P} y \; {\rm and} \; y\preceq_{\scriptscriptstyle P} x 
$$
for all $x, y\in P$.  We use the standard notation $[x]$ for the equivalence class of an element $x\in P$ and denote by $P/{\simeq_{\scriptscriptstyle P}}$ the set of all equivalence classes. The quasi-order $\preceq_{\scriptscriptstyle P}$ induces a genuine order on $P/{\simeq_{\scriptscriptstyle P}}$ in the obvious way; we keep the same notation $\preceq_{\scriptscriptstyle P}$ for this order. Thus, we have
$$
[x]\preceq_{\scriptscriptstyle P} [y] \;\;\; \Longleftrightarrow\;\;\; x\preceq_{\scriptscriptstyle P}y
$$
for all $x,y\in P$. When talking about the poset $P/{\simeq_{\scriptscriptstyle P}}$, we always assume it to be equipped with the induced ordering. 

As usual, we say that an element $x\in P$ is \textit{minimal} in $P$ if so is $[x]$ in $P/{\simeq_{\scriptscriptstyle P}}$ (equivalently, for any $y\in P$, the inequality $y\preceq_{\scriptscriptstyle P} x$ implies $y\simeq_{\scriptscriptstyle P} x$ ).

Throughout this paper, all group actions on metric spaces are assumed to be isometric by default. More formally, an \emph{action} of a group $G$ on a metric space $S$ is a homomorphism $$\alpha \colon G\to Isom(S).$$ We denote this action by $G\curvearrowright S$ omitting $\alpha$ from the notation as it will always be clear from the context or unnecessary to specify. The distance function on a metric space $S$ is denoted by $\d_S$.

Given a group $G$, let $A(G)$ be the set of all $G$-actions on metric spaces; to avoid dealing with proper classes, we assume all metric spaces under consideration to have cardinality at most $2^{\aleph_0}$. (The particular bound on the size of metric spaces is not important for our purposes). As observed in \cite{ABO}, elements of $A(G)$ can be naturally quasi-ordered according to the amount of information they provide about $G$.

\begin{defn}\label{def-poset}
Let $G\curvearrowright R$ and $G\curvearrowright S$ be two actions of a group $G$ on metric spaces. We say that $G\curvearrowright S$ \emph{dominates} $G\curvearrowright R$ and write $G\curvearrowright R \lA   G\curvearrowright S$ if for some $r\in R$ and $s\in S$, there is a constant $C$ such that
\begin{equation}\label{dSdR}
\d_R(r,gr)\le C\d_S(s, gs)+C\;\; \forall\, g\in G.
\end{equation}
Further, we write $G\curvearrowright S \eA G\curvearrowright R$ and say that the actions $G\curvearrowright R$ and $G\curvearrowright S$ are \emph{equivalent} if $G\curvearrowright S \lA G\curvearrowright R$ and $G\curvearrowright R \lA G\curvearrowright S$.
\end{defn}

The intuition behind this definition is that actions with larger orbits usually provide more information about the group. For example, it is not difficult to show the following (compare to Example \ref{Ex:Gen}).

\begin{ex}
If $G$ is generated by a finite set $X$, the action of $G$ on $(G, \d_X)$ by left multiplication dominates any other action of $G$ (see Lemma \ref{Lem:OLip}). At the other end of the spectrum, we have $G$-actions with bounded orbits, which are all equivalent and dominated by any other action of $G$.
\end{ex}

It is easy to see that the existential quantification over $r$ and $s$ in Definition \ref{def-poset} can be replaced by the universal one. 

\begin{lem}[{\cite[Lemma 3.3]{ABO}}]\label{Lem:E->A}
Let $G\curvearrowright R$ and $G\curvearrowright S$ be two actions of a group $G$ on metric spaces. Suppose that $G\curvearrowright R \lA G\curvearrowright S$; then for any  $r \in R$ and any $s\in S$, there exists a constant $C$ such that (\ref{dSdR}) holds. 
\end{lem}

As an immediate corollary, we obtain that $\lA$ is a transitive relation and, therefore, is a quasi-order on $A(G)$ \cite[Corollary 3.4]{ABO};  therefore, $\eA$ is an equivalence relation. 

\paragraph{3.2. The space of pseudo-length functions.}
To formalize the classification problem associated with the equivalence relation $\eA$ 
on $A(G)$ in the framework of Borel complexity theory, we have to organize group actions into a topological (or Borel) space. There are two natural ways to achieve this goal.

One possibility would be to restrict to $G$-actions on separable metric spaces. For countable groups, such actions can be encoded by countable structures (whose universes are countable dense subsets of the metric spaces) in a countable signature, and the standard topology on the space of countable structures can be used in this case (see \cite{Mar} for details). 

Although this approach appears conventional, it has a significant drawback: the structures encoding group actions bear excessive information at the ``small-scale", which is irrelevant in the context of geometric group theory. Instead, we observe that there is a natural equivalence-preserving map from $A(G)$ to the Polish space of pseudo-length functions $PL(G)$, defined below, which keeps all the information about the $G$-actions we care about. This allows us to replace $A(G)$ with $PL(G)$, which is much easier to deal with. 

\begin{defn}\label{Def:PLF}
A \emph{pseudo-length function} on a group $G$ is a map $\ell\colon G\to [0, \infty)$ satisfying the following conditions for all $g,h\in G$.
\begin{enumerate}
\item[(a)] $\ell(1)=0$;
\item[(b)] $\ell(g)=\ell(g^{-1})$;
\item[(c)] $\ell (gh)\le \ell(g)+\ell(h)$.
\end{enumerate}
We denote by $PL(G)$ the space of all pseudo-length functions on $G$ endowed with the topology of pointwise convergence. 
\end{defn}

\begin{ex}\label{Ex:WL}
For every generating set $X$ of a group $G$, the \emph{word length} function $\ell_X$ defined by $\ell_X(g)=|g|_X$  for all $g\in G$ is a pseudo-length on $G$. We denote by $WL(G)$ the subset of $PL(G)$ consisting of all word length functions on $G$; that is,
$$
WL(G)=\{ \ell_X\mid X\in Gen(G)\}.
$$
\end{ex}

The next example describes another class of pseudo-length functions that play an important role in our paper. For any action $G\curvearrowright S$ of a group $G$ on a metric space $S$ and any point $s\in S$, we define 
\begin{equation}\label{Eq:ls}
\ell_{G\curvearrowright S, s} (g) =\d_S(s, gs).
\end{equation}
It is straightforward to see that $\ell_{G\curvearrowright S, s}\in PL(G)$ for any choice of $s\in S$. 

\begin{defn}
Given $\ell_1, \ell_2\in PL(G)$, we write $\ell_1\lPL \ell_2$ if there exists $C>0$ such that $$\ell_1(g)\le C\ell_2(g)+C\;\;\; \forall\, g\in G.$$ Clearly, $\lPL $ is a quasi-order on $PL(G)$. By $\ePL$ we denote the induced equivalence relation. 
\end{defn}

\begin{lem}\label{Lem:ODqoPres}
Let $G\curvearrowright R$ and $G\curvearrowright S$ be two actions of a group $G$ on metric spaces. For any $r\in R$ and $s\in S$, we have $G\curvearrowright R\lA G\curvearrowright S$ if and only if $\ell_{G\curvearrowright R, r}\lPL\ell_{G\curvearrowright S, s}$. In particular, for any $s_1, s_2\in S$, we have $\ell_{G\curvearrowright S, s_1}\ePL\ell_{G\curvearrowright S, s_2}$.
\end{lem}

\begin{proof}
Assume first that $G \curvearrowright R\lA G \curvearrowright S$. By Lemma \ref{Lem:E->A}, there exists a constant $C$ such that (\ref{dSdR}) is satisfied. For any $g\in G$, we have
$$
\ell_{G\curvearrowright R , r} (g) = \d_R(r, gr) \le C\d_S(s, gs) + C=C\ell_{G\curvearrowright S, s} (g)+C.
$$ 
Thus, $\ell_{G\curvearrowright R, r} \lPL \ell_{G\curvearrowright S, s}$. This proves the forward implication in the first claim. The backward implication is a reformulation of Definition \ref{def-poset}.
\end{proof}

Conceptually, Lemma \ref{Lem:ODqoPres} justifies replacing of the classification problem $(A(G), \eA)$ with $(PL(G), \ePL)$. We record a few more elementary observations concerning the quasi-ordered sets $(Gen(G), \preccurlyeq)$ and $(PL(G),\lPL)$. Recall that $Gen(G)$ is endowed with the subspace topology induced by the inclusion $Gen(G)\subseteq 2^G$.

\begin{prop}\label{Prop:PLGen}
Let $G$ be a countable group.
\begin{enumerate}
\item[(a)] $PL (G)$ is a Polish space and $\lPL$ (respectively, $\ePL$) is a Borel quasi-order (respectively, equivalence relation) on $PL(G)$.
\item[(b)] $Gen (G)$ is a Polish space and $\preccurlyeq$ (respectively, $\sim$) is a Borel quasi-order (respectively, equivalence relation) on $Gen(G)$.
\item[(c)] The map $(Gen(G), \preccurlyeq)\to (PL(G),\lPL)$ sending every $X\in Gen(G)$ to $\ell_X\in PL(G)$ is Borel and quasi-order preserving; in particular, it is a Borel reduction of $\preccurlyeq$ to $\lPL$ and $\sim$ to $\ePL$.
\end{enumerate}
\end{prop}

\begin{proof}
Recall that every $G_\delta $ (in particular, every closed) subset of a Polish space is Polish with respect to the subspace topology \cite[Theorem 3.11]{Kec}. To prove (a), we first note that $PL (G)$ is a closed subspace of the Polish space $[0,\infty)^G$ (endowed with the product topology). Therefore, $PL (G)$ is Polish. Further, for any $n \in \NN$, the set $A_n$ defined by 
$$A_n = \{(\ell_1,\ell_2) \in PL(G)\times PL(G) \mid \ell_1(g)\le n\ell_2(g)+n\;\; \forall\,g \in G\}$$
is closed in $PL(G)\times PL(G)$. Since $\lPL$, considered as a subset of $PL(G)\times PL(G)$, equals $\bigcup_{n \in \NN}A_n$, the quasi-order $\lPL$ is Borel in $PL(G)\times PL(G)$; consequently, so is $\ePL$.

We now proceed with (b) and (c). It is straightforward to verify that, for any $g \in G$, the set $B_g = \big\{X \in 2^G \mid g \in \la X \ra \big\}$ is open in $2^G$. Note also that $Gen(G) = \bigcap_{g \in G} B_g$. Hence, $Gen(G)$ is a $G_\delta$ subset of $2^G$; in particular, it is Polish. Further, it is not difficult to show that, for any $g \in G$ and any $n \in \ZZ$, the set $$L_{g,n}=\{X \in Gen(G) \mid |g|_X \le n\}$$ is open in $Gen(G)$; therefore, the set $$E_{g,n}=\{ X \in Gen(G) \mid |g|_X = n\}= L_{g,n} \setminus L_{g, n-1}$$ is Borel. For any finite set of elements $g_1, \ldots, g_k$ and any integers $n_1, \ldots, n_k$, the preimage of the basic neighborhood $$\{ \ell \in PL(G)\mid \ell (g_i)=n_i, \; i=1, \ldots, k\}\subseteq PL(G)$$ under the map $X \mapsto \ell_X$ equals $\bigcap_{i=1}^k E_{g_i, n_i}$. This implies that the map $Gen(G) \to  PL(G)$ is Borel. Finally, for any $X,Y \in Gen(G)$, we obviously have $X \preccurlyeq Y \iff \ell_X \lPL \ell_Y$  and $X \sim Y \iff \ell_X \ePL \ell_Y$. Hence, $\preccurlyeq$ and $\sim$ are Borel in $Gen(G)\times Gen(G)$ being full preimages of the Borel subsets $\lPL$ and $\ePL$ of $PL(G)\times PL(G)$ under a Borel map.
\end{proof}

\paragraph{3.3. Quasi-isometry and the Svarc--Milnor lemma.}
Recall that the space $S$ is \emph{geodesic} if, for any $x, y\in S$, there is a geodesic path connecting $x$ to $y$ in $S$. 
A map $f\colon R\to S$ between two metric spaces $R$ and $S$ is a \emph{quasi-isometric embedding} if there is a constant $C$ such that for all $x,y\in R$ we have
\begin{equation}\label{def-qi}
\frac1C\d_R(x,y)-C\le \d_S(f(x),f(y))\le C\d_R(x,y)+C;
\end{equation}
if, in addition, $S$ is contained in the $C$--neighborhood of $f(R)$, $f$ is called a \emph{quasi-isometry}.  Two metric spaces $R$ and $S$ are \emph{quasi-isometric} if there is a quasi-isometry $R\to S$. It is well-known and easy to prove that quasi-isometry of metric spaces is an equivalence relation.

Let $G$ be a group acting on a metric space $S$. Given a point $s\in S$ or a subset $A\subseteq S$, we denote by $gs$ and $gA$ the images of $s$ and $A$ under the action of an element $g\in G$. Recall that the action of $G$ on  $S$ is said to be \emph{cobounded} if there exists a bounded subset $B\subseteq S$ such that $S=\bigcup\limits_{g\in G} gB$.

The following lemma is well-known (see, for example, \cite[Lemma 3.10]{ABO}). 

\begin{lem}\label{Lem:OLip}
Let $G=\langle X\rangle$ be a group acting on a metric space $S$. Suppose that for some $s\in S$, we have $\sup _{x\in X} \d_S(s, xs)<\infty $. Then the orbit map $g\mapsto gs$ is a Lipschitz map from $(G, \d_X)$ to $S$. In particular, if $G$ is finitely generated, the orbit map is always Lipschitz.
\end{lem}

For our purpose, we will also need a generalized version of the Svarc--Milnor lemma. This result is usually stated for proper and cobounded actions (see, for example, \cite[Chapter I, Proposition 8.19]{BH}), in which case it produces finite generating sets. The same argument works without the properness assumption, but the resulting generating sets are not necessarily finite (see {\cite[Lemma 3.11]{ABO}}).

\begin{lem}\label{Lem:MS}
Suppose that a group $G$ admits a cobounded action on a geodesic metric space $S$. Then there exists a generating set $X$ of $G$ and a $G$-equivariant quasi-isometry $(G, \d_X) \to S$. 
\end{lem}

\begin{rem}\label{Rem:Equiv}
For any action $G\act S$ and any $X\in Gen(G)$, the existence of a $G$-equivariant quasi-isometry $(G, \d_X) \to S$ implies the equivalence $G\act Cay(G,X) \eA G\act S$. 
\end{rem}

\begin{figure}
 \centering \hspace*{7mm} \small\def\svgscale{1}{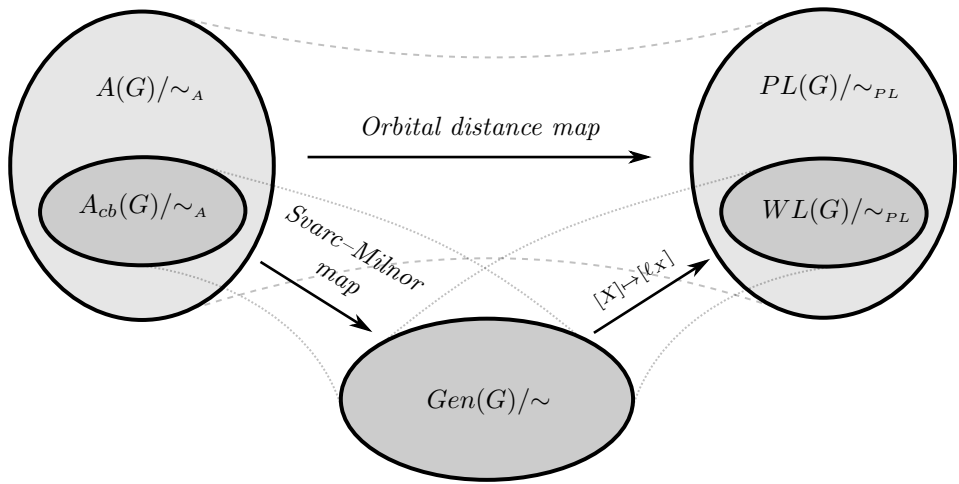}\\
  \caption{Order preserving maps between $A(G)/{\eA}$, $Gen(G)/{\sim}$, and $PL(G)/{\ePL}$.}\label{FigCD}
\end{figure}

The content of  Lemma \ref{Lem:ODqoPres}, Proposition \ref{Prop:PLGen}, and Lemma \ref{Lem:MS} can be visualized as follows. 
Let $A_{cb}(G)$ denote the set of all cobounded $G$-actions on geodesic metric spaces (of cardinality at most $2^{\aleph_0}$). It is straightforward to verify that, for any $G\act S_1, G\act S_2\in A_{cb}(G)$ and any $X_1, X_2\in Gen(G)$ satisfying the conclusion of Lemma \ref{Lem:MS} respectively, we have
$$
G\act S_1 \lA G\act S_2 \;\; \Longleftrightarrow\;\; X_1 \preccurlyeq X_2.
$$
Thus, Lemma \ref{Lem:MS} provides a well-defined order preserving map $A_{cb}/{\eA} \to Gen(G)/{\sim} $ called the \emph{Svarc--Milnor map} in \cite{ABO}. Similarly, Lemma \ref{Lem:ODqoPres} and Proposition \ref{Prop:PLGen} ensure that the \textit{orbital distance map} $A(G)/{\eA}\to PL(G)/{\ePL}$ given by  $[G\curvearrowright S]\mapsto [\ell_{G\act S,s}]$ (where $s$ is an arbitrary point of $S$) and the map $Gen(G)/{\sim} \to PL(G)/{\ePL}$ given by $[X]\mapsto [\ell_X]$ are well-defined and order-preserving. The reader can easily verify that these maps form a commutative diagram depicted on Fig. \ref{FigCD}.

%%%%%%%%%%%%%%%%%%%%%%%%%%%%%%%%%%%%%%%%%%%%%%%%%%%%%%%%%%%%%%%%%

\section{Loxodromic isometries of hyperbolic spaces} \label{Sec:GAHS}

%%%%%%%%%%%%%%%%%%%%%%%%%%%%%%%%%%%%%%%%%%%%%%%%%%%%%%%%%%%%%%%%%

In this section, we obtain several technical results on loxodromic isometries of hyperbolic spaces, which will be used to prove Theorems \ref{Thm:IsoEquiv} and \ref{Thm:Aniso}. 

\paragraph{4.1. Preliminaries on hyperbolic geometry.} \label{Sec:Prelim}
We begin by reviewing the necessary background on hyperbolic spaces and their isometry groups; our main references are the paper \cite{Gro} and books \cite{BH,DSU}.

Let $(S,d_S)$ be a metric space. Recall that the \emph{Gromov product} of two points $x,y\in S$ with respect to a point $z\in S$ is defined by
$$
(x,y)_z=\frac12\big( \d_S(x,z) +\d_S (y,z) -\d_S (x,y)\big).
$$
Further, let $\Delta $ be a geodesic triangle in $S$ with vertices $x$, $y$, $z$ and let  $[x,y]$, $[y,z]$, $[x,z]$ denote the sides of $\Delta$ connecting the corresponding vertices. Further, let $u$, $v$ be a pair of points on distinct sides of $\Delta$; for definiteness, suppose that $u\in [x,z]$ and $v\in [y,z]$. Points $u$ and $v$ are called \textit{conjugate} if $\d_S(u,z)=\d_S(v,z)\le (x,y)_z$.  

For $\delta \ge 0$, the notion of a $\delta$-hyperbolic space is typically defined using one of the following conditions (see, for example, \cite{Gro} or \cite[Chapter 2]{GdlH}).

\begin{enumerate}
    \item[({\bf H$_1$})] For any geodesic triangle $\Delta $ in $S$, the union of closed $\delta$-neighborhoods of any two sides of $\Delta$ contains the third side.
    \item[({\bf H$_2$})] For any geodesic triangle $\Delta $ in $S$ and any conjugate points $u$, $v$ on $\Delta$, we have $\d_S(u,v)\le \delta$.
    \item[({\bf H$_3$})] For any $x,y,z,t\in S$, we have $(x,z)_t\ge \min \{ (x,y)_t,\, (y,z)_t\} -\delta.$
\end{enumerate}

For convenience, we adopt the following.

\begin{defn}
A metric space $S$ is \textit{$\delta$-hyperbolic} for some $\delta\ge 0$ if $S$ is geodesic and all conditions ({\bf H$_1$})--({\bf H$_3$}) hold. A metric space is said to be \textit{hyperbolic} if it is $\delta$-hyperbolic for some $\delta\ge 0$.
\end{defn}

The following corollary of ({\bf H}$_1$) obtained by drawing diagonals will often be used without reference.

\begin{cor}\label{Cor:ngon}
    Any side of a geodesic $n$-gon in a $\delta$-hyperbolic space belongs to the closed $(n-2)\delta$-neighborhood of the union of the other sides. 
\end{cor}

By a \textit{path} (respectively, \textit{bi-infinite path}) in a metric space $S$ we mean a continuous map $p\colon I\to S$, where $I= [a,b]\subset \RR$ (respectively, $I=\RR$). By abuse of terminology, we identify paths and bi-infinite paths with their images in $S$. For a path $p$, we denote by $\ell(p)$ the length of $p$ and by $p_-=p(a)$ (respectively, $p_+=p(b)$) its origin (respectively, terminus). 

A (possibly bi-infinite) path $p$ in a metric space $S$ is said to be $(K,L)$\emph{--quasi-geodesic} for some $K\ge 1$, $L\ge 0$ if every subpath $q$ of $p$ satisfies the inequality 
$$
\ell(q)\le K\d_S(q_-, q_+)+L.
$$
Further, we say that $p$ is \emph{quasi--geodesic} if it is $(K,L)$--quasi-geodesic for some $K$ and $L$. 

Recall also that the \textit{Hausdorff distance} between two subsets $X$ and $Y$ of a metric space is at most $D\in [0, \infty)$, written $\d_{Hau}(X,Y)\le D$, if each of the subsets belongs to the closed $D$-neighborhood of the other. The following result is well-known (see, for example, {\cite[Theorem 1.7, Chapter III.H]{BH}}).

\begin{lem}\label{lem:Morse lemma}
For any $\delta\ge 0$, $K\ge 1$, and $L \ge 0$, there exists $M(\delta,K,L)>0$ satisfying the following condition. Suppose that $p$ and $q$ are $(K,L)$-quasi-geodesic paths in a $\delta$-hyperbolic space such that $p_-=q_-$ and $p_+=q_+$. Then $\d_{Hau}(p,q)\le M(\delta,K,L)$.
\end{lem}

Following \cite{Gro}, we say that a sequence $(x_n)$ of elements of a hyperbolic space $S$ \emph{converges to infinity} if $\lim_{i, j\to \infty}(x_i, x_j)_o= \infty $ for some (equivalently, any) point $o\in S$. We denote the set of all sequences of elements of $S$ converging to infinity by $\widetilde S$. Two sequences $(x_i), (y_i)\in \widetilde S$ are \emph{equivalent} if $\lim_{i, j\to \infty} (x_i,y_j)_o= \infty$. For a sequence $(x_i)\in \widetilde S$, we denote by $[(x_i)]$ its equivalence class.

As a set, the \textit{Gromov boundary} of a hyperbolic space $S$, denoted by $\partial S$,  is the set of equivalence classes in $\widetilde S$. Extending the Gromov product to $\partial S$ in a natural way defines a topology on $\widehat S = S \cup \partial S$ satisfying the following conditions (we refer the reader to Section~3.4.1 of \cite{DSU} for details).

\begin{prop}\label{Prop:Bord}
Let $S$ be a hyperbolic metric space and let $o\in S$.
\begin{enumerate}
\item[(a)] $S$ is dense in $\widehat S$ and the topology on $\widehat S$ extends the metric topology on $S$.
\item[(b)] A sequence of points $x_i\in S$ converges to a point $a=[(a_j)]\in \partial S$ if and only if $\lim\limits_{i,j\to \infty} (x_i, a_j)_o=\infty$. In particular, every $(x_i)\in \widetilde S$ converges to $[(x_i)]\in \partial S$.
\item[(c)] The rule $g[(x_i)]=[(gx_i)]$ for all $g\in Isom (S)$ and all $(x_i)\in \widetilde S$ yields a well-defined extension of the action of $Isom(S)$ on $S$ to an action on $\widehat S$ by homeomorphisms. 
\end{enumerate}
\end{prop}

Given a group $G$ acting on a hyperbolic space $S$, we denote by $\Lambda_S (G)$ the set of limit points of $G$ on $\partial S$. That is, $$\Lambda_S (G)=\overline{Gs}\cap \partial S,$$ where $s\in S$ and $\overline{Gs}$ is the closure of the $G$-orbit of $s$ in $\widehat S$. It is easy to show that this definition is independent of the choice of a particular orbit.

\paragraph{4.2. Boundary dynamics and equivalence of loxodromic elements.}\label{Sec:BDL} Recall that an element $g\in G$ is called \emph{elliptic} if it has bounded orbits in $S$. An element $g\in G$ is  \emph{loxodromic} if the map $\mathbb Z\to S$ defined by $n\mapsto g^ns$ is a quasi-isometric embedding for every $s\in S$ (here, we assume that $\mathbb Z$ is equipped with the standard metric). By $\L(G\curvearrowright S)$ we denote the set of all elements of $G$ that are loxodromic for the action $G\curvearrowright S$.

Recall also that, for every element $g\in \L(G\curvearrowright S)$, there exist $K\ge 0$, $L\ge 0$, and a $(K,L)$-quasi-geodesic bi-infinite path $L_g\colon \RR\to S$ such that $g$ stabilizes $L_g$ setwise and acts on it by translation. We call $L_g$ a \emph{$(K,L)$-quasi-axis} of $g$ (or simply a \emph{quasi-axis} of $g$ if specifying $K$ and $L$ is unnecessary). For instance, $L_g$ can be constructed as follows. 

\begin{defn}\label{Def:qa}
    Let $\gamma $ be any path connecting $s$ to $gs$ in $S$. By $L_g(\gamma)$, we denote the bi-infinite, $\la g\ra $-invariant path obtained by concatenating the consecutive geodesics 
    $$
   \ldots,\;  g^{-2}\gamma,\; g^{-1}\gamma,\; \gamma, \;g\gamma,\; g^2\gamma,\; \ldots .
    $$
    If $\gamma $ is geodesic, we call $L_g(\gamma)$ a \textit{standard quasi-axis of $g$ passing through $s$.}
\end{defn} 

Equivalently, loxodromic elements can be characterized by the following (see \cite[Lemma 8.1.G]{Gro} or \cite[Theorem 6.1.10]{DSU}).

\begin{lem} \label{Lem:NS}
Any $g\in \L(G\act S) $ fixes exactly two points $g^-, g^+\in \partial S$, and we have 
$$
\lim _{n \rightarrow \infty} g^n s=g^{+} \;\;\;  \forall s \in \widehat S \setminus \left\{g^{-}\right\}\;\;\;\;\; {\rm and }\;\;\;\;\; \lim _{n \rightarrow \infty} g^{-n} s=g^{-} \;\;\;  \forall s \in \widehat S \setminus \left\{g^{+}\right\}.
$$ 
In particular, $g^+=[(g^is)]$ and $g^-=[(g^{-i}s)]$ for any $s\in S$. 
\end{lem}

An action of a group $G$ on a hyperbolic space $S$ is said to be \textit{non-elementary} if $\Lambda_S(G)$ has at least $3$ (equivalently, infinitely many) points. If, in addition, $G$ does not fix any point of $\partial S$, the action is said to be of general type. 

Recall that elements $g,h\in \L(G\act S)$ are \textit{independent} if $\{ g^-, g^+\} \cap \{ h^-, h^+\}=\emptyset$. Actions of general type can be characterized by the existence of at least two independent loxodromic elements (see Sections 8.1, 8.2 of \cite{Gro} or Chapter 6 of \cite{DSU}). In fact, general type isometry groups of hyperbolic spaces contain an abundance of independent loxodromic elements, as shown by the following.

\begin{lem}\label{dense}
Suppose that a group $G$ acts non-elementarily on a hyperbolic space $S$. Then $\Lambda_S (G)$ has no isolated points and the set $\{ g^+ \mid g\in \L(G\act S)\}$ is dense in $\Lambda_S (G)$. Furthermore, if the action is of general type, then for any non-empty open subsets $A,B\subseteq \Lambda_S (G)$, there exists a loxodromic element $g\in G$ such that $g^+\in A$ and $g^-\in B$.
\end{lem}

Given an action of a group $G$ on a hyperbolic space $S$, we consider the diagonal $G$-action on $\partial S \times \partial S$  defined by the rule $$g(x,y)=(gx,gy)$$ for all $g\in G$ and $(x,y)\in \partial S \times \partial S$. In what follows, we will often work with the co-diagonal $G$-invariant subset 
$$
\partial S\otimes \partial S = \{ (x, y)\in \partial S \times \partial S \mid x\ne y\},
$$ which we endow with the subspace topology induced by the inclusion in $\partial S \times \partial S$. 

\begin{defn}
For every pair of distinct points $x,y\in \partial S$, we denote by $\overline{Orb}_G(x,y)$ the closure of the $G$-orbit of the pair $(x,y)\in \partial S \otimes \partial S$ in $\partial S \otimes \partial S$. 
\end{defn}

Recall that the action of a group on a topological space is \emph{minimal} if all orbits are dense.

\begin{lem}\label{Lem:Orb}
Let $G$ be a group acting on a hyperbolic space $S$. Suppose that the action $G\curvearrowright S$ is of general type. Then for any $g\in \L(G\curvearrowright S)$, the diagonal action of $G$ on the set $\overline{Orb}_G(g^-, g^+)$ is minimal.
\end{lem}

\begin{proof}
Suppose that $(x,y)\in \overline{Orb}_G(g^-, g^+)$. We fix some $s\in S$ and sequences $(x_i), (y_i)\in \widetilde S$ such that $x=[(x_i)]$ and $y=[(y_i)]$. To show that the $G$-orbit of $(x,y)$ is dense in $\overline{Orb}_G(g^-, g^+)$, it suffices to prove that $(g^-, g^+) \in \overline{Orb}_G(x,y)$. By Proposition \ref{Prop:Bord}, this amounts to showing that, for every $r>0$, there exists $a\in G$ and $M\in \NN$ such that
\begin{equation}\label{Eq:axigj}
\min\{ (ax_i, g^{-j}s)_s, \, (ay_i, g^{j}s)_s\} \ge r \;\;\;\;\; \forall \, i,j\ge M.
\end{equation}

Given $r>0$, we choose the constant $M$ and the element $a$ as follows. Since the element $g$ is loxodromic, it admits a standard $(K,L)$-quasi-axis $L_g$ passing through $s$ for some $K,L\ge 0$. Note also that $\sup_{i\in \NN} (x_i, y_i)_s<\infty $ since $x\ne y$. Let \begin{equation}\label{Eq:D}
D=\sup_{i\in \NN} (x_i, y_i)_s+3\delta +M(\delta, K,L) +\d_S(s,gs),
\end{equation}
where $M(\delta, K,L)$ is provided by Lemma \ref{lem:Morse lemma}. Proposition \ref{Prop:Bord} and our assumption $(x,y)\in \overline{Orb}_G(g^-, g^+)$ allow us to find  $b\in G$ and $N\in \NN$ such that 
\begin{equation}\label{Eq:xibgj}
\min\{ (x_i, bg^{-j}s)_s, \, (y_i, bg^{j}s)_s\} > r+D \;\;\;\;\; \forall \, i,j\ge N.
\end{equation}
Using the assumption that $g$ is loxodromic again, we can find an integer $M\ge N$ such that 
\begin{equation}\label{Eq:M}
\min\{ \d_S (s, bg^{-M}s), \d_S (s, bg^{M}s)\} \ge N\d_S(s, gs) + D.
\end{equation}

\begin{figure}
\centering
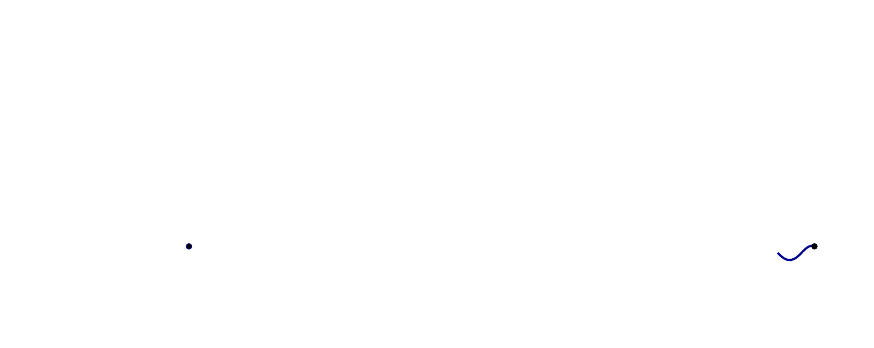
\caption{Proof of Lemma \ref{Lem:Orb}}
\label{Fig:Orb}
\end{figure}

By condition ({\bf H}$_3$) from the definition of a hyperbolic space, we have 
\begin{equation}\label{Eq:xNyN}
\begin{split}
(x_N, y_N)_s &\ge\min \{ (x_N, bg^{-N}s)_s,\, (bg^{-N}s, bg^Ns)_s,\, (y_N, bg^Ns)_s)\} -2\delta \\ &\ge  \min \{ r+D,\, (bg^{-N}s, bg^Ns)_s\} -2\delta.
\end{split}
\end{equation}
If the last minimum equals $r+D$,  we obtain $D\le (x_N, y_N)_s - r + 2\delta< (x_N, y_N)_s + 2\delta $, which contradicts  (\ref{Eq:D}).  Hence, the minimum equals $(bg^{-N}s, bg^Ns)_s$, and we can rewrite (\ref{Eq:xNyN}) as
$$
(bg^{-N}s, bg^Ns)_s\le (x_N, y_N)_s +2\delta.
$$
Let $p$ be any geodesic in $S$ connecting $bg^{-N}s $ to $bg^Ns$ (see Fig. \ref{Fig:Orb}). Condition ({\bf H}$_2$) easily implies that exists $t\in p$ such that $$\d_S(s,t)\le (bg^{-N}s, bg^Ns)_s+\delta \le  (x_N, y_N)_s+3\delta.$$ Applying Lemma \ref{lem:Morse lemma} to $p$ and the segment of $bL_g$ going from $bg^{-N}s $ to $bg^Ns$, we obtain that there exists an integer $k\in [-N, N]\subseteq [-M, M]$ such that $\d_S(t, bg^ks)\le M(\delta, K,L) + \d(s, gs)$; consequently, we have  
\begin{equation}\label{Eq:gk}
\d_S(bg^{k}s, s)\le (x_N, y_N)_s+3\delta +M(\delta, K,L) +\d(s, gs) \le D.
\end{equation}
We let $a=g^{-k}b^{-1}.$

We now show that the constant $M$ and the element $a$ chosen above satisfy (\ref{Eq:axigj}).
Let $i,j\ge M$. Using (\ref{Eq:M}) and (\ref{Eq:gk}), we obtain $$\d_S(bg^{\pm M}s, bg^ks)\ge \d_S(bg^{\pm M}s, s)- \d_S(s, bg^ks)\ge N\d(s,gs).$$ On the other hand, we have  
$\d_S(g^{\pm M}s, g^ks)\le |k\mp M|\d_S(s, gs)$ by the iterated triangle inequality. Therefore, $|k\mp M| \ge  N$. In particular, for any integer $j\ge M$, we have $k+j \ge k+M \ge N$. Combining the obvious inequality $(y_i, bg^{k+j}s)_{bg^ks} \ge (y_i, bg^{k+j}s)_{s} - \d_S(s,bg^{k}s)$  with (\ref{Eq:xibgj}) and (\ref{Eq:gk}), we obtain $(y_i, bg^{k+j}s)_{bg^ks} \ge  r$
for all $i,j\ge M$. Hence,
$$
(ay_i, g^{j}s)_s = (y_i, a^{-1}g^{j}s)_{a^{-1}s}= (y_i, bg^{k+j}s)_{bg^ks} \ge r.
$$
Using a symmetric argument, we also obtain $(ax_i, g^{-j}s)_s\ge r$ for all $i,j\ge M$, thus completing the proof of  (\ref{Eq:axigj}).
\end{proof}

Given $\e\ge 0$, we say that two paths $p$, $q$ in a metric space $S$ are \textit{oriented $\e$-close} if 
$$
\max\{ \d_S(p_-, q_-), \, \d_S (p_+, q_+)\}\le \e.
$$

\begin{defn}[Bestvina--Fujiwara, \cite{BF}]\label{Def:loxeq}
Let $G$ be a group acting on a hyperbolic space $S$. Elements $g,h\in \L(G\curvearrowright S)$ are \emph{equivalent}, denoted $g\sGS h$, if there exist quasi-axes $L_g$, $L_h$ of $g$ and $h$, respectively, and a constant $\e\ge 0$ such that for any $r>0$, there is $a\in G$ such that $aL_g$ and $L_h$ contain oriented $\e$-close subsegments of length at least $r$. 
\end{defn}

\begin{rem}\label{Rem:loxeq}
    In Definition  \ref{Def:loxeq}, we can replace ``\textit{there exist quasi-axes $L_g$ and $L_h$ of $g$ and $h$, respectively, and a constant $\e\ge 0$}" with ``\textit{for any quasi-axes $L_g$ and $L_h$ of $g$ and $h$, respectively, there exists a constant $\e\ge 0$}". This follows from the obvious observation that the Hausdorff distance between any two quasi-axes of the same loxodromic element is finite.
\end{rem}

As noted in \cite{BF}, it is not difficult to check using Lemma \ref{lem:Morse lemma} that the equivalence of loxodromic elements is indeed an equivalence relation.  This fact becomes apparent if we restate the definition in terms of the boundary dynamics as follows. 

\begin{lem}[{\cite[Lemma 2.8]{HM21}}]\label{Lem:HM}
Let $G$ be a group acting on a hyperbolic space $S$. For any elements $g,h\in \L(G\curvearrowright S)$ the following conditions are equivalent:
\begin{enumerate}
\item[(a)] $g\sGS h$;
    \item[(b)] $\overline{Orb}_G(g^+, g^-)= \overline{Orb}_G(h^+, h^-)$;
    \item[(c)] $\overline{Orb}_G(g^+, g^-)\cap  \overline{Orb}_G(h^+, h^-)\ne \emptyset$.
\end{enumerate}
\end{lem}

We mention a corollary of this result relating equivalence and independence of loxodromic elements.

\begin{cor}\label{Cor:LoxInd}
    Let $G$ be a group acting on a hyperbolic space $S$, $g,h\in \L(G\curvearrowright S)$. Suppose that $g\not\sGS h$ and $g\not\sGS h^{-1}$, then $g$ and $h$ are independent. 
\end{cor}

\begin{proof}
    Arguing by contradiction, suppose that $g$ and $h$ are not independent. Replacing one or both elements with their inverses if necessary, we can assume that $g^+=h^+$ while keeping the condition $g\not\sGS h$, thanks to our assumption. By Lemma \ref{Lem:NS}, we have $\lim\limits_{n\to \infty} g^{-n} h^-=g^-$ and, obviously, $\lim\limits_{n\to \infty} g^{-n} h^+=g^+$. Hence, $(g^-,g^+)\in \overline{Orb}(h^-,h^+)$, which contradicts $g\not\sGS h$ by Lemma \ref{Lem:HM}. 
\end{proof}

We also provide a couple of equivalent reformulations of Definition \ref{Def:loxeq}, which will be used later in our paper.

\begin{lem}\label{Lem:loxeq}
Let $G$ be a group acting on a $\delta$-hyperbolic space $S$. For any $g,h\in \L(G\curvearrowright S)$, the following conditions are equivalent.
\begin{enumerate}
\item[(a)] Elements $g$ and $h$ are equivalent.
\item[(b)] For every $s\in S$, there exists $\e>0$ such that for any $N\in \NN$, there are $a\in G$ and integers $m,n>N$ satisfying the inequality  $\max\{ \d_S(as, s), \, \d_S(ag^{m}s, h^{n}s)\} \le \e. $
\end{enumerate}
Furthermore, if $g$ and $h$ admit $(K,L)$-quasi-axes $L_g$ and $L_h$ for some $K\ge 1$, $L\ge 0$, then conditions (a) and (b) are equivalent to the following.
\begin{enumerate}
\item[(c)] For any $\ell>0$, there exists $a \in G$ such that $aL_g$ and $L_h$ contain oriented $(2\delta+2M(\delta,K,L))$-close subsegments of length at least $\ell$, where $\delta$ is the hyperbolicity constant of $S$ and the constant  $M(\delta, K,L)$ is provided by Lemma \ref{lem:Morse lemma}.
\end{enumerate}
\end{lem}

\begin{proof}
First, note that for any $f\in \L(G\act S)$, any quasi-axis $L_f$ of $f$, and any $s\in S$, there exists $B\ge 0$ such that $L_f$ belongs to the $B$-neighborhood of the orbit $\langle f\rangle s$. Applying this observation to the elements $g$ and $h$, we easily obtain the equivalence of (a) and (b).

Next, we prove the implication (a) $\Rightarrow$ (c). By Remark \ref{Rem:loxeq}, there exists $\e>0$ such that, for any $r>0$, there is $a\in G$ such that $aL_g$ and $L_h$ contain oriented $\e$-close segments $p$ and $q$ of length at least $r$.  Henceforth, we assume that $r$ is sufficiently large (the exact value will be specified later), and fix an element $a$ as above. Let $x_1$, $x_2$ be points on $p$ so that 
\begin{equation}\label{Eq:xy}
\min \{ \d_S(x_1, p_-),\, \d_S(x_2, p_+)\} = D +\e,
\end{equation}
where $D=2(\delta +M(\delta, K,L))$ (see Fig. \ref{Fig:Lgh}).
Let $[p_-, q_-]$ and $[p_+, q_+]$ denote geodesic paths connecting $p_-$ to $q_-$ and $p_+$ to $q_+$, respectively. Condition ({\bf H}$_1$) from the definition of a $\delta$-hyperbolic space easily implies that any side of a geodesic quadrilateral in $S$ belongs to the closed $2\delta$-neighborhood of the union of the other three sides. Combining this with Lemma~\ref{lem:Morse lemma}, we obtain points $y_1, y_2\in [p_-, q_-]\cup q \cup [p_+, q_+]$ such that $\d_S(x_i, y_i)\le D$ for $i=1,2$. Clearly,  (\ref{Eq:xy}) guarantees that $y_1, y_2\in q$. Since $L_g$ and $L_h$ are $(K,L)$-quasi-geodesic, we have $$\d_S(x_1, x_2) \ge \d_S(p_-, p_+) -2(D+\e) \ge (r-L)/K - 2(D+\e);$$ consequently  
$$
\d_S(y_1, y_2) \ge \d_S(x_1, x_2) - 2D \ge (r-L)/K - 4D -2\e.
$$
Let $p^\prime$ and $q^\prime$ denote the subpaths of $p$ and $q$ connecting $x_1$ to $x_2$ and $y_1$ to $y_2$, respectively. Clearly, $p^\prime$ and $q^\prime$ are oriented $D$-close. Taking $r$ to be sufficiently large, we can ensure that these segments have length at least $\ell$ and (c) follows. To complete the proof of the lemma, it remains to note that the implication (c) $\Rightarrow$ (a) is trivial. 
\end{proof}

\begin{figure}
  \centering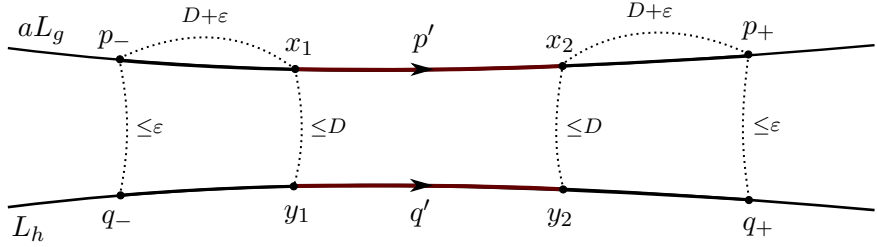
  \caption{The proof of (a) $\Rightarrow$ (b) in Lemma \ref{Lem:loxeq}}\label{Fig:Lgh}
\end{figure}

\paragraph{4.3. Constructing inequivalent loxodromic elements.}\label{Sec:IneqL} In this section, we obtain some technical results necessary for the proof of Theorem \ref{Thm:Aniso}. 

\begin{defn}
 We say that the action of a group $G$ on a hyperbolic space $S$ satisfies the \textit{Bestvina-Fujiwara condition} if there exist independent elements $g,h\in \L(G\act S)$ such that $g\not\sGS h$.   
\end{defn} 

Our terminology is motivated by the paper \cite{BF} by Bestvina and Fujiwara, where this condition was used to derive several interesting properties of the group $G$ and the action $G\act S$, including the following. 

\begin{prop}[Bestvina-Fujiwara, {\cite[Proposition 2]{BF}}]\label{Prop:BF-original}
  Suppose that the action of a group $G$ on a hyperbolic space $S$ satisfies the Bestvina-Fujiwara condition. Then there exist elements $f_i\in \L(G\act S)$ ($i\in \NN$) such that $f_i\not\sGS f_i^{-1}$ and $f_i\not\sGS f_j^{\pm 1}$ for all $i$ and all $j\ne i$.
\end{prop}

In our paper, we will need a strengthening of this result. By a \emph{combinatorial path} in a connected graph $S$ we mean a graph homomorphism $p \colon I \to S$, where $I \subset \RR$ is an interval whose endpoints are in $\ZZ\cup\{\pm\infty\}$ and integer points are vertices.

\begin{prop}\label{Prop:BF}
    Let $G$ be a group acting on a hyperbolic space $S$ and let $o\in S$. Suppose that the action $G \act S$ is of general type and there exist $g,h \in \L(G \act S)$ such that $g \not\sGS h$. Then there exist $K\ge 1$, $L \ge 0$, and a sequence $(g_n)_{n \in \NN} \subset \L(G\act S)$ satisfying the following conditions.
    \begin{itemize}
        \item[(a)]
        For any $n\in\NN$, there exists a $(K,L)$-quasi-axis $L_{g_n}$ of $g_n$ passing through $o$. If in addition, $S$ is a hyperbolic graph, then we can take $L_{g_n}$ to be a combinatorial path.
        \item[(b)]
        For any $r \ge 0$ and any $n\in\NN$, we have (see \eqref{eq:A^+sigma})
        \begin{equation}\label{Eq:SC}
        \sup\big\{\diam (L_{g_n}\cap aL_{g_m}^{+r})\mid a\in G, m\ne n\big\} <\infty.
        \end{equation}
        In particular, $g_m$ and $g_n$ are independent and non-equivalent for any $m\ne n$.
    
        \item[(c)] For every $n\in \NN$, we have $g_n\not\sim_{G\act S} g_n^{-1}$.
    \end{itemize}
\end{prop}

\begin{rem}\label{Rem:BF} Note that the assumption of our proposition is a priori weaker and the conclusion is stronger than that of Proposition \ref{Prop:BF-original}. Indeed, Bestvina and Fujiwara assume the existence of two \textit{independent}, non-equivalent elements in $\L(G\act S)$. In our proposition, independence is omitted from the assumption. In particular, this allows us to apply Proposition \ref{Prop:BF} in the situation (occurring in the proof of Theorem \ref{Thm:Iso}),  when we have an element $g\in \L(G\act S)$ such that $g$ is not equivalent to $h=g^{-1}$. Furthermore, our result provides a small-cancelation-like property (\ref{Eq:SC}) that is significantly stronger than plain non-equivalence claimed in Proposition \ref{Prop:BF-original}; this property is crucial for the proof of Theorem \ref{Thm:Aniso}. 
\end{rem}

We first use a purely topological argument to prove the following.

\begin{lem}\label{Lem:BF}
    Under the assumptions of Proposition \ref{Prop:BF}, the action $G\act S$ satisfies the Bestvina-Fujiwara condition.
\end{lem}

\begin{proof}
Let
    $$
        \Lambda_S(G)\otimes \Lambda_S(G) = \{(x,y) \in \Lambda_S(G) \times \Lambda_S(G) \mid x \neq y\}.
    $$
We claim that, for any $f\in \L(G\act S)$, the interior of the subset  $\overline{Orb}_G(f^-, f^+)$ in $ \Lambda_S(G)\otimes \Lambda_S(G)$ is empty. Indeed, suppose that a non-empty subset $U \subset \overline{Orb}_G(f^-, f^+)$ is open in $\Lambda_S(G) \otimes \Lambda_S(G)$. Since the action $G \act S$ is of general type, the action $G \act \Lambda_S(G) \otimes \Lambda_S(G)$ is topologically transitive (see \cite[8.2.H]{Gro}). Hence, $\bigcup_{a\in G} aU$ is dense in $\Lambda_S(G) \otimes \Lambda_S(G)$. This implies the equality $\overline{Orb}_G(f^-, f^+)=\Lambda_S(G) \otimes \Lambda_S(G)$, which contradicts the existence of non-equivalent loxodromic elements by Lemma \ref{Lem:HM}.

Let $g, h\in \L(G\act S)$ be non-equivalent elements. The set $$V=\overline{Orb}_G(g^-, g^+)\cup \overline{Orb}_G(h^-, h^+)$$ is nowhere dense in $\Lambda_S(G) \otimes \Lambda_S(G)$ being a union of two nowhere dense sets. Therefore, the complement of $V$ in $\Lambda_S(G) \otimes \Lambda_S(G)$ is non-empty; clearly, it is also open. Lemma \ref{dense} allows us to find $k\in \L(G\act S)$ such that $(k^+, k^-) \notin V$. By Lemma~\ref{Lem:HM}, we have $g\not\sGS k$ and $h\not \sGS k$. Note that if $g^{-1}\sGS h$ and $g^{-1}\sGS k$, then $h\sGS k$, and we obtain a contradiction. Thus, $g^{-1}$ is not equivalent to at least one of the elements $h$, $k$, and we complete the proof by applying Corollary \ref{Cor:LoxInd}.
\end{proof}

An interesting feature of our proof of Proposition \ref{Prop:BF} is that we will systematically use quasi-morphisms to derive geometric facts. We will need the following result, which can be easily extracted from the proof of Theorem 1 in \cite{BF} as explained below. For the notion of a quasi-morphism subordinate to a given action, see Definition \ref{Def:subord}.

\begin{lem}[Bestvina--Fujiwara, \cite{BF}]\label{Lem:BFq}
Let $G$ be a group acting on a hyperbolic space. Suppose that there exist $f_1, f_2\in \L(G\act S)$ such that $f_i\not\sGS f_i^{-1}$ for $i=1,2$,  $f_1\not\sGS f_2$, and $f_1\not\sGS f_2^{-1}$. Then there exist quasi-morphisms $q_1, q_2\colon G\to \RR$ subordinate to the action $G\act S$ such that $q_i$ is unbounded on $\la f_i\ra$ and is zero on $\la f_j\ra$ for $i=1,2$ and $j\ne i$.
\end{lem}

\begin{proof}[On the proof]
The existence of quasi-morphisms $q_i$ (denoted by $h_{f_i^a}$ in \cite{BF}) with specified values on  $\langle f_i\rangle $ and $\langle f_j\rangle$ is proved in \cite[Proposition 5]{BF}. The fact that these quasi-morphisms are subordinate to $G\act S$ is clear from their definition. Indeed, the authors of \cite{BF} fix a point $x_0\in S$ and define each $q_i$ as the difference of two quantities, $c_w(x_0, gx_0)$ and $c_{w^{-1}}(x_0, gx_0)$ in the notation of \cite{BF}; each of these quantities is easily seen to be bounded by $M\d_S(s,gs)$ for some $M\ge 0$ independent of $g$ (the latter fact is stated explicitly in  \cite[Lemma 3.9]{F98}). 
\end{proof}

Let $q\colon G\to \RR$ be a quasi-morphism. We denote by $D(q)$ the \textit{defect} of $q$, defined by 
$$
D(q)=\sup_{g,h\in G} |q(gh)-q(g)-q(h)|.
$$
A quasi-morphism $G\to \RR$ is called a \textit{quasi-character} if its restriction to every cyclic subgroup of $G$ is a homomorphism. It is well-known and straightforward to check that the  \textit{homogenization} of $q$ defined by the formula
$$
\widetilde q(g)=\lim_{n\to \infty} \frac{q(g^n)}n \;\;\;\;\; \forall\, g\in G
$$
is a quasi-character and $\sup_{g\in G}|q(g)-\widetilde q(g)|\le D(q)$ (see, for example, \cite[Lemma 2.21]{scl}). 

\begin{lem}\label{Lem:qgg-1}
 Let $G$ be a group acting on a hyperbolic space and let $q\colon G\to S$ be a quasi-character subordinate to $G\act S$. If $g$ is an element of $G$ such that $|q(g)|>0$, then $g\in \L(G\act S)$ and $g\not\sGS g^{-1}$.
\end{lem}

\begin{proof}
Since $q$ is subordinate to $G\curvearrowright S$, there are $s\in S$ and $M\ge 0$ such that (\ref{Eq:subord}) holds; in particular, we have $\d_S(s, g^ns)\ge |q(g^n)|/M-1=  n|q(g)|/M-1$. We also have $\d_S(s, g^ns) \le \sum_{i=1}^n\d_S(g^{i-1}s, g^is)=n\d_S(s, gs)$, hence the map $n \mapsto g^ns$ is a quasi-isometric embedding.
Thus, $g\in \L(G\act S)$. Further, assume that $g\sim_{G\act S} g^{-1}$. According to Lemma \ref{Lem:loxeq}, there exists $r\ge 0$ such that for any $N\in \NN$, there are integers $m,n\ge N$ and $a\in G$ such that
$$
\max\{ \d_S(as, s), \, \d_S(g^nag^ms, s)\} = \max\{ \d_S(as, s), \, \d_S(ag^ms, g^{-n}s)\} \le r. 
$$
Using this inequality and (\ref{Eq:subord}), we obtain 
\begin{equation*}
     Mr+M \ge |q(g^nag^m)| \ge (m+n)|q(g)| - |q(a)| - 2D(q)\ge 2N|q(g)| - Mr-M - 2D(q),
\end{equation*}
which is a clear nonsense if $N$ is sufficiently large. 
\end{proof}

 We are now ready to prove the main result of this subsection. 

\begin{proof}[Proof of Proposition \ref{Prop:BF}]
By Lemma \ref{Lem:BF}, the action $G\act S$ satisfies the Bestvina-Fujiwara condition. By Proposition \ref{Prop:BF-original} and Lemma \ref{Lem:BFq} there exist $f_1,f_2\in \L(G\act S)$ and quasi-morphisms $q_1,q_2\colon G\to \RR$  subordinate to $G\act S$ such that, for any $i=1,2$ and $j\ne i$, $q_i$ is unbounded on $\la f_i\ra$ and identically $0$ on $\la f_j\ra$. Passing to homogenizations, we can assume that $q_1$ and $q_2$ are quasi-characters. We fix any $M\ge 0$ such that 
\begin{equation}\label{Eq:qisub}
    |q_i(g)|\le M\d_S(o, go)+M\;\;\;\;\; \forall\,g \in G,\, {\forall\,} i\in \{ 1,2\}.
\end{equation} 
Finally, rescaling the quasi-characters $q_i$ and replacing the elements $f_i$ with their powers if necessary, we can assume that 
\begin{equation}\label{Eq:qifi}
q_i(f_i)= 3 \;\;\;\;\; {\rm and} \;\;\;\;\; D(q_i)\le 1\;\;\;\;\; {\forall\,} i\in \{ 1,2\}.
\end{equation}

Let $\gamma _1$, $\gamma_2$ be any geodesics in $S$ going from  $o$ to $f_1o$ and $f_2o$, respectively.
For any word $w=f_{i_1}f_{i_2}\ldots f_{i_n}$, where $f_{i_j}\in \{ f_1, f_2\}$ for all $1\le j\le n$, we define a path $\gamma_w$ connecting $o$ to $wo$ in $S$ as the concatenation  $$\gamma_w=\gamma_{i_1} (f_{i_1}\gamma_{i_2}) \ldots  (f_{i_1}f_{i_2}\cdots f_{i_{n-1}} \gamma_{i_n}),$$ where $f\gamma$ denotes the image of a path $\gamma$ in $S$ under $f\in G$.   By (\ref{Eq:qisub}) and (\ref{Eq:qifi}), we have
$$ 
2M\d_S (o, wo) +2M \ge (q_1+q_2)(w) \ge \sum\limits_{j=1}^n \Big(q_1(f_{i_j}) + q_2(f_{i_j})\Big) - 2(n-1) >n\ge \frac{\ell(\gamma_w)}{A}, 
$$ 
where  $A=\max\{ \ell(\gamma_{1}),\, \ell(\gamma_2)\}$. This implies that the path $\gamma_w$ is 
$(K,L)$-quasi-geodesic for some $K$ and $L$ that depend on $A$ and $M$ only.  

For every $n\in \NN$, we let $$g_n= f_1f_2^{10^n}.$$
Applying the argument from the previous paragraph to powers of the word $f_1f_2^{10^n}$, we obtain that every $g_n$ is loxodromic and admits a $(K,L)$-quasi-axis $L_{g_n}=L_{g_n}\big(\gamma_{f_1f_2^{10^n}}\big)$ (in the notation of Definition \ref{Def:qa}) passing through $o$. By Lemma \ref{Lem:qgg-1}, elements $g_n$ also satisfy condition (c). It remains to prove (b). 

Arguing by contradiction, suppose that there exists $r\ge 0$ and $n\in \NN$ such that (\ref{Eq:SC}) is false; that is, for arbitrarily large $R\in \RR$, there are $m\in \NN\setminus\{n\}$, $a\in G$, and points $x_1, x_2\in L_{g_n}$, $y_1, y_2 \in aL_{g_m}$ such that $\d_S(x_i, y_i)\le r$ for $i=1,2$ and $\d_S(x_1, x_2)\ge R$ (see Fig. \ref{Fig:Lgmgn}).
Let $B=\ell(\gamma_{g_n})$ and let $g_n^{k_1}o$ and $g_n^{k_2}o$ be points on $L_{g_n}$ such that $\d_S(x_i, g_n^{k_i}o)\le B$ for $i=1,2$. Let $k=k_2-k_1$; without loss of generality, we can assume that $k\ge 0$. Similarly, we fix $h_1, h_2\in G$ such that $h_1o, h_2o\in aL_{g_m}$, $\d_S(y_i, h_io)\le A$ for $i=1,2$, and the element $h=h_1^{-1}h_2$ or its inverse (depending on the orientation of $L_{g_m}$) can be represented by a word $f_2^\alpha(f_1f_2^{10^m})^{\ell}f_1f_2^\beta$ for some integers $\ell \ge 0$ and $\alpha, \beta\in [0, 10^m]$.
Clearly, we have $$Bk= B(k_2-k_1)\ge \d_S(o, g_n^{k_2-k_1}o)=\d_S(g_n^{k_1}o, g_n^{k_2}o)\ge R-2B$$
Thus, taking $R$ to be sufficiently large, we can ensure that 
\begin{equation}\label{Eq:k1k2}
k\ge \max\{ 50, 5 + 2M (A+B+r) + 2M\}.
\end{equation}

In what follows, we write $r\stackrel{\e}\approx s$ for some numbers $r,s\in \RR$ and $\e\ge 0$ if $|r-s|\le \e$. Evaluating $q_1$ and $q_2$ at the product $f_2^\alpha(f_1f_2^{10^m})^{\ell}f_1f_2^\beta $ and using (\ref{Eq:qifi}), we obtain
\begin{equation}\label{Eq:qih1}
|q_i(h)| \stackrel{\ell+3}\approx \left\{\begin{aligned}
    & 3\ell +3,\; & {\rm for \;} i=1, \\
& 3\alpha +3\beta + 3\cdot 10^m\ell \stackrel{6\cdot 10^m}\approx 3\cdot 10^m\ell, & {\rm for \;} i=2.
\end{aligned}
\right.
\end{equation}
On the other hand, we have $h=h_1^{-1} h_2= (h_1^{-1}g_n^{k_1})\cdot g_n^{k} \cdot (h_2^{-1}g_n^{k_2})^{-1}$. Since 
$$   \left|q_i\left(h_j^{-1}g_n^{k_j}\right)\right|\le M\d_S\left(g_n^{k_j}0, h_jo\right) +M \le M(A+B+r)+M,
$$
for $i,j=1,2$, we obtain
$|q_i(h) - kq_i(g_n)|\le 2 + 2M(A+B+r)+2M\le k-3$ by (\ref{Eq:k1k2}).
Therefore, 
\begin{equation}\label{Eq:qih2}
 |q_i(h)| \stackrel{k-3}\approx k|q_i(g_n)| \stackrel{k}\approx \left\{\begin{aligned}
& 3k,\; & {\rm for \;} i=1, \\
& 3\cdot 10^nk, & {\rm for \;} i=2.
\end{aligned}\right.   
\end{equation}
Combining (\ref{Eq:qih1}) and (\ref{Eq:qih2}) for $i=1$ yields $3k\stackrel{2k+\ell}\approx  3\ell+3$, which is equivalent to
\begin{equation}\label{Eq:kl}
\frac25 \ell +\frac35 \le k\le 4\ell +3.
\end{equation}

We now consider two cases. Assume first that $m<n$. Using (\ref{Eq:qih1}) and (\ref{Eq:qih2}) for $i=2$, we obtain  $3\cdot 10^n k \le 3\cdot 10^m \ell +2k +\ell +6\cdot 10^m$. 
Dividing by $3\cdot 10^m\ge 30$ and taking into account that $10^{n-m}\ge 10$, we obtain $149k/15  \le 31\ell/30 +2$, which contradicts the left inequality in (\ref{Eq:kl}). Similarly, if $m>n$, we first note that 
$3\cdot 10^m \ell \le 3\cdot 10^n k +2k+\ell + 6\cdot 10^m$; then,
dividing by $3\cdot 10^m\ge 30$, we obtain $29\ell/30 \le k/6 +2$, which contradicts the right inequality in (\ref{Eq:kl}) as $k\ge 50$ by (\ref{Eq:k1k2}).

\begin{figure}
  \centering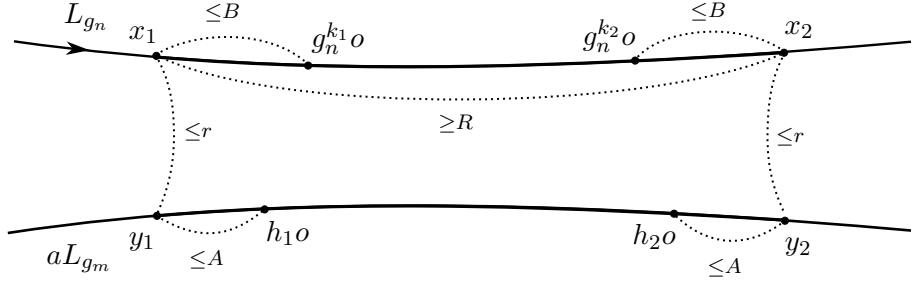
  \caption{Proof of Proposition \ref{Prop:BF}.}\label{Fig:Lgmgn}
\end{figure}
\end{proof}

%%%%%%%%%%%%%%%%%%%%%%%%%%%%%%%%%%%%%%%%%%%%%%%

\section{Modifying group actions on hyperbolic spaces} \label{Sec:ModGA} 

%%%%%%%%%%%%%%%%%%%%%%%%%%%%%%%%%%%%%%%%%%%%%%%

In this section, we consider certain modifications of group actions on hyperbolic spaces, which can be informally described as compressing hyperbolic spaces along axes of loxodromic elements. Our main result is Proposition \ref{Prop:EmbPi}, which presents the main step in the proof of Theorem \ref{Thm:Aniso}. 

\paragraph{5.1. Passing to cobounded actions.} 
We will need the following simplification of \cite[Proposition 3.1]{Bow} (in the notation of Bowdich's paper, the result below is obtained by letting $L(x,y)$ be a geodesic connecting $x$ and $y$ in $S$). Alternatively, it is a particular case of \cite[Proposition 2.5]{KR}, whose proof makes use of an earlier result of Bowditch. 

\begin{prop}[Bowditch, Kapovich--Rafi]\label{Prop:Bow}
For any $\delta, D\ge 0$, there are $\e=\e(\delta, D)\ge 0$ and $C=C(\delta, D)\ge 0$ with the following property. Let $S$ be a $\delta$-hyperbolic graph, and let $R$ be a graph obtained from $S$ by adding edges. Suppose that for any vertices $x$, $y$ connected by an edge in $R$, every geodesic in $S$ connecting $x$ and $y$ has diameter at most $D$ in $R$. Then: 
\begin{enumerate}
\item[(a)] $R$ is $\e $-hyperbolic;
\item[(b)] for any (combinatorial) geodesic path $p$ in $S$ and any geodesic $q$ in $R$ with the same endpoints as $p$, the Hausdorff distance between $p$ and $q$ in $R$ is at most $C$.     
\end{enumerate}  
\end{prop}

Given a subset $A$ of a metric space $S$ and any $\sigma\in [0, \infty)$, we let 
\begin{align}\label{eq:A^+sigma}
A^{+\sigma}=\bigcup_{a\in A} \{ s\in S\mid \d_S(a,s)\le \sigma\}.    
\end{align}
The proof of the following result is inspired by an idea of Sisto, who established the second claim of the proposition (concerning the independence of loxodromic elements) in \cite[Proposition A.1 and Corollary A.2]{BFGS}. Our contribution is the proof of non-equivalence. We partially reproduce Sisto's argument to keep our exposition self-contained. 

\begin{prop}\label{Prop:Cob}
Let $G$ be a group acting on a hyperbolic space $S$. For any  elements $g, h\in \L(G\act S)$, there exists a cobounded $G$-action on a hyperbolic space $R$ such that $G\act R\lA G\act S$, and if $g$ and $h$ are non-equivalent (respectively, independent) with respect to the action on $S$, then they are non-equivalent (respectively, independent) with respect to the action on $R$. \end{prop} 

\begin{proof}
Suppose that $S$ is $\delta$-hyperbolic. Note that passing to an equivalent $G$-action on a hyperbolic space preserves the set of loxodromic elements, as well as the independence and equivalence relations.  Since every action of a group on a geodesic metric space is equivalent under the relation $\sim_A$ to an action on a graph (see, for example, \cite[Proposition 3.14]{AHO20}), we can assume that $S$ is a graph without loss of generality. We fix any vertex $s\in S$ and let $L_g$, $L_h$ be standard quasi-axes  of $g$ and $h$ passing through $s\in S$ (see Definition \ref{Def:qa}).

For any $A \in \NN$, we define $R_A$ to be the graph obtained from $S$ by adding an edge between any two vertices $x,y \in S$ such that there exists a geodesic in $S$ connecting $x$ to $y$ that does not intersect  $Gs^{+A}$. We call an edge of $R_A$ \textit{new} if it is not an edge of $S$. Since our definition of new edges is $G$-equivariant, the action $G\act S$ extends to a $G$-action on $R_A$ by graph automorphisms. Further, given a vertex $u$ of $S$, consider a geodesic $p$ in $S$ going from $u$ to a nearest vertex in $Gs$. If $\ell(p)\ge A+1$, let $v$ be the vertex on $p$ such that  $\d(u,v)=\ell(p)-A-1$ (see Fig. \ref{Fig:Cobounded}). By the choice of $p$, the initial segment $[u,v]$ of $p$ does not intersect $Gs^{+A}$. Hence, $\d_{R_A}(u,v)\le 1$ and, consequently, $u$ belongs to the $(A+2)$-neighborhood of $Gs$ in $R_A$. Thus, the action $G \act R_A$ is cobounded. Obviously, we also have $G\act R_A\lA G\act S$. 

\begin{figure}
  \centering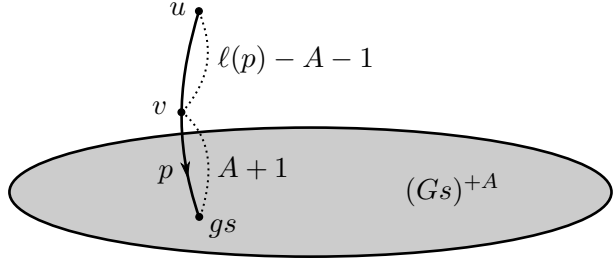
  \caption{Proving that the action  $G\act R_A$ is cobounded.}\label{Fig:Cobounded}
\end{figure}

It is easy to see that the assumptions of Proposition \ref{Prop:Bow} are satisfied for $D=2\delta +1$. Hence, there exist $\e, C\ge 0$ (which depend only on $\delta$) such that $R_A$ satisfies conditions (a) and (b). Further, we fix any $K\ge 1$, $L\ge 0$ such that $L_g$ and $L_h$ are $(K,L)$-quasi-geodesic in $S$. Let $M=\max\{ M(\e, K,L),\, M(\delta, K,L)\}$, where $M(\e, K,L)$ and $M(\delta, K,L)$ are provided by Lemma \ref{lem:Morse lemma}, and let $r = 2(\e + M)
$  and
\begin{equation}\label{Eq:A}
A = r  + 3\delta + M + C +\max\{ \d_S(s,gs), \, \d_S(s, hs)\}.
\end{equation}
We first prove the following.

\medskip

\noindent {\bf Claim.} \textit{For any new edge $e$ of $R_A$ and any endpoint $x$ of $e$, we have $\d_{R_A}(x, Gs)>A$}.

\medskip

Let $f$ be an element of $G$ such that $\d_{R_A}(x, fs) = \d_{R_A}(x, Gs)$ and let $p$ be a geodesic in $R_A$ going from $fs$ to $x$. Without loss of generality, we can assume that $p$ contains no new edges (otherwise, we replace $x$ with the origin of the first new edge in $p$). Thus, $\d_S(x, Gs)\le \ell(p)=\d_{R_A}(x, Gs)$. On the other hand, we have $\d_S(x, Gs)>A$ as $e$ is a new edge, and the claim follows. As a corollary, we obtain:
\begin{itemize}
    \item[$(\ast)$]
    \textit{For any vertices $x,y \in (L_g\cup L_h)^{+r}$ in $R_A$, no geodesic in $R_A$ connecting $x$ to $y$ contains new edges; in particular, $\d_{R_A}(x,y) = \d_S(x,y)$.}
\end{itemize}

Indeed, by the claim and inequality $A \ge r + \max\{ \d_S(s,gs), \, \d_S(s, hs)\}$, vertices $x$ and $y$ belong to $(L_g\cup L_h)^{+r}$ in $S$. Combining with Corollary \ref{Cor:ngon} applied to the obvious geodesic pentagon (represented by red lines in Fig. \ref{Fig:penta}) and Lemma \ref{lem:Morse lemma}, we obtain that any geodesic $p$ in $S$ connecting $x$ to $y$, belongs to $(L_g\cup L_h)^{+(r +3\delta + M)}$ in $S$. By part (b) of Proposition \ref{Prop:Bow}, any geodesic $q$ in $R_A$ connecting $x$ to $y$ belongs to the closed $C$-neighborhood of $p$ in $R_A$, and hence to $Gs^{+A}$ in $R_A$. Using the claim again, we conclude that $q$ contains no new edges.

Property ($\ast$) immediately implies that $g, h\in \L(G\act R_A)$,  $L_g$ and $L_h$ are $(K,L)$-quasi-axes of $g$ and $h$ in $R_A$, and if $g$ and $h$ are independent with respect to the action on $S$, then they are independent with respect to the action on $R_A$. Further, suppose that $g\sim_{G\act R}h$. By Lemma \ref{Lem:loxeq}, for every $\ell \ge 0$, there exists $a \in  G$ such that $aL_g$  and $L_h$ contain subsegments $p$ and $q$ of length at least $\ell$ that are oriented
$r$-close in $R_A$. By ($\ast$), the same configuration appears in $S$, which implies that $g\sGS h$. 
\end{proof}

We record a useful corollary. Since it only uses the portion of Proposition \ref{Prop:Cob} proved in the appendix to \cite{BFGS}, we attribute it to Sisto.

\begin{cor}[Sisto]\label{Cor:Sisto}
A group $G$ is weakly hyperbolic if and only if there exists a generating set $X$ of $G$ such that the Cayley graph $Cay(G,X)$ is hyperbolic of general type. 
\end{cor}

\begin{proof}
    Suppose $G$ is weakly hyperbolic. By Proposition \ref{Prop:Cob} (equivalently, Corollary A.2 in \cite{BFGS}), there exists a cobounded, general type action of $G$ on a hyperbolic space. To complete the proof, it remains to apply Lemma \ref{Lem:MS} and note that the property of being of general type is preserved under the equivalence of actions.
\end{proof}

\begin{figure}
  \centering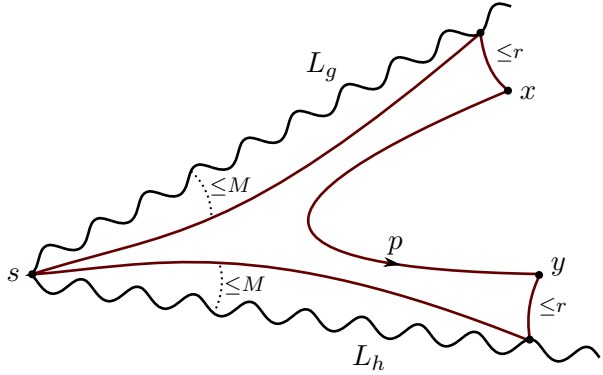
  \caption{Red lines represent geodesics in $S$, all distances are measured in $S$.}\label{Fig:penta}
\end{figure}

\paragraph{5.2. Compressing hyperbolic spaces along axes of loxodromic isometries.}
We now present a general construction that will allow us to compress certain group actions on hyperbolic spaces in infinitely many ``independent" directions. For simplicity, we restrict to the particular case of group actions on Cayley graphs, which is sufficient for our goals. 

Throughout this subsection, we fix a group $G$, a generating set $X\in Hyp(G)$, and $\delta\ge 0$  such that $Cay(G,X)$ is $\delta$-hyperbolic. 
Further, we fix an infinite sequence of loxodromic elements $g_i\in \L(G\act Cay(G,X))$, $i\in \NN$. Depending on the situation, we will assume that $(g_i)$ satisfies one or both of the following conditions. 
\begin{enumerate}
    \item[({\bf L}$_1$)] There exist $K\ge 1$, $L\ge 0$ such that every $g_i$ admits a $(K,L)$-quasi-axis $L_{g_i}$ in $Cay(G,X)$ passing through $1$.

    \item[({\bf L}$_2$)] For any $i\in \NN$ and any $r\ge 0$, we have $
   \sup\{\diam(L_{g_i}\cap (aL_{g_j})^{+r})\mid a\in G,\; j\ne i\}<\infty $.
\end{enumerate}

Given a word $w \in G$ in the alphabet $X^{\pm 1}$ and $n \in \NN$, we define $S(w,n) $ to be the set of all subwords of the words $w^{\pm n}$ and let $S(w,\infty)=\bigcup_{n\in \NN}S(w,n)$. For every $i\in \NN$, let $w_i$ denote the label of the segment of $L_{g_i}$ going from $1$ to $g_i$. Thus, $w_i$ is a word in the alphabet $X^{\pm 1}$ representing $g_i$ in $G$ and every subpath of $L_{g_i}$ is labeled by a word from $S(w_i,\infty)$. With every sequence $(n_i)\in (\NN\cup \{\infty\})^\NN$, we associate a new generating set  of $G$ defined by   
\begin{equation}\label{Eq:Wni}
      W_{(n_i)}= X\cup\left(\bigcup_{i \in \NN}S(w_i,n_i)\right)
    \end{equation}
In this notation, we have the following.

\begin{lem}\label{lem:Cay(G,W) hyperbolic}
    Suppose that $(g_i)$ satisfies {\rm({\bf L}$_1$)}; then for any sequence $(n_i)\in (\NN\cup \{\infty\})^\NN$, the set
    $W=W_{(n_i)}$ belongs to $Hyp(G)$ and we have $G \act Cay(G,W) \lA G \act Cay(G,X)$.
\end{lem}

\begin{proof}
  The Cayley graph $Cay(G,W)$ is obtained from $Cay(G,X)$ by adding edges corresponding to generators from $\bigcup_{i \in \NN}S(w_i,n_i)$. Suppose that $x,y\in G$ are connected by an edge in $Cay(G,W)$ labeled by some $w\in \bigcup_{i \in \NN}S(w_i,n_i)$. Let $p$ be the path in $Cay(G,X)$ starting at $x$ and labeled by $w$, and let $q$ be a geodesic connecting $x$ to $y$ in $Cay(G,X)$. By ({\bf L}$_1$), $p$ is $(K,L)$-quasi-geodesic. Hence, $q\subset p^{+M(\delta,K,L)}$ in $Cay(G,X)$ by Lemma~\ref{lem:Morse lemma}. Note that every vertex of $p$ is connected to $x$ by an edge in $Cay(G,W)$ as the set $W$ is closed under taking subwords. It follows that $q$ belongs to the closed $(M(\delta,K,L)+1)$-neighborhood of $x$ in $Cay(G,W)$. Thus, the assumptions of Proposition~\ref{Prop:Bow} are satisfied for $D=2(M(\delta, K,L)+1)$, and hence $W\in Hyp(G)$. Clearly, the inclusion $X\subseteq W$ implies the inequality $G \act Cay(G,W) \lA G \act Cay(G,X)$.
\end{proof}

To prove our next result, we will need the following simplification of \cite[Lemma 10]{Ol}. For a path $p$ in a metric space, we denote by $p_{\pm}$ the set $\{ p_-, p_+\}$.

\begin{lem}[Olshanskii]\label{Lem:Ols}
Let $\mathcal P$ be a geodesic $n$-gon in a $\delta$-hyperbolic space for some $n\ge 2$ and let $Q$ be a subset of sides of $\mathcal P$. Assume that the total length of all sides from $Q$ is at least $10^3cn$ for some $c\ge 30 \delta$. Then there exist $q\in Q$, another side $p$ of $\mathcal P$, and subsegments $u$ and $v$ of $p$ and $q$, respectively, such that $\d_{Hau}(u_\pm, v_\pm)\le 13\delta$ and $\min \{ \ell(u), \, \ell(v)\} >c$.
\end{lem}

To simplify the statement of the next lemma, we adopt the convention $x/\infty=0$ for all numbers $x\in \RR$.

\begin{lem}\label{lem:prelim for embedding K_sigma}
   Suppose that the sequence $(g_i)$ satisfies {\rm ({\bf L}$_1$)} and {\rm ({\bf L}$_2$)}. Then, there exist $\alpha >0$ and $(N_i)\in \NN^\NN$ such that, for any $(n_i)\in (\NN\cup \{\infty\})^\NN$ satisfying $n_i\ge N_i$ and any $j, k\in \NN$, we have 
   \begin{equation}\label{Eq:|wj|}
   \lceil k/n_j\rceil \ge |g_j^{k}|_{W} \ge \alpha k/n_j -2,
   \end{equation}
   where $W=W_{(n_i)}$. Furthermore, if at least two elements of $(n_i)$ are not equal to $\infty$,  then $W\in Hyp_{gt}(G)$.
\end{lem}

\begin{proof}
The left inequality in (\ref{Eq:|wj|}) follows immediately from the inclusion $S(w_j,n_j)\subseteq W$. Thus, we only need to prove the right inequality. Let $r=13\delta +2M$, where $M=M(\delta, K,L)$ is provided by Lemma \ref{lem:Morse lemma},  and let
$$
N_i=\lceil\sup\{\diam(L_{g_i}\cap (aL_{g_j})^{+r})\mid a\in G,\; j\ne i\} +2M +L +30\delta\rceil.
$$
Note that $N_i<\infty $ for all $i$ by {\rm ({\bf L}$_2$)}. Further, let $\alpha = 1/(2\cdot 10^3K)$. 

Suppose, for contradiction, that $|g_j^{k}|_{W} < \alpha k/n_j -2$ for some $j,k\in \NN$; equivalently, we have $k> n_j(|g_j^k|_W+2)/\alpha$. For a word $w$ in $X^{\pm 1}$, we denote by $\| w\|$ the number of letters in $w$. Using  ({\bf L}$_1$) and the inequalities $n_j\| w_j\|\ge n_j\ge N_j\ge \max\{ L,\, 2M\}$, we obtain
\begin{equation}\label{Eq:cn}
\begin{split}
   |g_j^k|_X  & \ge (k\| w_j\| -L)/K\ge   
\Big(2\cdot 10^3 K n_j \| w_j\| \big(|g_j^k|_W +2\big) -L\Big)/K \\ & >  2\cdot 10^3 n_j \| w_j\| \big(|g_j^k|_W +1\big) \ge  10^3 \big(n_j \| w_j\| +2M\big)\big(|g_j^k|_W +1\big).
\end{split}
\end{equation}

Let $q$ (respectively, $s$) be a geodesic in $Cay(G,X)$ (respectively, $Cay(G,W)$) going from $1$ to $g_j^k$ (see Fig. \ref{Fig:Ols}). Replacing every edge $e$ of $s$ with a geodesic $t_e$ in $Cay(G,X)$ going from $e_-$ to $e_+$, we obtain a path $t$ in $Cay(G,X)$ connecting $1$ to $g_j^k$ and consisting of at most $|g_j^k|_W$ geodesic segments. Thus, we can think of $\mathcal P=qt^{-1}$ as a geodesic $(|g_j^k|_W+1)$-gon. Since $n_j\ge N_j\ge 30\delta$, and $\ell(q)=|g_j^k|_X$, the inequality (\ref{Eq:cn}) allows us to apply Lemma \ref{Lem:Ols} to $\mathcal P$ and $Q=\{ q\} $ with $c=n_j\| w_j\|+2M$. We conclude that there are subsegments $u$ of $q$ and $v$ of some $t_e$ (where $e$ is an edge of $s$) such that  $\d_{Hau}(u_\pm, v_\pm)\le 13\delta$ and 
\begin{equation}\label{Eq:duv}
    \min\{\ell(u), \ell(v)\} >c = n_j\| w_j\| +2M.
\end{equation}

\begin{figure}
 \centering{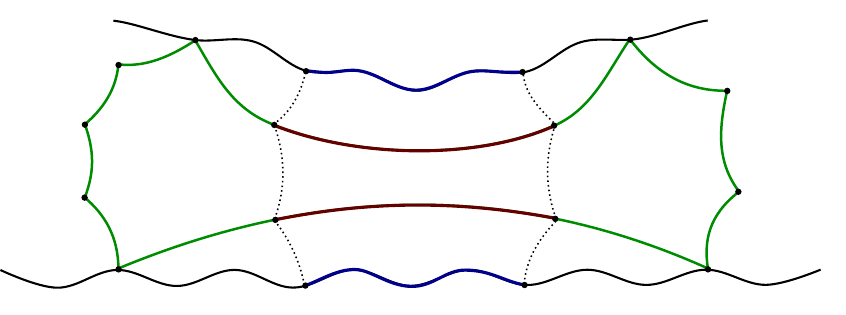}\\
  \caption{Proof of Lemma \ref{lem:prelim for embedding K_sigma}.}\label{Fig:Ols}
\end{figure}

In particular, (\ref{Eq:duv}) implies that $\ell (t_e)>n_j\| w_j\|$. Hence, the label of $e$ cannot belong to $X\cup S(w_j, n_j)$. Thus, $e$ is labeled by a word $w'\in S(w_{j'}, n_{j'})$ for some $j'\ne j$. Using Lemma \ref{lem:Morse lemma}, we conclude that, for some $a\in G$, $L_{g_j}$ and $aL_{g_{j'}}$ contain subsegments $u_0$ and $v_0$, respectively, such that $\d_{Hau}((u_0)_\pm, (v_0)_{\pm})\le 13\delta +2M =r$ and 
\begin{equation*}
    \begin{split}
        \min\{\d_X(u_{0-}, u_{0+}), \d_X(v_{0-}, v_{0+})\} &  > n_j\|w_j\|\ge N_j\\ &\ge \sup\{\diam(L_{g_j}\cap (aL_{g_{j'}})^{+r})\mid a\in G,\; j'\ne j\},
    \end{split}
\end{equation*}  
which contradicts ({\bf L}$_2$).

We now prove the second claim of the lemma. We already know that $W\in Hyp(G)$ by Lemma \ref{lem:Cay(G,W) hyperbolic}, so we only have to show that $Cay(G,W)$ is of general type. Without loss of generality, we can assume that $n_1, n_2<\infty$. By (\ref{Eq:|wj|}), we have $g_1,g_2\in \L(G\act Cay(G,W))$. We will show that $g_1$ and $g_2$ are independent with respect to the action on $Cay(G,W)$.

Arguing by contradiction, suppose that $\{ g_1^-, g_1^+\} \cap \{ g_2^-, g_2^+\} \ne \emptyset$ in $\partial Cay(G,W)$. Replacing elements $g_i$ with their inverses if necessary, we can assume that $g_1^+= g_2^+$. It is well-known (and easy to see using Lemma \ref{lem:Morse lemma}) that this equality is equivalent to the existence of a constant $B\ge 0$ satisfying the following: for any $k\in \NN$, there is $m(k)\in \NN$ such that $\d_W(g_1^k, g_2^{m(k)})\le B$.  To derive a contradiction, we consider another generating set 
$$
U=W_{(n_1, \infty, \infty, \ldots)}=X\cup S(g_1,n_1)\cup\left(\bigcup_{i=2}^\infty S(g_i, \infty)\right) 
$$
Note that $W\cup \la g_2\ra \subseteq U$. Hence, for any $k\in \NN$, we have
$$
|g_1^k|_{U}\le \d_U(g_1^k, g_2^{m(k)})+|g_2^{m(k)}|_U \le  \d_W(g_1^k, g_2^{m(k)})+1 \le B+1.
$$
For sufficiently large $k$, this contradicts the right inequality in (\ref{Eq:|wj|}) applied to $U$.
\end{proof}

\paragraph{5.3. Embedding $Q_{K_\sigma}$ in $Hyp_{gt}(G)$.} We refer the reader to Definition \ref{Def:EKs} for the definition of $\Pi$ and $Q_{K_\sigma}$. 

\begin{prop}\label{Prop:EmbPi}
Suppose that a group $G$ admits a general type action on a hyperbolic space $S$ and contains two non-equivalent loxodromic elements. Then there is a Borel map $f\colon \Pi\to Hyp_{gt}(G)$ satisfying the following conditions.
\begin{enumerate}
    \item[(a)] For any $X\in f(\Pi)$, we have $G\act Cay(G,X)\lA G\act S$.
    \item[(b)] For any $r,s\in \Pi$, we have $f(r)\preccurlyeq f(s)$ if and only if $r\,Q_{K_\sigma}s$.
\end{enumerate}  
\end{prop}

\begin{proof} 
By Lemma \ref{Lem:BF}, we can find independent elements $g, h\in \L(G\act S)$ such that $g\not\sGS h$. By Proposition \ref{Prop:Cob}, we can assume that $G\act S$ is cobounded without loss of generality. Lemma \ref{Lem:MS} yields a generating set $X \in Hyp_{gt}(G)$ such that the actions $G\act Cay(G,X)$ and $G\act S$ are equivalent. In particular, $G\act Cay(G,X)$ is of general type, and $g$, $h$ are loxodromic and non-equivalent with respect to this action. Let $(g_i)$ be a sequence of elements satisfying properties (a) and (b) in Proposition \ref{Prop:BF} applied to $G\act Cay(G,X)$ and $o=1$.  Thus, the sequence $(g_i)$ satisfies conditions ({\bf L$_1$}) and ({\bf L$_2$})

Let $\alpha $ and $(N_i)$ be as in Lemma \ref{lem:prelim for embedding K_sigma}. We define a map $f \colon \Pi \to Gen(G)$ by the rule 
    \begin{equation}\label{Eq:Psi}
    f(r)= W_{(2^{i-r(i)}N_i)} \;\;\; \forall \, r\in \Pi, 
    \end{equation}
where $W_{(2^{i-r(i)}N_i)}$ is given by (\ref{Eq:Wni}) for the sequence $n_i=2^{i-r(i)}N_i$. It is straightforward to see that this map is Borel. By Lemma \ref{lem:prelim for embedding K_sigma}, we have $f(r)\in Hyp_{gt}(G)$ for all $r\in \Pi$,  and Lemma \ref{lem:Cay(G,W) hyperbolic} implies (a). 
   
To prove (b), suppose first that $r, s\in \Pi$ and $r\, Q_{K_\sigma}s$. Let $k = \sup_{i\in\NN}(r(i)-s(i))$. Then $2^{i-s(i)}\le 2^{i-r(i)+k}$. Combining this inequality with (\ref{Eq:Wni}) and (\ref{Eq:Psi}), it is easy to see that $|w|_{f(r)}\le 2^k$ for all $w\in f(s)$. This implies that $f(r)\preccurlyeq f(s)$. 

Next, suppose that $f(r)\preccurlyeq f (s)$; i.e., there is $C\ge 0$ such that 
\begin{equation}\label{Eq:gleC}
|g|_{f(r)} \le C\;\;\; \forall\, g\in f(s). 
\end{equation}
We want to show that $r\, Q_{K_\sigma}s$. Suppose, for contradiction, that $\sup_{i\in\NN}(r(i)-s(i))=\infty$. Then there exists $i\in\NN$ such that $\alpha 2^{r(i)-s(i)} -2>C$. Using \eqref{Eq:|wj|}, we obtain
$$
\left|g_i^{2^{i-s(i)}N_i}\right|_{f(r)} \ge \alpha\cdot  \frac{ 2^{i-s(i)}N_i}{2^{i-r(i)}N_i} -2 = \alpha 2^{r(i)-s(i)} -2>C,
$$
which contradicts (\ref{Eq:gleC}) as $g_i^{2^{i-s(i)}N_i}\in f(s)$ by definition. 
\end{proof}

%%%%%%%%%%%%%%%%%%%%%%%%%%%%%%%%%%%%%%%%%%%%%%%%%%%%%%%%%%%%%%%%%

\section{Isotropic weakly hyperbolic groups.}\label{Sec:IWHG}

%%%%%%%%%%%%%%%%%%%%%%%%%%%%%%%%%%%%%%%%%%%%%%%%%%%%%%%%%%%%%%%%%

This section is devoted to the study of general-type actions of isotropic groups on hyperbolic spaces. In particular, we show that all such actions satisfy a weak form of the isotropy condition and exhibit a coarse version of the marked spectrum rigidity phenomenon (for more details on the latter, see \cite{WXY} and references therein).

\paragraph{6.1. Translation length functions and group actions on hyperbolic spaces.}\label{Sec:TLR} In this subsection, we establish several auxiliary results that describe the relation between the translation length functions on a group $G$ computed with respect to general type actions on hyperbolic spaces that are comparable with respect to $\lA$. We begin by recalling the following.

\begin{defn}
Let $G\curvearrowright S$ be an action of a group $G$ on a metric space $S$ and let $s\in S$. The \emph{translation length} of an element $g\in G$ is defined by the formula
$$
\tau_{G\curvearrowright S}(g)=\liminf_{n\to \infty} \frac{\d_S (s,g^ns)}{n}.
$$ 
\end{defn}

By the triangle inequality, the sequence $\d_S(s, g^ns)$ is subadditive. The well-known and easy-to-prove result about subadditive sequences (sometimes referred to as Polya--Szeg\"o theorem) implies that the actual limit always exists and, moreover, 
\begin{equation}\label{Eq:Liminf}
\tau_{G\curvearrowright S}(g)=\lim_{n\to \infty}\frac{\d_S (s,g^ns)}{n}=\inf_{n\in \NN} \frac{\d_S (s,g^ns)}{n} .
\end{equation}

Further, it is easy to see that $\tau_{G\curvearrowright S}(g)$ is independent of the choice of a particular point $s$ and $\tau_{G\curvearrowright S}(g^k)=|k|\tau _{G\curvearrowright S}(g)$ for all $k\in \ZZ$. If $S$ is hyperbolic, we obviously have $\tau_{G\curvearrowright S}  (g)>0$ if and only if $g\in \L (G\curvearrowright S)$. We will frequently use these properties below without further reference.

\begin{defn}
    Given two functions $\alpha, \beta\colon G\to [0, \infty)$, we write $\alpha \preccurlyeq_{Lip} \beta $ if there exists $K\in [0, \infty)$ such that $\alpha(g)\le K\beta(g)$ for all $g\in G$. Further, we say that $\alpha$ and $\beta$ are \emph{Lipschitz equivalent} and write $\alpha \sim_{Lip} \beta $ if $\alpha \preccurlyeq_{Lip}  \beta $ and $\beta \preccurlyeq_{Lip} \alpha$.

\end{defn}
Our first goal is to prove the following.  

\begin{prop}\label{Prop:tau}
Let $G\curvearrowright S$ and $G\curvearrowright T$ be general type actions of a group $G$ on hyperbolic spaces $S$ and $T$. We have $G\curvearrowright S\lA G\curvearrowright T$ if and only if $\tau_{G\curvearrowright S} \preccurlyeq_{Lip} \tau_{G\curvearrowright T}$. In particular,  $G\curvearrowright S\sim  G\curvearrowright T$ if and only if $\tau_{G\curvearrowright S}$ and  $\tau_{G\curvearrowright T}$ are Lipschitz equivalent.
\end{prop}

\begin{rem}
    For cobounded actions, a similar result was obtained in \cite[Theorem~2.14]{ABO}. However, as noted in \cite[Example 6.14]{ABO}, the direct counterpart of that theorem does not hold for non-cobounded actions. This is precisely the reason why we work with Lipschitz equivalence here, which is stronger than the equivalence relation considered in \cite[Theorem 2.14]{ABO}.
\end{rem}

We will need the following lemma proved in {\cite[Lemma 6.12]{ABO}} (an analogous result can be found in \cite[Chapter~5]{GdlH}).

\begin{lem}\label{Lem:x0xn}
Let $x_0, x_1, \ldots, x_n$ be a sequence of points in a $\delta$--hyperbolic space $S$. Suppose that there exists $C\ge 0$ such that $(x_{j-1}, x_{j+1})_{x_j} \le C$ for all $1\le j \le n-1$ and $\d_S (x_{j-1}, x_{j})> 2C+16\delta$ for all $1\le j\le n$. Then
$$
\d_S (x_0, x_n) \ge \sum\limits_{j=1}^n \d_S (x_{j-1}, x_{j}) - 2(n-1)(C+8\delta).
$$
\end{lem}

\begin{proof}[Proof of Proposition \ref{Prop:tau}]
Suppose first that  $G\curvearrowright S\lA G\curvearrowright T$. By Definition \ref{def-poset}, this means that there exist a constant $K$ and points $s\in S$, $t\in T$ such that $$\d_S (s,fs) \le K\d_T(t, ft)+K$$ for all $f\in G$. Let $g$ be an arbitrary element of $G$. Applying the previous inequality to powers of $g$, we obtain
$$
\tau_{G\curvearrowright S}(g)=\lim\limits_{n\to \infty} \frac{\d_S(s, g^ns)}{n} \le \lim\limits_{n\to \infty} \frac{Kd_T(t, g^nt)+K}{n}= K\lim\limits_{n\to \infty} \frac{\d_T(t, g^nt)}{n}=K\tau_{G\curvearrowright T}(g).
$$
Thus, $\tau_{G\curvearrowright S} \preccurlyeq_{Lip} \tau_{G\curvearrowright T}$.

To prove the converse implication, we argue by contradiction. Suppose that there exists $K\ge 0$ such that
\begin{equation}\label{Eq:tau}
\tau_{G\curvearrowright S} (g)\le K\tau_{G\curvearrowright T}(g)\;\;\;\;\; \forall\, g\in G,
\end{equation}
but $G\curvearrowright S\not \lA G\curvearrowright T$. Let us fix arbitrary points $s\in S$, $t\in T$. By Definition \ref{def-poset}, $G\curvearrowright S\not \lA  G\curvearrowright T$ means that for every $i\in \NN$ there exists an element  $f_i\in G$ such that
\begin{equation}\label{Eq:dfi}
\d_S (s,f_is) \ge i\d_T(t, f_it)+i.
\end{equation}
Note that inequality (\ref{Eq:dfi}) remains true if we pass to a subsequence $f^\prime_i=f_{k(i)}$ of $(f_i)$ for some $k(1)<k(2)<\ldots $; that is, 
$$
\d_S(s, f^\prime_i s)\ge k(i)\d_T(t, f^\prime _it)+k(i)\ge  i\d_T(t, f^\prime _it)+i.
$$
Thus, passing to a subsequence if necessary, we can additionally ensure that each of the sequences $(f_is)_{i \in \NN}$, $(f_i^{-1}s)_{i \in \NN}$ either converges to infinity or does not contain any subsequence convergent to infinite. In particular, the closure of the set $\{ f_is, f_i^{-1}s\mid i\in \NN\}$ in $S\cup \partial S$ contains at most $2$ points of $\partial S$.

Since the action of the group $G$ on $S$ is of general type, Lemma \ref{dense} guarantees the existence of an element $z\in \L(G\curvearrowright S)$ such that the closure of the set $\{ f_is, f_i^{-1}s\mid i\in \NN\}$ in $S\cup \partial S$ does not intersect $\{ z^+, z^-\}$. By the definition of the topology on $\partial S$, this implies that there is a constant $C$ such that
\begin{equation}\label{Eq:supz}
(z^ns, f_is)_s\le C\;\;\; {\rm and}\;\;\; (z^{-n}s, f_i^{-1}s)_s\le C
\end{equation}
for all $n,i\in \mathbb N$.

Suppose that $S$ is $\delta$-hyperbolic. Again passing to a subsequence of $(f_is)$ if necessary, we can assume that
\begin{equation}\label{Eq:fi}
\d_S(f_is,s) > 2C+16\delta
\end{equation}
for all $i\in \mathbb N$. Further since $z$ is loxodromic, there exists $k\in \mathbb N$ such that
\begin{equation}\label{Eq:zk}
\d_S(z^ks,s) > 2C+16\delta.
\end{equation}

For every $i\in \mathbb N$, we define $h_i=f_i^{-1} z^k$ and consider the sequence of points
\begin{equation}\label{Eq:xj}
x_0=s, \;\;\; x_1=f_i^{-1}s,\;\;\; x_2=h_is,\;\;\; x_3=h_if_i^{-1}s, \;\;\; x_4 = h_i^2s, \;\;\; \ldots
\end{equation}
For any odd $j\in \NN$, the obvious equality $(ga,gb)_{gc}=(a,b)_c$ for all $a,b,c\in S$ and all $g\in G$ implies
$$
(x_{j-1}, x_{j+1})_{x_j}= (x_0, x_2)_{x_1} = (s, h_is)_{f_i^{-1}s} = (f_is, z^ks)_s.
$$
Similarly, if $j$ is even, we obtain
$$
(x_{j-1}, x_{j+1})_{x_j}= (x_1, x_3)_{x_2} = (f_i^{-1}s, h_if_i^{-1}s)_{h_is} = (z^{-k}s, f_i^{-1}s)_s.
$$
In either case, we have $(x_{j-1}, x_{j+1})_{x_j}\le C$ by (\ref{Eq:supz}). This inequality, (\ref{Eq:fi}), and (\ref{Eq:zk}) allow us to apply Lemma \ref{Lem:x0xn} to any initial subsequence of the sequence (\ref{Eq:xj}).
For every $m\in \mathbb N$, we obtain
\begin{equation*}
    \begin{split}
\d_S(h_i^ms,s) & = \d_S(x_0, x_{2m})\ge  \sum\limits_{j=1}^{2m}  \d_S (x_{i-1}, x_{i}) - 2(2m-1)(C+8\delta) \\ 
&\ge  m\big(\d_S(f_i^{-1}s,s)+\d_S(z^ks,s)\big) - 2(2m-1)(C+8\delta).
\end{split}
\end{equation*}

Consequently,
$$
\tau_{G\curvearrowright S}(h_i)=\lim_{m\to \infty} \frac{\d_S(h_i^ms,s)}{m} \ge \d_S(f_i^{-1}s,s)-C_1,
$$
where $C_1$ is a constant independent of $i$. On the other hand, we have
$$
\tau _{G\curvearrowright T}(h_i) \le \d_T(h_it, t) \le \d_T(f_i^{-1}z^kt, f_i^{-1} t)+ \d_T(f_i^{-1}t,t)=\d_T(f_i^{-1}t,t)+ C_2,
$$
where $C_2=\d_T(z^kt, t)$ is also independent of $i$. Combining with (\ref{Eq:tau}) and (\ref{Eq:dfi}), we obtain 
\begin{equation*}
\begin{split}
i\d_T(f_it,t)+i & \le \d_S(f_is,s) \le \tau_{G\curvearrowright S}(h_i)+C_1 \le  K\tau_{G\curvearrowright T}(h_i)+C_1\\
& \le  K\d_T(f_it,t) + KC_2+C_1
\end{split}
\end{equation*}
for all $i\in \NN$, which is a clear nonsense. This contradiction completes the proof.
\end{proof}

\begin{lem}\label{Lem:supgh}
Let $G$ be a group acting on a hyperbolic space $T$ and let $g,h\in \L(G\act T)$. Suppose that for some point $t\in T$, there exist strictly increasing sequences $(m_i)$, $(n_i)$ of natural numbers such that $$\sup_{i\in \NN} \big|\d_T(t, g^{m_i}t) - \d_T(t, h^{n_i}t)\big|<\infty.$$
Then $\lim\limits_{i\to \infty} m_i/n_i$ exists and equals  $\tau_{G\curvearrowright T} (h)/\tau_{G\curvearrowright T} (g).$
\end{lem}

\begin{proof} The assumptions of the lemma imply that $\lim\limits_{i\to \infty}\frac{\d_T(t, h^{n_i}t)}{\d_T(t, g^{m_i}t)}=1$. Therefore, 
$$
\lim_{i\to \infty}\frac{m_i}{n_i}=
\lim_{i\to \infty}\frac{\d_T(t, h^{n_i}t)}{\d_T(t, g^{m_i}t)} \cdot \lim_{i\to \infty}\frac{m_i}{n_i}= 
\frac{\lim\limits_{i\to \infty} \d_T(t, h^{n_i}t)/n_i}{\lim\limits_{i\to \infty} \d_T(t, g^{m_i}t)/m_i} =  \frac{\tau_{G\curvearrowright T} (h)}{\tau_{G\curvearrowright T} (g)}. 
$$\end{proof}

\begin{defn}
Let $G\curvearrowright R\lA G\curvearrowright S$ be two actions of a group of $G$ on metric spaces $R$ and $S$. The associated \textit{translation length compression function} $Comp^{G\act S}_{G\act R}$ is a map from the set $ \{ g\in G\mid \tau_{G\act S}(g)>0\}$ to $[0, \infty)$ defined by the formula
$$
Comp^{G\act S}_{G\act R}(g) =\frac{\tau_{G\act R}(g)}{\tau_{G\act S}(g)} \;\;\;\;\; \forall\, g\in G.
$$
\end{defn}

Note that if $S$ is hyperbolic, the domain of $Comp^{G\act S}_{G\act R}$ coincides with $\L(G\act S)$.

\begin{lem}\label{Lem:tlcf}
  Let $G\curvearrowright R\lA G\curvearrowright S$ be two actions of a group of $G$ on hyperbolic spaces. For any $g,h\in \L(G\act S)$ such that $g\sim_{G\act S} h$, we have $Comp^{G\act S}_{G\act R}(g)=Comp^{G\act S}_{G\act R}(h)$. 
\end{lem}

\begin{proof}
We fix some points $r\in R$, $s\in S$ and consider arbitrary $g,h\in \L(G\curvearrowright S)$ such that $g\sim_{G\act S} h$. By Lemma \ref{Lem:loxeq}, there exist $\e>0$ and strictly increasing sequences $(m_i)$, $(n_i)$ of natural numbers such that, for any $i\in \mathbb N$, we can find $a\in G$ satisfying the inequality 
\begin{equation}\label{Eq:dass}
\max\{ \d_S(as, s), \, \d_S(ag^{m_i}s, h^{n_i}s)\} \le \e.
\end{equation}
In particular, we have 
$$
|\d_S(g^{m_i}s, s) - \d_S(h^{n_i}s,s)|= |\d_S(ag^{m_i}s, as) - \d_S(h^{n_i}s,s)| \le 2\e
$$ 
for all $i\in \NN$.
By Lemma \ref{Lem:supgh}, this implies
\begin{equation}\label{Eq:h/gR}
\lim\limits_{i\to \infty} \frac{m_i}{n_i}= \frac{\tau_{G\curvearrowright S}(h)}{\tau_{G\curvearrowright S}(g)}\ne 0
\end{equation}

Combining (\ref{Eq:dass}) with the inequality $G\curvearrowright R\lA G\curvearrowright S$, we obtain that there exists a constant $\e^\prime$ such that
$\max\{ \d_R(ar, r), \, \d_R(ag^{m_i}r, h^{n_i}r)\} \le \e^\prime$ for all $i$. As above, we obtain
$|\d_R(g^{m_i}r, r) - \d_R(h^{n_i}r,r)|\le 2\e^\prime.$
In particular, $g$ and $h$ are either simultaneously loxodromic or simultaneously non-loxodromic with respect to the action on $R$. In the latter case, we have $Comp^{G\act S}_{G\act R}(g)=Comp^{G\act S}_{G\act R}(h)=0$. In the former case, we apply Lemma \ref{Lem:supgh} again and obtain
$$
\frac{\tau_{G\curvearrowright S}(h)}{\tau_{G\curvearrowright S}(g)}=\lim\limits_{i\to \infty} \frac{m_i}{n_i}= \frac{\tau_{G\curvearrowright R}(h)}{\tau_{G\curvearrowright R}(g)}, 
$$
which is equivalent to $Comp^{G\act S}_{G\act R}(g)=Comp^{G\act S}_{G\act R}(h)$.
\end{proof}

\paragraph{6.2. Weakly isotropic group actions on hyperbolic spaces.}\label{Sec:EqDefI} 
By Definition \ref{Def:Iso}, isotropic group actions are necessarily cobounded. It is sometimes convenient to consider a weaker yet closely related condition that may hold for non-cobounded actions.

\begin{defn}
    An action of a group $G$ on a metric space is \textit{weakly isotropic} if its restriction to some orbit is isotropic.
\end{defn}

A word of warning: neither isotropy nor weak isotropy of group actions is invariant under equivalence. Furthermore, even for cobounded actions, weak isotropy can be strictly weaker than isotropy, which underscores the non-triviality of the equivalence between conditions (a)–-(g) and (h) in Theorem \ref{Thm:Iso} below. We outline an example illustrating the latter phenomenon; since this example is not used in the proof of any of our results, we leave the verification of details to the reader.

\begin{ex} Let $G=\bigoplus_{i=1}^\infty \la a_i\mid a_i^2=1\ra$ and let $\Gamma $ denote the Cayley graph of $G$ with respect to the generating set $\{ a_1, a_2, \ldots\}$. We pick an infinite sequence of linearly independent (over $\QQ$) numbers $\ell_i\in [1,2]$ and endow $\Gamma $ with a metric $\d_\Gamma$ by identifying every edge labeled by $a_i$ with $[0, \ell_i]$. The action of $G$ on itself by left multiplication obviously extends to a cobounded (isometric) action of $G$ on $\Gamma$. Using linear independence of $\{\ell_i\mid i\in \NN\}$, it is easy to show that for any $a,b,c,d\in G$, the equality $\d_\Gamma (a,b)=\d_\Gamma(c,d)$ implies $a^{-1}b=c^{-1}d$, which in turn implies that the restriction of the action $G\act \Gamma$ to the set of vertices is isotropic. On the other hand, it is not difficult to show that the action $G\act \Gamma$ is not isotropic itself.
\end{ex}

%We briefly sketch the proof. For any vertices  $u$, $v$, $u'$, $v'$ of $\Gamma$ such that $\d_\Gamma (u,v)=\d_\Gamma(u',v')$, the linear independence condition for numbers $\ell_i$ easily implies that $u^{-1}v=(u')^{-1}v'$ in $G$. Hence, the element $g=u^{-1}v$ maps the pair $(u,v)$ to $(u',v')$. Thus, the restriction of the action $G\act \Gamma$ to the set of vertices is isotropic and, therefore, $G\act \Gamma$ is weakly isotropic. 

%On the other hand, suppose that the action $G\act \Gamma $ is isotropic with the isotropy constant $D\ge 0$. For a natural number $n$, we let $e_n=a_2a_4\cdots a_{2n}$ and $o_n=a_1a_3\cdots a_{2n-1}$. We can find arbitrarily large $m,n\in \NN$ such that 
%$$
%|\d_\Gamma (e_m,1)- \d_\Gamma (e_n, 1)| = \left|\sum_{i=1}^m \ell_{2i}  - \sum_{i=1}^n \ell_{2i-1}\right|\le 1.
%$$
%This allows us to find a point $x\in \Gamma$ such that  $\d_\Gamma(x, e_m)\le 1$ and $\d_X (1,x)=\d_X(1, o_n)$. By our assumption, there is $g\in G$ such that $\max\{ \d_\Gamma (g,1), \, \d_\Gamma (gx, o_n)\}\le D$, and hence $\d_\Gamma (ge_m, o_n)\}\le D+1$. The former inequality implies $|g|_A\le D$, while the latter inequality yields $$|g|_A\ge |e_mo_n|_A - |ge_mo_m|_A\ge m+n - D-1$$ (note that every element has order $2$ in $G$). This leads to a contradiction if $m$ and $n$ are sufficiently large.  

Below, we provide several equivalent characterizations of weakly isotropic, general-type group actions on hyperbolic spaces playing a fundamental role in the proof of the main results of our paper. Recall that we endow  $\partial S\times \partial S$ with the product topology and denote by $\Lambda_S(G)\otimes\Lambda_S(G)$ its subspace $\{ (x,y)\mid x,y\in \Lambda_S(G),\, x\ne y\}$. For the definition of a minimal element in a quasi-ordered set, we refer to the beginning of Subsection \hyperref[Sec:CompGA]{3.1}.

\begin{thm}\label{Thm:Iso}
For any general type action $G\curvearrowright S$ of a group $G$ on a hyperbolic space $S$, the following conditions are equivalent.
\begin{enumerate}
\item[(a)] The action $G\curvearrowright S$ is weakly isotropic.
\item[(b)] Every element of $\L(G\curvearrowright S)$ is equivalent to its inverse.
\item[(c)] Any two elements of $\L(G\curvearrowright S)$ are equivalent.
\item[(d)] For any $g\in \L(G\act S)$, we have $\overline{Orb}(g^-,g^+)=\Lambda_S(G)\otimes \Lambda_S(G)$.
\item[(e)] The induced action of $G$ on $\Lambda_S(G)\otimes \Lambda_S(G)$ is minimal.
\item[(f)] $G\curvearrowright S$ is minimal among all general type $G$-actions on hyperbolic spaces.
\item[(g)]  There exist a minimal element $X\in Hyp_{gt}(G)$ and a $G$-equivariant quasi-isometric embedding $(G,\d_X)\to (S,\d_S)$ (equivalently, $G\act Cay(G, X)\eA G\act S$).
\end{enumerate}
Furthermore, if $G\act S$ is cobounded, conditions (a)--(g) are equivalent to the following.
\begin{enumerate}
    \item[(h)] The action $G\act S$ is isotropic. 
\end{enumerate}
\end{thm}

To prove the theorem, we will need a couple of lemmas. The first one is an elementary fact in hyperbolic geometry. For any two points $u$, $v$ in a geodesic metric space, we denote by $[u,v]$  a geodesic connecting $u$ to $v$ (which may not be unique).

\begin{lem}\label{Lem:4pts}
Let $S$ be a $\delta$-hyperbolic space.  Let also $x,y,u,v\in S$ and let $t$ be a point on a geodesic $[x,y]$. Suppose that there exists a constant $K$ such that
\begin{equation}\label{Eq:K}
(x,u)_t\ge \d_S(t,x) -K> 3\delta {\rm \;\;\;\;\; and \;\;\;\;\;} (y,v)_t\ge \d_S(t,y) -K> 3\delta.
\end{equation}
Then, for any geodesic $[u, v]$, there exist $c,d\in [u,v]$ such that $d$ belongs to the segment $[c,v]$ of $[u,v]$ and we have 
$$
\max\{ \d_S(x, c), \, \d_S(y, d)\} \le K+2\delta.
$$
\end{lem}
\begin{proof}
We fix any geodesics $[u,v]$, $[t,u]$, and $[t,v]$ (see Fig. \ref{Fig:GroPro}). Let $a$ be the point on the segment $[x,t]$ of $[x,y]$ such that $\d_S(a,x)=K$. By (\ref{Eq:K}), we have $\d_S(t,a)=\d_S(t,x)-K\le (x,u)_t $. Therefore, by ({\bf H$_2$}), there exists $b\in [t,u]$ such that $\d(a,b)\le \delta$. Further, using ({\bf H$_3$}) twice and then (\ref{Eq:K}) we obtain
$$
0=(x,y)_t \ge \min\{ (x,u)_t, (u,v)_t, (y,v)_t\} -2 \delta \ge \min\{ (u,v)_t, 3\delta\} - 2\delta.
$$
Hence, $(u,v)_t\le 2\delta$, and consequently $(v,t)_u=\d_S(t,u)-(u,v)_t \ge \d_S(t,u)-2\delta$.
On the other hand, using (\ref{Eq:K}), we obtain
\begin{equation*} 
\begin{split}
\d_S(b,u) &=\d_S(t,u)-\d_S(t,b) \le \d_S(t,u)-(\d_S(t,x)-\d_S(x,a)-\d_S(a,b))\le\\ & \d_S(t,u)- (\d_S(t,x)-K-\delta) < \d_S(t,u)-2\delta \le (v,t)_u.
\end{split}
\end{equation*}
Therefore, ({\bf H$_2$}) yields  $c\in [u,v]$ such that $\d_S(b,c)\le \delta$ and 
\begin{equation}\label{Eq:duc}
\d_S(u,c) =\d_S(u,b)\le (v,t)_u
\end{equation}
In particular, we have $\d_S(x, c)\le \d_S(x,a)+ \d_S(a,b)+\d_S(b,c)\le K+2\delta.
$
\begin{figure}
  \centering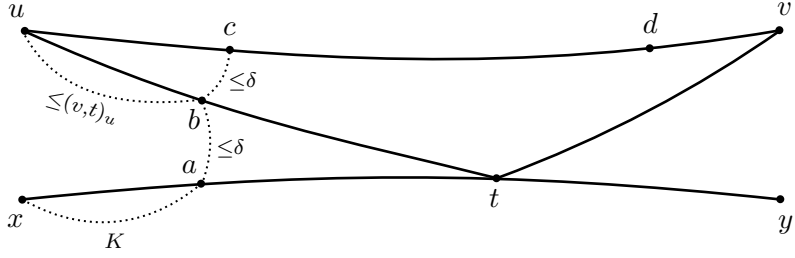
  \caption{The proof of Lemma \ref{Lem:4pts}}\label{Fig:GroPro}
\end{figure}
Similarly, we choose $d\in [u,v]$ such that $\d_S(y, d)\le K+2\delta$ and $\d_S(v,d) \le (u,t)_v$. Combining the latter inequality with (\ref{Eq:duc}), we obtain
$$
\d_S(u,c) \le (v,t)_u = \d_S(u,v) - (u,t)_v  \le \d_S(u,v)-\d_S(v,d) =\d_S(u,d).
$$
Therefore, $d$ belongs to the segment $[c,v]$ of $[u,v]$.
\end{proof}

\begin{lem}\label{Lem:xiyi}
Let $G$ be a group acting on a hyperbolic space $S$. Suppose that $\L(G\act S)\ne \emptyset$. Then for any $s\in S$, there exists a constant $L$ such that the following holds. For any $x\in Gs$ and any $t\in S$, there exists a sequence $(x_i)$ of elements of $Gs$ converging to infinity such that for every $i$,
$$
(x,x_i)_t\ge \d_S(t,x)-L. 
$$
\end{lem}

\begin{proof}
    Let $x=as$ for some $a\in G$. Let $g\in \L(G\act S)$. Since $g^-\ne g^+$, there is $C\in \RR$ such that $(g^{-n}s, g^ns)_s\le C$ for all $n\in \NN$. Assume that $S$ is $\delta$-hyperbolic. By ({\bf H$_3$}), for every $n\in \NN$, we have 
    $$
    C\ge (g^{-n}s, g^ns)_s=(ag^{-n}s, ag^ns)_{x}\ge \min\{ (ag^{-n}s, t)_{x},\, (ag^ns, t)_x\}-\delta.
    $$
    Therefore, we have either $(ag^{-n}s, t)_{x}\le C+\delta$ or $(ag^{n}s, t)_{x}\le C+\delta$. One of the two inequalities must hold infinitely many times; without loss of generality, we can assume that there exists an infinite sequence of positive integers $(n_i)$ such that 
    \begin{equation}\label{Eq:agni}
        (ag^{n_i}s, t)_{x}\le C+\delta
    \end{equation}
for all $i\in \NN$. Let $x_i=ag^{n_i}s$. Clearly, $(x_i)$ converges to $ag^+$. Combining (\ref{Eq:agni}) with the obvious equality $(x_i,t)_x+(x,x_i)_t=\d_S(t,x)$, we obtain $(x, x_i)_t\ge \d_S(t,x) -C-\delta$. Letting $L=C+\delta$ finishes the proof.    
\end{proof}

We are now ready to prove the main result of this subsection.

\begin{proof}[Proof of Theorem \ref{Thm:Iso}]
We will prove the following implications.
$$
\begin{tikzcd}[column sep=small]
& (a) \arrow[rr, Leftarrow]\arrow[ld, Rightarrow] && (e) \arrow[rd, Leftarrow]&\\
 (b) \arrow[rr, Rightarrow] && (c) \arrow[rd, Rightarrow] \arrow[ld, Leftarrow]\arrow[rr, Rightarrow] && (d) \\
& (g) \arrow[rr, Leftarrow] && (f) &
\end{tikzcd}
$$
Moreover, in the course of proving (e) $\Longrightarrow$ (a), we will establish (e) $\Longrightarrow$ (h) whenever the action is cobounded. Since the implication (h) $\Longrightarrow$ (a) is obvious for any action, this will complete the proof of the theorem.
\smallskip

\emph{(a) $\Longrightarrow$ (b).} 
Let $g\in \L(G\act S)$. We fix $t\in S$ such that $G\act Gt$ is isotropic and let $D$ be the corresponding isotropy constant. For every $n\in \NN$, we have $\d_S(t, g^{-n}t)=\d_S(t, g^nt)$. Hence, there exists $a\in G$ (depending on $n$) such that $\max\{\d_S(at,t), \d_S(ag^nt, g^{-n}t)\}\le D$. Consequently, for any $s\in S$, we have 
$\max\{\d_S(as,s), \d_S(ag^ns, g^{-n}s)\}\le D+2\d_S(s,t)$. By Lemma \ref{Lem:loxeq}, this implies that $g\sim_{G\act S} g^{-1}$. 

\smallskip

\textit{(b) $\Longrightarrow$ (c)}. The contrapositive of this implication is established in Proposition \ref{Prop:BF}.

\smallskip

\textit{(c) $\Longrightarrow$ (d).} Fix any $g\in \L(G\act S)$. By Lemma \ref{Lem:HM} (b), for any $h\in \L(G\act S)$, we have $\overline{Orb}(g^-,g^+)=\overline{Orb}(h^-,h^+)$; in particular, $(h^-, h^+)\in \overline{Orb}(g^-,g^+)$. Using Lemma \ref{dense}, we obtain $$\Lambda_S(G)\otimes \Lambda_S(G) \subseteq \overline{ \{ (h^-, h^+)\mid h\in \L(G\act S)\}} \subseteq \overline{Orb}(g^-,g^+).$$

\textit{(d) $\Longrightarrow$ (e).} This implication follows immediately from Lemma \ref{Lem:Orb}.

\smallskip

\textit{(e) $\Longrightarrow$ (a)}. Fix some $s\in S$.  Assuming (e), we will show that  

\begin{enumerate}
    \item[(+)] \textit{for any $\sigma \ge 0$, the action of $G$ on $(Gs)^{+\sigma}$ is isotropic.}
\end{enumerate} 
Specifying $\sigma =0$, we obtain the desired implication (e) $\Longrightarrow$ (a), while for cobounded actions, (+) establishes (e) $\Longrightarrow$ (h). Informally, the idea is to show that every pair of points in $(Gs)^{+\sigma}$ can be moved by an element of $G$ so that the resulting pair is close to the orbit $\langle h\rangle s$ for a fixed $h\in \L(G\act S)$. Then, an additional ``sliding along"  $\langle h\rangle s$ allows us to make any two equidistant pairs of points close to each other.  

\begin{figure}
\centering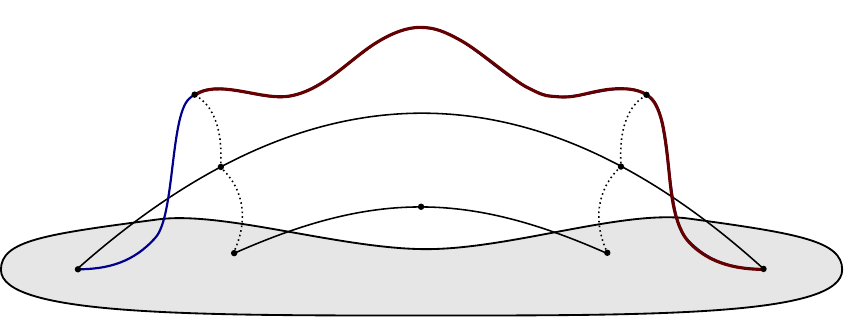
  \caption{Finding the points $h^ks$ and $h^\ell s$ in the proof of (e) $\Longrightarrow$ (a).}\label{Fig:ThmIso}
\end{figure}

More precisely, we fix any $h\in \L(G\act S)$ and let $L$ be the constant provided by Lemma~\ref{Lem:xiyi}. Consider any points $x_0,y_0, x^\prime_0, y^\prime_0\in (Gs)^{+\sigma}$ such that $\d_S(x_0,y_0)=\d_S(x^\prime_0, y^\prime_0)$. Clearly, there are $x,y,x^\prime, y^\prime \in Gs$ such that
\begin{equation}\label{Eq:xyx'y'}
   \max\{ \d_S(x, x_0),\, \d_S(y, y_0),\, \d_S(x^\prime, x_0^\prime),\, \d_S(y^\prime, y_0^\prime),\}\le \sigma,
\end{equation} 
which implies 
\begin{equation}\label{Eq:xyr}
    |\d_S(x,y)-\d_S(x^\prime, y^\prime)|\le 4\sigma.
\end{equation} 
We denote $\d_S(x,y)$ by $r$ for brevity. Obviously, the action of $g$ on $(Gs)^{+\sigma}$ is cobounded; hence, to prove (+), it suffices to verify Definition \ref{Def:Iso} for all points $x_0,y_0, x^\prime_0, y^\prime_0\in (Gs)^{+\sigma}$ such that $\d_S(x_0,y_0)=\d_S(x^\prime_0, y^\prime_0)$ is sufficiently large. Thus,  without loss of generality, we can assume that 
$$r> 2(L+4\delta).$$

We fix some geodesic $[x,y]$ and denote by $t\in S$ the midpoint of $[x,y]$ (see Fig. \ref{Fig:ThmIso}); thus,  we have 
$\d_S(x,t)=\d_S(y,t)=r/2> L+4\delta$.
By Lemma \ref{Lem:xiyi}, there exist sequences $(x_i), (y_i)\subset Gs$ converging to infinity such that for every $i$,
\begin{equation} \label{Eq:xxit}
\min\{(x, x_i)_t,\, (y, y_i)_t\} \ge r/2- L>4\delta.
\end{equation}
Further, by (e) and parts (b) and (c) of Proposition \ref{Prop:Bord}, there exist $a\in G$ and $i\in \NN$ such that $\min\{ (x_i, ah^{-i}s)_{t},\, (y_i, ah^is)_t\}> r/2$.
Combining this inequality with (\ref{Eq:xxit}) and applying ({\bf H$_3$}), we obtain
$$
\min\{ (x, ah^{-i}s)_{t},\, (y, ah^is)_t\}> r/2 -L-\delta >3\delta.
$$
Applying Lemma \ref{Lem:4pts} to the points $x$, $y$, $ah^{-i}s$, $ah^is$ and the constant $K=L+\delta$, we obtain  $c,d\in [ah^{-i}s, ah^is]$ such that $d$ is located between $c$ and $ah^is$ on $[ah^{-i}s, ah^is]$ and
\begin{equation}\label{Eq:dxahs}
\max\{ \d_S(x, c), \, \d_S(y, d)\} \le L+3\delta.
\end{equation}

Since $h$ is loxodromic, it has a standard quasi-axis $L_h$. By Lemma \ref{lem:Morse lemma}, there exists $M\ge 0$ such that every geodesic in $S$ with endpoints on $L_h$ lies in the closed $M$-neighborhood of the corresponding segment of $L_h$. Note also that the quasi-axis $L_h$ belongs to the $\d_S(s,hs)$-neighborhood of the orbit $\la h\ra s$. Therefore, there is an integer $k\in [-i,i]$ such that $ \d_S(c, ah^{k}s) \le M+\d_S(s,hs)$. Further, using ({\bf H$_1$}) for the geodesic triangle with vertices $c$, $ah^ks$, and $ah^is$, and then applying Lemma \ref{lem:Morse lemma} to the geodesic $[ah^ks, ah^is]$ and the segment $p$ of $aL_h$ between these points, we obtain $\ell \ge k$ such that $\d_S(d, ah^{\ell}s) \le \delta + M+\d_S(s,hs)$. Summarizing, we have
$$
\max\{ \d_S(x, ah^ks),\, \d_S(y, ah^\ell s)\} \le C,
$$
where the constant $C=L+M+\d_S(s,hs)+4\delta$ is independent of $x$ and $y$. 
Letting $b=h^{-k}a^{-1}$ and $m=\ell-k\ge 0$, we can rewrite the latter inequality as 
\begin{equation}\label{Eq:xyh}
  \max\{ \d_S(bx, s), \, \d_S(by, h^{m} s) \} \le C;
\end{equation}
informally, this step can be described as sliding the segment $[x, y]$ along $L_h$ so that $x$ is ``aligned" with $s$. Repeating the same argument for $x^\prime$ and $y^\prime$, we obtain  $b^\prime\in G$ and a non-negative integer $m^\prime $ such that
\begin{equation}\label{Eq:xyh'}
    \max\{ \d_S(b^\prime x^\prime, s), \, \d_S(b^\prime y^\prime, h^{m^\prime} s)\} \le C.
\end{equation}

\begin{figure}

\hspace{1.8cm}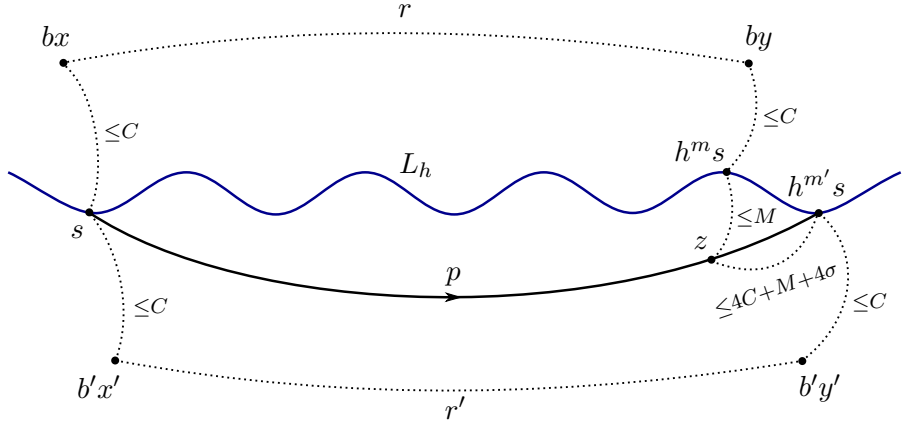
  \caption{After sliding along $L_h$.}\label{Fig:Sliding}
\end{figure}

Without loss of generality, we can assume that $m^\prime \ge m$. Let $p$ be a geodesic in $S$ connecting $s$ to $h^{m^\prime} s$ (see Fig. \ref{Fig:Sliding}). Let $r^\prime = \d_S(x^\prime, y^\prime)= \d_S(b^\prime x^\prime, b^\prime y^\prime)$; note that $|r-r^\prime|\le 4\sigma$ by (\ref{Eq:xyr}). Using Lemma \ref{lem:Morse lemma}, we can find $z\in p$ such that $\d_S(h^ms,z)\le M$. By  (\ref{Eq:xyh}), and (\ref{Eq:xyh'}), we have $|\d_S(s,z)-r|\le 2C+M$ and $|\d_S(s,h^{m^\prime} s)-r^\prime|\le 2C$. It follows that $\d_S(z,h^{m^\prime} s)=\d_S(s, h^{m^\prime} s) - \d_S(s,z) \le 4C+M+4\sigma$.
Consequently, 
$$
\d_S(by, b^\prime y^\prime) \le \d_S(by,h^ms)+\d_S(h^ms, z)+ \d_S(z, h^{m^\prime} s)+ \d_S (h^{m^\prime} s, b^\prime y^\prime) \le 6C+2M+4\sigma.
$$
Letting $g=(b^{\prime})^{-1}b$ we conclude that
\begin{equation*}
    \begin{split}
        \max\{ \d_S(gx,x^\prime),\, \d_S(gy, y^\prime)\} & \le  \max\{ \d_S(bx,b^\prime x^\prime),\, \d_S(by, b^\prime y^\prime)\} \\&\le \max\{ 2C, 6C+2M+4\sigma\}= 6C+2M+4\sigma.
    \end{split}
\end{equation*}
Finally, using (\ref{Eq:xyx'y'}), we obtain $$\max\{ \d_S(gx_0,x_0^\prime),\, \d_S(gy_0, y_0^\prime)\} \le 6C+2M+6\sigma.$$ 
Since the constant $6C+2M+6\sigma$ is independent of the choice of equidistant pairs $(x_0,y_0), (x^\prime_0, y^\prime_0)\in (Gs)^{+\sigma}\times (Gs)^{+\sigma}$, the action of $G$ on $(Gs)^{+\sigma}$ is isotropic.

\textit{(c) $\Longrightarrow$ (f).} 
Suppose that $G\act R$ is a general type $G$-action on a hyperbolic space $R$ such that $G\act R\lA G\act S$. Arguing by contradiction, assume that $G\act R\not\eA G\act S$. By Proposition \ref{Prop:tau}, there must exist a sequence of elements $(g_i)$ of $G$ such that \begin{equation}\label{Eq:limgi}
\lim\limits_{i\to \infty}\frac{\tau_{G\act R}(g_i)}{\tau_{G\act S} (g_i)}=0.
\end{equation}
Since all elements of $\L(G\act S)$ are equivalent, Lemma \ref{Lem:tlcf} and  (\ref{Eq:limgi}) imply that $Comp_{G\act R}^{G\act S}(g)=0$ for all $g\in\L(G\act S)$, i.e., no element of the group $G$ acts loxodromically on $R$. This  contradicts the assumption that $G\act R$ is of general type. 

\smallskip

\textit{(f) $\Longrightarrow$ (g).}  By Proposition \ref{Prop:Cob}, there exists a cobounded general type $G$-action on a hyperbolic space $R$ such that $G\act R\lA G\act S$. By the minimality of $G\act S$, we must have $G\act R\eA G\act S$. Furthermore, by Lemma \ref{Lem:MS}, there exist $X\in Gen()$ and a $G$-equivariant quasi-isometry $f\colon (G, \d_X)\to R$.  In particular, this implies that $X\in Hyp_{gt}(G)$ and $G\act Cay(G,X)\eA G\act R$. Therefore, $G\act Cay(G,X)\eA G\act S$, which is easily seen to be equivalent to the claim that, for any $s\in S$, the rule $g\mapsto gs$ defines a quasi-isometric embedding $(G, \d_X) \to S$. The minimality of $X$ follows from that of $G\act S$.

\smallskip

\textit{(g) $\Longrightarrow$ (c).}  This implication follows immediately from Proposition \ref{Prop:EmbPi}.
\end{proof}

\begin{proof}[Proof of Theorem \ref{Thm:IsoEquiv}]
    The action of $G$ on $Cay(G,X)$ is cobounded for any $X \in Gen(G)$. Therefore, we have $\Lambda_S(G)=\partial S$ for all $X\in Hyp(G)$ and the desired result follows from the equivalence of conditions (e), (g), and (h) in Theorem \ref{Thm:Iso}. 
\end{proof}

Recall that a weakly hyperbolic group $G$ is isotropic, if the action $G\act Cay(G,X)$ is isotropic for any $X\in Hyp_{gt}(G)$. This definition can be reformulated as follows.

\begin{prop}\label{Prop:IsoEqDef}
A weakly hyperbolic group $G$ is isotropic if and only if all general type actions of $G$ on hyperbolic spaces are weakly isotropic. 
\end{prop}

\begin{proof}
We first prove the contrapositive of the forward implication. Suppose that there exists a general type action  $G\act S$ of a group $G$ on a hyperbolic space $S$ that is not weakly isotropic. Combining Theorem  \ref{Thm:Iso} and Lemma \ref{Lem:BF}, we can find independent, non-equivalent elements $g,h \in \L(G\act S)$. Further, by Proposition \ref{Prop:Cob}, there exists a cobounded general type $G$-action on a hyperbolic space $R$ such that $g,h \in \L(G\act R)$ and $g\not\sim_{G\act R} h$.
By Lemma \ref{Lem:MS}, there is $X\in Gen(G)$ and a $G$-equivariant quasi-isometry $(G, \d_X)\to R$. Clearly, we have $X\in Hyp_{gt}(G)$, and the set $\L(G\act Cay(G,X))$ contains non-equivalent elements. The latter condition and Theorem~\ref{Thm:Iso} imply that $G\act Cay(G,X)$ is not isotropic. Therefore, $G$ is anisotropic.

To prove the backward implication, suppose that all general type actions of $G$ are weakly isotropic. Then $G\act Cay(G,X)$ is isotropic for any $X\in Hyp_{gt} (G)$ by the equivalence of conditions (a) and (h) in  Theorem \ref{Thm:Iso}. By definition, this means $G$ is isotropic. 

\end{proof}

Recall that, for a group $G$ and a function $\tau\colon G\to \RR$, we denote by $[\tau]$ the set of all non-zero multiples of $\tau$. That is, $[\tau]= \{ c\tau \mid c\in \RR\setminus\{0\}\}.$

\begin{thm}\label{Thm:LoxRigAct}
Let $G\curvearrowright S$ and $G\curvearrowright T$ be two general type actions of an isotropic group $G$ on hyperbolic spaces. We have $G\curvearrowright S \eA G\curvearrowright T$ if and only if  $[\tau_{G\curvearrowright S}]=[\tau_{G\curvearrowright T}]$.
\end{thm}

\begin{proof}
By Proposition \ref{Prop:IsoEqDef}, all general type actions of $G$ on hyperbolic spaces are weakly isotropic. Since $G\curvearrowright S \eA G\curvearrowright T$ we have $\L(G \act S) = \L(G \act T)$, and the forward implication follows from Lemma \ref{Lem:tlcf} and the equivalence of (a) and (c) in Theorem \ref{Thm:Iso}. The backward implication follows from Proposition \ref{Prop:tau}. 
\end{proof}

\paragraph{6.4. Examples.}\label{Sec:Ex}  We begin with examples of anisotropic groups. Recall that a loxodromic element $g$ of a group $G$ acting on a hyperbolic space $S$ satisfies the \textit{weak weak proper discontinuity condition} (abbreviated \textit{WWPD}) if the $G$-orbit of $(g^-,g^+)$ in $\partial S\times \partial S$ is discrete. The WWPD property was originally defined by Bestvina, Bromberg, and Fujiwara in \cite{BBF} as a weakening of the weak proper discontinuity property considered in \cite{BF}, which in turn can be thought of as a weakening of the acylindricity condition introduced by Bowditch \cite{Bow08}. The original definition of WWPD in \cite{BBF} was formulated in a different way; the characterization in terms of the boundary dynamics given above is due to Handel and Mosher \cite{HM21}. It is well-known and easy to see that every $g\in \L(G\act S)$ satisfies the WWPD condition whenever the action $G\act S$ is proper or acylindrical. For details, we refer the reader to \cite{BBF,HM21,Osi18}.

\begin{cor}\label{Cor:WWPD}
    Let $G\act S$ be a general type group action on a hyperbolic space. If $G$ contains a loxodromic element satisfying the WWPD  condition (e.g., if the action is acylindrical or proper), then $G\act S$ is not weakly isotropic; in particular, $G$ is anisotropic.
\end{cor}

\begin{proof}
Let $g\in \L(G\act S)$ be a loxodromic element satisfying the WWPD condition. Since the action $G\act S$ is of general type, the limit set $\Lambda_S (G)$ has no isolated points (see Lemma~\ref{dense}). Hence, we cannot have the equality $\Lambda_S(G)\otimes \Lambda_S(G) = \overline{Orb}_G(g^-,g^+)$. By Theorem \ref{Thm:Iso}, this implies that the action $G\act S$ is not weakly isotropic. Therefore, $G$ is anisotropic by Proposition \ref{Prop:IsoEqDef}.
\end{proof}

We mention some particular cases covered by Corollary \ref{Cor:WWPD}. For the definition and examples of acylindrically hyperbolic groups, see \cite{Osi16,Osi18}.

\begin{ex}
    Every acylindrically hyperbolic group is anisotropic. In particular, every non-virtually-cyclic hyperbolic group is anisotropic. Examples of groups satisfying the assumption of Corollary \ref{Cor:WWPD} that are not acylindrically hyperbolic can be found among big mapping class groups (see \cite{Ras}).
\end{ex}

Further, we discuss the relation between the property of being isotropic and quasi-morphisms on weakly hyperbolic groups. Recall that for any group $G$, the set of all quasi-morphisms $G\to \RR$ has a natural structure of a linear vector space over $\RR$; its quotient space by the subspace spanned by bounded maps and homomorphisms $G\to \RR$ is denoted by $\widetilde{QH}(G)$. 
The following proposition can be thought of as a generalization of the main result of \cite{IPS} from trees to arbitrary hyperbolic spaces. 

\begin{prop}\label{Prop:qm}
If a  weakly hyperbolic group $G$ is anisotropic, then $\dim \widetilde{QH}(G)=\infty$.
\end{prop}
\begin{proof}
   Since $G$ is anisotropic, there exists $X\in Hyp_{gt}(G)$ such that $G\act Cay(G,X)$ is not isotropic. Since this action is cobounded, Theorem \ref{Thm:Iso} implies that $\L(G\act Cay(G,X))$ contains non-equivalent elements. Applying Lemma \ref{Lem:BF}, we conclude that $G\act Cay(G,X)$ satisfies the Bestvina--Fujiwara condition. By the main result of \cite{BF}, any group $G$ admitting an action on a hyperbolic space satisfying the Bestvina--Fujiwara condition has infinite dimensional $\widetilde{QH}(G)$.
\end{proof}

In what follows, we use the standard action of $SL_2(\CC)$ on $\mathbb H^3$ by M\"obius transformations. For our purpose, it suffices to know that a matrix $A$ from  $SL_2(\CC)$ acts loxodromically on  $\mathbb H^3$ if and only if $(tr(A))^2 \notin [0,4]$. Further, it is well-known and easy to see that the action of $SL_2(\ZZ)$ on $\mathbb H^3$ is of general type. It follows that \textit{$SL_2(R)$ is weakly hyperbolic for any subring $R$ of $\CC$}. For details, we refer the reader to \cite{Ber}.

As a corollary of Proposition \ref{Prop:qm}, we obtain the result about uniformly perfect groups stated in the introduction.

\begin{proof}[Proof of Proposition \ref{Prop:UP}]
Let $G$ be a weakly hyperbolic group. It is easy to derive from the definition that every quasi-morphism $G\to \RR$ is uniformly bounded on the set of all commutators. Therefore, a uniformly perfect group $G$ has no unbounded quasi-morphisms. Further, it is straightforward to see that $\dim\widetilde{QH(G)}<\infty$  for every boundedly generated group $G$ (see, for example, \cite[Proposition 5]{Kot}). Thus, a uniformly perfect or boundedly generated group cannot be anisotropic by Proposition \ref{Prop:qm}.
\end{proof}

In particular, Proposition \ref{Prop:UP} allows us to classify all isotropic groups of the form $SL_2(R)$, where $R$ is a ring of algebraic numbers or a subfield of $\CC$. As we explained above, $SL_2(R)$ is weakly hyperbolic for any subring $R$ of $\CC$.

\begin{proof}[Proof of Corollary \ref{Cor:SL2}]
(a) The claim about isotropy of $SL_2(F)$ for a subfield $F\subseteq \CC$ follows immediately from the well-known fact that $SL_2(F)$ is uniformly perfect. Indeed, it is not difficult to show that every element of $SL_2(F)$ is a product of at most $4$ elementary matrices of the form $\left(\begin{array}{cc} 1 & \ast \\ 0 &1\end{array}\right)$ or $\left(\begin{array}{cc} 1 & 0\\ \ast  &1\end{array}\right)$, and every such matrix is a commutator unless $|F|=2$ or $3$. Thus, Proposition \ref{Prop:UP} applies.

(b) We first note that, for any integer $d<0$ and any $R\subseteq \mathcal O_d$, the group $SL_2(R)$ is anisotropic by Corollary~\ref{Cor:WWPD} since the standard action of the Bianchi group $PSL_2(\mathcal O_d)$ on $\mathbb H^3$ is proper.  

Suppose now that there is no integer $d<0$ such that $R\subseteq \mathcal O_d$. Then $SL_2(R)$ is boundedly generated by Theorem \ref{Thm:MSR}. Hence, $SL_2(R)$ is isotropic by Proposition \ref{Prop:UP}.
\end{proof}

Considering $p$-adic valuations on $\QQ$ corresponding to a finite set of primes, we obtain the following. 

\begin{ex}
    For any primes $p_1, \ldots, p_n$, the group $SL_2\big(\ZZ[1/p_1, \ldots, 1/p_n]\big)$ is isotropic.
\end{ex}

Our next goal is to provide examples of anisotropic weakly hyperbolic groups with a continuum of general type hyperbolic structures. Given a field $F$, any embedding $\alpha\colon F\to \CC$ induces an embedding $\widehat\alpha \colon SL_2(F)\to SL_2(\CC)$. Composing $\widehat \alpha$ with the standard action $SL_2(\CC)\act \mathbb H^3$, we obtain an action of $SL_2(F)$ on $\mathbb H^3$. We call it the \textit{action of $SL_2(F)$ induced by the embedding $\alpha$.}

\begin{lem}\label{Lem:Emb}
Let $F$ be a field, $\alpha, \beta\colon F\to \CC$ two distinct embeddings. Suppose that there exists $f \in F$ such that $\alpha (f)$ and $\beta(f)$ are real and not equal. Then the actions of $SL_2(F)$ on $\mathbb H^3$ induced by $\alpha$ and $\beta$ are not equivalent. 
\end{lem}

\begin{proof}
Without loss of generality, we can assume that  $\alpha (f)< \beta(f)$. Fix any rational $r$ such that $|\alpha(f)-r|< 1$ and $|\beta(f)-r|> 1$. We think of $\QQ$ as a subfield of $F$ and let $x=f-r$. Since the restrictions $\alpha$ and $\beta$ to $\QQ$ are the identity maps, we have $
|\alpha (x)| = |\alpha(f)-r|< 1$
and 
$|\beta (x)| = |\beta(f)-r|> 1$.
Consider the matrix
$$
A=\left(
\begin{array}{cc}
x & x^2-1 \\
1 & x \\
\end{array}
\right).
$$
Since $tr(A)=2x$, $A$ is elliptic (respectively, loxodromic) with respect to the action induced by $\alpha$  (respectively, $\beta$); in particular, these actions are not equivalent.
\end{proof}

\begin{cor}\label{Prop:F}
For any countable transcendental extension $F$ of $\QQ$, the restriction of $\sim$ to $Hyp_{gt} (SL_2(F))$ is Borel bi-reducible with $=_{\RR}$.
\end{cor}

\begin{proof}
By Proposition \ref{Prop:UP}, $SL_2(F)$ is isotropic. By Corollary \ref{Cor:Iso}, the restriction of $\sim $ to $Hyp_{gt}(SL_2(F))$ is smooth. Thus, it suffices to show that 
\begin{equation}\label{Eq:Hgt}
|\H_{gt}(SL_2(F))|\ge 2^{\aleph_0}.
\end{equation}

Let $X$ (respectively, $Y$) be a transcendence basis of $F$ (respectively, $\RR$) over $\QQ$ and let $x\in X$. Every map sending $x$ to an element $y\in Y$ extends to an injection $X\to Y$, which in turn extends to an embedding $F\to \CC$ (by first extending to $\QQ(X)$ in a natural way, and then to $F$ using the facts that $F/\QQ(X)$ is algebraic and $\CC$ is algebraically closed). Lemma~\ref{Lem:Emb} guarantees that these embeddings give rise to distinct general type hyperbolic structures on $SL_2(F)$ for different elements $y\in Y$. Since $|Y|=2^{\aleph_0}$, inequality (\ref{Eq:Hgt}) follows.
\end{proof}

\begin{rem}
Lemma \ref{Lem:Emb} can be used to provide many other examples of isotropic groups $G$ such that the restriction of $\sim$ to $Hyp_{gt} (G)$ is Borel bi-reducible to $=_{\RR}$. For example, this property holds for $G=SL_2(F)$ whenever $F$ is a countable field such that some subfield of $F$ admits $2^{\aleph_0}$ real embeddings. For example, we can take $F$ to be the algebraic closure of $\QQ$ or $F=\QQ(\sqrt{p_1}, \sqrt{p_2}, \ldots))$, where $p_1, p_2, \ldots$ is any infinite sequence of pairwise distinct primes. We leave the (obvious) details to the reader.
\end{rem}

We return to the relation between group actions on hyperbolic spaces and quasi-morphisms.

\begin{proof}[Proof of Proposition \ref{Prop:qm-intr}]
Suppose that $q\colon G\to \RR$ is an unbounded quasi-morphism subordinate to a general type action of a group $G$ on a hyperbolic space $S$. Replacing $q$ with $\widetilde q$ (see the discussion before Lemma \ref{Lem:qgg-1}) we can assume that $q$ is a quasi-character. Since $q$ is unbounded, there is $g\in G$ such that $|q(g)|>D(q)$. By Lemma \ref{Lem:qgg-1}, $g\in \L(G\act S)$ and $g\not\sGS g^{-1}$. By Theorem \ref{Thm:Iso}, this implies that the action $G\act S$ is not weakly isotropic. Hence, $G$ is not isotropic by Proposition \ref{Prop:IsoEqDef}.
\end{proof}

\begin{cor}\label{Cor:GAT}
\begin{enumerate}
\item[(a)] Let $G=H_{A^t=B}$ be an $HNN$-extension of a group $H$ with associated subgroups $A$ and $B$. If $A\ne H\ne B$, then $G$ is an anisotropic weakly hyperbolic group.

\item[(b)] The amalgamated product $G=A\ast_CB$ is anisotropic if $C\ne A$ and $|C\backslash B/C|\ge 3$.
\end{enumerate}
\end{cor}

\begin{proof}
    To prove (a) and (b), we first note that our assumptions $A\ne H\ne B$ (respectively, $A\ne C\ne B$) imply that the action of $G$ on the Bass-Serre tree $T$ associated to the graph of groups decomposition is of general type (see, for example, \cite[Proposition 4.13]{MO}). Further, in case (a), the natural homomorphism $G\to G/\ll A\cup B \rr =\langle t\rangle\cong \ZZ$ is clearly unbounded and subordinate to the action on $T$. Similarly, in case (b) the existence of an unbounded quasi-morphisms $q\colon G\to \RR$ was established by Fujiwara in \cite{Fuj}; the fact that these quasi-morphisms are subordinate to $G\act T$ is not mentioned explicitly in the paper, but is obvious from the construction.  To complete the proof, it remains to apply Proposition \ref{Prop:qm-intr}. 
\end{proof}

Finally, we record an elementary fact backing up Example \ref{Ex:I}. For a group $G$ and a cardinal number $n$, we denote by $G^n$ the direct sum of $n$ copies of $G$.

\begin{prop}\label{Prop:DS}
Let $G$ be an isotropic weakly hyperbolic group. Then, for every cardinal number $n$, the group $G^n$ is isotropic and satisfies 
$$
|\H_{gt}(G^n)|=n|\H_{gt}(G)|.
$$
\end{prop}
\begin{proof}
Let $G^n=\oplus_{i\in I} G_i$, where $|I|=n$ and $G_i\cong G$ for all $i\in I$, and let $X\in Hyp_{gt}(G^n)$. It is well-known (see, for example, \cite[Theorem 4.5]
{Osi17}) that the action of every normal subgroup $N\lhd G^n$ on $\Gamma=Cay(G^n,X)$ either has bounded orbits or is of general type and $\Lambda_{\Gamma}(N)=\Lambda_{\Gamma}(G^n)$. If all $G_i$ have bounded orbits in $\Gamma$, then $G^n$ has no loxodromic elements (the assumption that $G^n$ is a direct \textit{sum} of subgroups $G_i$ is used here), which contradicts our assumption that $X\in Hyp_{gt}(G)$. Therefore, there is $i_0\in I$ such that the induced action $G_{i_0}\act \Gamma$ is of general type and $\Lambda_{\Gamma}(G_{i_0})=\Lambda_{\Gamma}(G^n)$. Since $G_{i_0}$ is isotropic, the action of $G_{i_0}$ on $\Lambda_{\Gamma}(G_{i_0})\otimes \Lambda_{\Gamma}(G_{i_0})$ is minimal by Theorem \ref{Thm:Iso} and Proposition \ref{Prop:IsoEqDef}; therefore, so is the action of $G^n$ on  $\Lambda_{\Gamma}(G^n)\otimes \Lambda_{\Gamma}(G^n)= \Lambda_{\Gamma}(G_{i_0})\otimes \Lambda_{\Gamma}(G_{i_0})$. Using Theorem \ref{Thm:Iso} again, we obtain that the action $G^n\act \Gamma$ is isotropic. Since this is true for any $X\in Hyp_{gt}(G^n)$, the group $G^n$ is isotropic. 

Further, we have a decomposition $G^n=G_{i_0}\oplus E$. It is well-known that the action of $E$ on $\Gamma$ must be elliptic in this situation (see, for example, \cite[Lemma 4.20]{ABO}). It follows that $G^n \act \Gamma$ is equivalent to the action of $G^n$ obtained by composing the homomorphism $G^n\to G^n/E\cong G_{i_0}$ with the action $G_{i_0}\act \Gamma$. By Proposition \ref{Prop:IsoEqDef} and Theorem \ref{Thm:Iso}, the latter action is equivalent to  $G_{i_0}\act Cay(G_{i_0}, Y)$ for some $Y\in Hyp_{gt} (G_{i_0})$. Thus, we obtain $|\H_{gt}(G^n)|\le n|\H_{gt}(G)|$. The opposite inequality follows from the fact that inequivalent actions of distinct direct summands $G_i$ of $G^n$ give rise to inequivalent actions of $G^n$. 
\end{proof}

%%%%%%%%%%%%%%%%%%%%%%%%%%%%%%%%%%%%%%%%%%%%%%%%%%%%%%%%%%%%%%%%%

\section{Proofs of the main results}\label{Sec:DV}

%%%%%%%%%%%%%%%%%%%%%%%%%%%%%%%%%%%%%%%%%%%%%%%%%%%%%%%%%%%%%%%%%

\paragraph{7.1. Infinitary logic and definable subsets of $PL(G)$.}\label{Sec: BSPL}
Throughout this subsection, we fix an arbitrary countable group $G$. Every element $\ell\in PL(G)$ gives rise to a structure $G_\mathcal \ell$ in the signature
$$
\Sigma=\{ 1, ^{-1}, \cdot, \{ L_q\}_{q\in \QQ}\},
$$
where $1$ is a constant symbol, $^{-1}$ (respectively, $\cdot$) is a unary (respectively, binary) operation, and each $L_q$ is a unary predicate, as follows. The universe of the structure is $G$. The constant $1$ and operations $^{-1}$, $\cdot$ are interpreted in the usual way. Finally, $G_\ell\models L_q(g)$ for some $g\in G$ and $q\in \QQ$ if and only if $\ell(g)\le q$.

As usual, by $\mathcal L_{\omega_1, \omega}$ we denote the infinitary analog of the first order language in the signature $\Sigma$ that allows countable conjunctions and disjunctions but only finite sequences of quantifiers. More precisely, terms and atomic formulas in the signature $\Sigma$ are defined in the usual way. Further, let $\mathcal{L}_0$ denote the set of all atomic formulas. If $\alpha=\beta+1$ for some ordinal $\beta$, let $\mathcal{L}_\alpha$ be the union of $\mathcal{L}_\beta$  and the set of all well-formed formulas of the following two types.
\begin{enumerate}
\item[(a)] $\forall\, x\; \phi$, $\exists\, x\;\phi$, or $\neg\phi$, where $\phi \in \L_\beta$.
\item[(b)] $\bigwedge_{n\in \NN}\phi_n$ or $\bigvee_{n\in \NN}\phi_n$, where $\phi_n\in\mathcal{L}_\beta$ for all $n\in \NN$.
\end{enumerate}
If $\alpha>0$ is a limit ordinal, we define $\mathcal{L}_\alpha=\bigcup_{\beta<\alpha}\mathcal{L}_\beta$. By definition, $$\Lw=\bigcup_{\alpha<\omega_1}\mathcal{L}_\alpha$$ where $\omega_1$ is the first uncountable ordinal. Satisfaction of formulas in $\Sigma$-structures is defined in the usual way. For more details, we refer the reader to \cite{Mar}.

Given a pseudo-length function $\ell\in PL(G)$, we define the \emph{corresponding pseudo-distance function} $\d_\ell \colon G\times G\to [0, \infty)$ by the formula 
$$
\d_\ell (g,h)= \ell(g^{-1}h).
$$
Further, for any $f,g,h\in G$, we let 
$$
(g,h)_f^\ell = \frac12 \Big( \ell(f^{-1}g) +\ell(f^{-1}h) - \ell(g^{-1}h)\Big).
$$
Thus, $(g,h)_f^\ell$ is simply the Gromov product in the pseudo-metric space $(G, \d_\ell)$.

The expressive power of $\L_{\omega_1,\omega}$ far exceeds that of the first order logic. In fact, every ``explicitly definable" property of the pseudo-metric space $(G, \d_\ell)$ can be encoded by an $\L_{\omega_1, \omega}$-formula. To illustrate this, we consider two examples relevant to the discussion below. 

\begin{ex}\label{Ex:Lww} Let $A$ be a non-negative real number.
\begin{enumerate}
    \item[(a)] For any $g\in G$, the inequality $\ell(g)> A$  can be expressed by the formula 
$$
\sigma (g,A)=  \bigvee\limits_{q\in \QQ\cap (A,\infty)} \lnot L_q(g)
$$

\item[(b)] Let $Q(A)=\{ (B,C,D)\in \QQ^3\mid  B+C - D> 2A\}$. For any $f,g,h\in G$, the inequality $(g,h)_f^\ell> A$ can be expressed by 
$$
\bigvee_{(B,C,D)\in Q(A)} \Big( \sigma (f^{-1}g, B) \wedge \sigma(f^{-1}h, C)\wedge \lnot\sigma(g^{-1}h, D)\Big). 
$$
Indeed, $(g,h)_f^\ell> A$ can be rewritten as $\ell(f^{-1}g)+\ell(f^{-1}h)-\ell(g^{-1}h) > 2A$, which in turn is equivalent to the conjunction of $\ell(f^{-1}g)>B$, $\ell(f^{-1}h)>C$, and $\ell(g^{-1}h) \le D$ for some $(B,C,D)\in Q(A)$.
\end{enumerate} 
\end{ex}

\begin{defn}\label{Defn:PLHyp}
Let $G$ be a group. We say that a pseudo-length function $\ell \in PL(G)$ is \emph{hyperbolic} (respectively \emph{hyperbolic of general type}) if there is an action (respectively, general type action) of $G$ on a hyperbolic space $S$ and a point $s\in S$ such that $\ell(g)=\d_S(s,gs)$ for all $g\in G$. By $HypPL(G)$ (respectively, $HypPL_{gt}(G)$) we denote the subspace of $PL(G)$ consisting of all hyperbolic pseudo-length functions (respectively, hyperbolic pseudo-length functions of general type) on $G$.
\end{defn}

Lemma \ref{Lem:ODqoPres} confirms that studying the restriction of the equivalence relation $\ePL$ to $HypPL_{gt}(G)$ provides an appropriate formal framework for classifying general type $G$-actions on hyperbolic spaces. 
The following result (together with Example \ref{Ex:Lww}) will help us express the properties of pseudo-lengths introduced in Definition \ref{Defn:PLHyp} in $\L_{\omega_1, \omega}$ . Note that the fact that a metric space satisfying ({\bf H}$_3$) embeds equivariantly in a hyperbolic space was already proved in Corollary A.10 of the arXiv version of \cite{BHM11} following the Bonk--Schramm's method, while we prove it below using the Isbell's injective hull construction.

\begin{prop}\label{Prop:HypPL} 
Let $G$ be a group.
\begin{enumerate}
\item[(a)] A pseudo-length function $\ell$ on $G$ is hyperbolic if and only if there exists $\delta \ge 0$ such that 
\begin{equation}\label{Eq:4pt}
(x, z)_t^\ell \ge \min\big\{(x, y)_t^\ell,\, (y, z)_t^\ell\big\} - \delta\;\;\;\;\; \forall \, x, y, z, t \in G.
\end{equation}
\item[(b)] A hyperbolic pseudo-length function $\ell\in PL(G)$ is of general type if and only if there exist two elements $g_1,g_2\in G$ and constants $\lambda >0$, $C,N \ge 0$ such that $\ell(g_i^k)\ge \lambda k$ for all $k\in \NN$ and $i=1,2$, and 
$(g_1^m, g_2^n)_1^\ell\le C$  for all $m,n\in \ZZ$.
\end{enumerate}
\end{prop}

\begin{proof}
The forward implications in (a) and (b) easily follow from the definitions, so we focus on the ``if" claims. Suppose that $\ell\in PL(G)$ satisfies (\ref{Eq:4pt}). The set 
$$
K = \{g \in G \mid \ell(g) = 0\}
$$
is clearly a subgroup of $G$. We denote the set of cosets $G/K$ by $X$ and define a map $\d_X\colon X \times X \to [0,\infty)$ by the formula $$\d_X(gK,hK) = \ell(g^{-1}h)\;\;\;\;\; \forall\, g,h \in G.$$ Note that $\d_X$ is well-defined since we have $\ell(k_1gk_2)=\ell(g)$ for all $g \in G$  and $k_1,k_2 \in K$. Obviously, $\d_X$ is a metric on $X$ and $G$ acts on $(X,\d_X)$ isometrically by left multiplication. 

Recall that a metric space $Y$ is said to be \textit{injective} if for every metric space $B$ and every $1$-Lipschitz map $f\colon A\to Y$ defined on a subset $A \subset B$, there exists a $1$-Lipschitz
extension $B\to Y$ of $f$. In particular, every injective space is geodesic. We now apply Isbell's \textit{injective hull} construction to the space $(X, \d_X)$ to get a geodesic hyperbolic space $E(X)$ equipped with the desired action of $G$. 
More precisely, we define
$$ 
        \Delta(X) = \big\{f \in \RR^X \mid \forall\, x,y \in X,\, f(x)+f(y) \ge \d_X(x,y)\big\}
$$
and write $g\le f$ for $f,g\in \Delta(X)$ if $g(x)\le f(x)$ for all $x\in X$. Further, consider the set of minimal elements       
$$
E(X) = \big\{f \in \Delta (X) \mid \forall\, g \in \Delta (X)\;  \big(g\le f \Rightarrow g=f\big)\big\}.
$$
Further,  for all $f,g\in E(X)$, let
$$d_{E(X)}(f,g) = \sup_{x \in X}|f(x)-g(x)|.$$ 

As shown in \cite[Theorem 2.1]{Isb64},  $(E(X), \d_{E(X)})$ is an injective metric space and the map $\iota \colon X \to E(X)$ defined by $\iota(x)(t)=\d_X(x,t)$ for all $x,t \in X$ is isometric (see also \cite[Theorem 3.3]{Lan13}). Furthermore, since $X$ satisfies ({\bf H}$_3$), the space $E(X)$ is hyperbolic by \cite[Proposition 1.3]{Lan13}. 

For any $\phi \in Isom(X)$, the map $\widetilde{\phi} \colon \RR^X \to \RR^X$ defined by 
$$
\widetilde\phi(f)(t) = f(\phi^{-1}(t))\;\;\;\;\; \forall \, f \in \RR^X,\; \forall\, t \in X,
$$
satisfies $(\widetilde{\phi}\circ\iota)(x) = (\iota\circ\phi)(x)$ for any $x \in X$ and induces an isometry $\widetilde{\phi}|_{E(X)} \in Isom(E(X))$. The group homomorphism $Isom(X) \to Isom(E(X))$ given by $\phi \mapsto \widetilde{\phi}|_{E(X)}$ yields an isometric action $G \act E(X)$ whose restriction to $X$ is $G \act (X,\d_X)$. Hence, for any $g \in G$, we have 
\begin{equation}\label{Eq:l(g)}
\ell(g) = \d_X(K, gK) = \d_{E(X)}(s,gs)
\end{equation}
where $s = \iota(K)\in E(X)$. This completes the proof of the ``if" direction of part (a). 

Finally, suppose that $g_1$ and $g_2$ are elements satisfying the inequalities listed in (b). A straightforward verification using (\ref{Eq:l(g)}), shows that $g_1$ and $g_2$ are independent loxodromic isometries of $E(X)$. Thus, the action $G \act E(X)$ is of general type.
\end{proof}

Given a sentence $\sigma \in \Lw$, we define 
$$
\Mod(\sigma)=\{ \ell\in PL(G)\mid G_\ell \models \sigma\}.
$$
We will need the following version of the well-known Lopez-Escobar theorem (see \cite{LE} or \cite[Theorem~16.8]{Kec}). A similar result for integer-valued length functions on groups was obtained in \cite{JOO}. Our proof essentially repeats the same argument.

\begin{prop}\label{Prop:LE}
Let $G$ be a group and let $\ell\in PL(G)$. For any countable ordinal $\alpha$, any $\Lw$-formula $\sigma(x_1,\cdots,x_n)\in \mathcal{L}_\alpha$ with finitely many free variables $x_1,\cdots,x_n$, and any tuple $\overline{g} =(g_1,\cdots,g_n)\in G^n$,
the set $$\Mod(\sigma,\overline{g}) =
\{\ell\in PL(G) \mid G_\ell\models \sigma(g_1,\cdots,g_n)\} $$
is a Borel subset in $PL(G)$. In particular,  $\Mod(\sigma)$ is a Borel subset of $PL(G)$ for any sentence $\sigma \in \Lw$.
\end{prop}

\begin{proof}
We proceed by transfinite induction on $\alpha$. If $\alpha=0$, $\sigma$ is an atomic formula. If it does not contain any predicate symbols, then $\sigma(g_1,\cdots,g_n)$ is a formula in the language of groups and $\Mod(\sigma,\overline{g})$ is either equal to $PL(G)$ or empty depending on whether $G$ satisfies $\sigma(g_1, \ldots, g_n)$ or not. In particular, $\Mod(\sigma,\overline{g})$ is Borel. If $\sigma=L_q(t(x_1,\cdots,x_n))$ for some $q \in \QQ$, where $t(x_1,\cdots,x_n)$ is a term, then $\Mod(\sigma,\overline{g})$ coincides with the set 
$$\{\ell\in PL(G) \mid\ell(t(g_1,\cdots,g_n)) \le q \},$$ 
which is closed in $PL(G)$.

Further, assume that $\alpha>0$ and that the lemma holds for any $\beta<\alpha$. We fix some $\overline{g}\in G^n$. If $\alpha$ is a limit ordinal, the claim follows immediately from the inductive assumption. If $\alpha=\beta+1$, there are three cases to consider.

\medskip 

\noindent{\it Case 1.} First assume that $\sigma=\forall\, x\;\tau(x,x_1,\cdots,x_n)$ or $\sigma= \exists\, x\; \tau(x,x_1,\cdots,x_n)$, where $\tau \in \mathcal{L}_\beta$. Then $\Mod(\sigma,\overline{g})$ equals $\bigcap_{g\in G} \Mod(\tau,(g,\overline{g}))$ or $\bigcup_{g\in G} \Mod(\tau,(g,\overline{g}))$, respectively. Since $G$ is countable and $\Mod(\tau,(g,\overline{g}))$ is Borel for any $g\in G$ by the inductive assumption, $\Mod(\sigma,\overline{g})$ is also Borel.

\medskip

\noindent{\it Case 2.} Next, suppose that $\sigma=\neg \tau(x_1,\cdots,x_n)$, where $\tau \in \mathcal{L}_\beta$.  Then $\Mod(\sigma,\overline{g})$ coincides with $PL(G)\setminus \Mod(\tau,\overline{g})$ and the latter set is Borel by the inductive assumption.

\medskip

\noindent{\it Case 3.} Finally, let $\sigma=\bigwedge_{i\in \NN} \tau_i \left(x_1,\cdots,x_n\right)$ or
$\sigma=\bigvee_{i\in \NN} \tau_i \left(x_1,\cdots,x_n\right)$, where $\tau_i \in \mathcal{L}_\beta$  for all $i$ (some of the variables $x_1,\cdots x_n$ can be absent in some $\tau_i$). In this case, $\Mod(\sigma,\overline{g})$ equals $\bigcap_{i\in \NN}  \Mod\left(\tau_i,\left(g_1,\cdots,g_{n}\right)\right)$ or $\bigcup_{i\in \NN} \Mod\left(\tau_i,\left(g_1,\cdots,g_n\right)\right)$) and the claim follows.
\end{proof}

\begin{prop}\label{Prop:HypPLBorel}
For every countable group $G$, $HypPL(G)$ and $HypPL_{gt}(G)$ are Borel subsets of $PL(G)$.
\end{prop}

\begin{proof}
    It is easy to verify that conditions listed in parts (a) and (b) of Proposition~\ref{Prop:HypPL} can be encoded by sentences in $\L_{\omega_1,\omega}$. For example, elaborating the idea behind Example \ref{Ex:Lww} (b), we can obtain the sentence encoding the first condition as follows. Set $x_1=x$, $x_2=y$, $x_3=z$, and $x_4=t$. Part (a) of Proposition~\ref{Prop:HypPL} tells us that there exists a continuous function $f\colon \RR^6\to \RR$ such that $\ell \notin HypPL(G)$ if and only if for any $\delta\ge 0$ there exist $x_1, \ldots, x_4\in G$ such that 
\begin{equation}\label{Eq:faij}
f\Big(\big(\ell(x_i^{-1}x_j)\big)_{1\le i< j \le 4}\Big)>  \delta.
\end{equation}
Given $\delta\ge 0$, we define a subset $A_\delta \subseteq (\QQ^2)^6$  by the formula
\begin{equation*}
A_\delta = \left\{ \big((a_{i,j},b_{i,j})\big)_{1\le i< j \le 4} \in (\QQ^2)^6\; \left| \; 
\begin{aligned}
& \text{ for any } i<j \text{ and any } \,(c_{i,j})_{1\le i< j \le 4}\in \RR^6, \\
&\hspace{14mm}\text{ we have }   
a_{i,j}<b_{i,j} \text{ and }  \\
& a_{i,j}<c_{i,j}\le b_{i,j} \; \Longrightarrow \;
         f\big((c_{i,j})_{1\le i< j \le 4}\big)> \delta
\end{aligned}
         \right.\right\}
\end{equation*}
and let 
$$
\sigma_\delta(x_1,\cdots,x_4) = \bigvee_{\big((a_{i,j},b_{i,j})\big)_{1\le i< j \le 4} \in A_\delta}\;\; \bigwedge_{1\le i< j \le 4} \Big(\neg L_{a_{i,j}}(x_i^{-1}x_j) \wedge L_{b_{i,j}}(x_i^{-1}x_j)\Big).$$ Since $f$ is continuous, $\sigma_\delta(x_1, \ldots, x_4)$ is true if and only if (\ref{Eq:faij}) holds. Now, letting 
$$
\sigma = \bigwedge_{\delta\in \QQ\cap [0, \infty)} \exists\, x_1 \; \ldots \exists\, x_4 \; \sigma_{\delta} (x_1, \ldots x_4),
$$
we obtain $\Mod(\lnot \sigma)=HypPL(G)$.

Similarly, conditions listed in part (b) of Proposition~\ref{Prop:HypPL} can be rewritten as an $\L_{\omega_1, \omega}$-sentence using formulas from Example \ref{Ex:Lww}.  Therefore, $HypPL(G)$ and $HypPL_{gt}(G)$ are Borel by Proposition~\ref{Prop:LE}.
\end{proof}

\begin{rem}
In a similar manner, one can define Borel subsets of $PL(G)$ corresponding to elliptic, linear, parabolic, and quasi-parabolic $G$-actions on hyperbolic spaces. Since these subsets are not used in our paper, we leave this as an exercise for the interested reader.
\end{rem}

Combining Proposition \ref{Prop:HypPLBorel} with part (c) of Proposition \ref{Prop:PLGen}, we obtain the following corollary, confirming part (c) of Proposition \ref{Prop:Hyp}.

\begin{cor}\label{Cor:Hyp}
For every countable group $G$, $Hyp(G)$ and $Hyp_{gt}(G)$ are Borel subsets of $Gen(G)$.
\end{cor}

\begin{rem}
    In a similar way, we can prove the analogs of Proposition \ref{Prop:HypPLBorel} and Corollary~\ref{Cor:Hyp} for $G$-actions on spaces with virtually any ``explicitly definable" geometry. The use of $\L_{\omega_1,\omega}$ provides a convenient universal formalization of our approach, but can be avoided in most particular cases.
\end{rem}

\paragraph{7.2. Complexity of $(Hyp_{gt}(G),\sim)$ and $(HypPL_{gt}(G), \ePL)$.} \label{Sec:PMT}
In what follows, we consider the spaces of functions $\RR^\Omega$, where $\Omega$ is some set; we always endow these spaces with the product topology. By \cite[Proposition 3.3]{Kec}, the space $\RR^\Omega$ is Polish whenever $\Omega $ is countable. By $\bf 0$ we denote the zero function $\Omega\to \RR$. Given two functions $f,g\in \RR^\Omega\setminus\{ \bf 0\}$, we write $f\approx g$ if there is $c\in \RR$ such that $f(\omega)=cg(\omega)$ for all $\omega\in \Omega$.   Clearly, $\approx$ is an equivalence relation on the subspace $\RR^\Omega\setminus\{ \bf 0\}$ of $\RR^\Omega$. 

\begin{lem}\label{Lem:smooth}
    For any countable set $\Omega$,  the equivalence relation $\approx$ on  $\RR^\Omega\setminus\{ \bf 0\}$ is smooth.
\end{lem}

\begin{proof}
We first note that $P=\RR^\Omega\setminus\{ \bf 0\}$ is Polish being a $G_\delta $ subset of $\RR^\Omega$. Further, we show that $\approx$ is a closed subset of $P\times P$. Indeed, suppose a sequence $(p_i, q_i)$ converges to $(p,q)$ in $P \times P$ and  $q_i=c_ip_i$ for some $c_i\in \RR$.  Consider any $\omega\in \Omega$ such that $p(\omega)\ne 0$. By convergence, we have $p_i(\omega)\to p(\omega )$ and $q_i(\omega)\to q(\omega)$. In particular, there is $N\in \NN$ such that $p_i(\omega)\ne 0$ for all $i\ge N$ and we can write $c_i=q_i(\omega)/p_i (\omega)$. Therefore the sequence $(c_i)$ converges to $q(\omega)/p(\omega)$, which easily implies the equality $q=cp$. 

It remains to use the well-known fact that every $G_\delta $ (in particular, every closed) equivalence relation on a Polish space is smooth (see \cite[Theorem 1.1]{HKL}).
\end{proof}

\begin{lem}\label{Lem:TauBorel}
    Let $G$ be a countable group, then the map $ PL(G)\to \RR^G$ defined by  $\ell \mapsto \tau_\ell$, where $\tau_\ell(g)=\liminf_{n\to \infty} \ell(g^n)/n$ for each $g \in G$, is Borel.
\end{lem}

\begin{proof}
    For any $g \in G$ and $n \in \NN$, the map $PL(G) \ni \ell \mapsto \ell(g^n)/n \in \RR$ is continuous. Hence, for any $g \in G$, the map $PL(G) \ni \ell \mapsto \tau_\ell(g) \in \RR$ is Borel since limit inferior of measurable functions is measurable, and the claim follows.
\end{proof}

\begin{lem}\label{Lem:red}
    For any countable group $G$, the equivalence relation $\ePL$ on $PL(G)$ is Borel reducible to $E_{K_\sigma}$.
\end{lem}
\begin{proof}
We enumerate the group $G=\{g_1, g_2, \ldots\}$ and consider the map sending a function $\ell \in PL(G)$ to the function $f_\ell \in \RR^\omega$ (where $\RR^\omega$ denotes the set of all infinite sequences of reals) defined by $f_\ell (i) = \log_2 (\ell(g_i)+1)$. It is straightforward to see that this function reduces $\sim$ to the orbit equivalence relation of the natural action of $\ell^\infty$ on $\RR^\omega$. The latter OER is Borel reducible (in fact, bi-reducible) to $E_{K\sigma}$ as observed by Rosendal \cite{Ros}.
\end{proof}

The following immediate corollary of Theorem \ref{Thm:Iso} and Proposition \ref{Prop:EmbPi} and \ref{Prop:IsoEqDef} subsumes Theorem \ref{Thm:Aniso}.

\begin{cor}\label{Cor:PiEmb}
Let $G\curvearrowright S$ be a general type action of a group $G$ on a hyperbolic space $S$. Suppose that $G\act S$ is not weakly isotropic. Then there exists a Borel map $f\colon \Pi \to Hyp_{gt}(G)$ such that the following conditions hold.
\begin{enumerate}
    \item[(a)] For every $X\in f(\Pi)$, we have $G\curvearrowright Cay(G,X) \lA G\curvearrowright S$.
    \item[(b)] For all $r,s\in \Pi$, we have $f(r)\preccurlyeq f(s)$ if and only if $r\, Q_{K_\sigma}\,s$.
\end{enumerate} 
In particular, for any anisotropic weakly hyperbolic group $G$, $E_{K_\sigma}$ is Borel reducible to the restriction of $\sim$ to $Hyp_{gt}(G)$.
\end{cor}

Finally, we obtain a strengthened version of Theorem \ref{Thm:main}. 

\begin{thm}\label{Thm:main-full}
Let $G$ be a countable weakly hyperbolic group. The restriction of $\sim$ to $Hyp_{gt}(G)$ and the restriction of $\ePL$ to $HypPL_{gt}(G)$ are simultaneously Borel bi-reducible to one of the relations $=_n$ ($n\in \NN$), $=_\NN$, $=_\RR$, $E_{K_\sigma}$, and each possibility realizes for a suitable group $G$.  
In particular, the restriction of $\sim$ to $Hyp_{gt}(G)$ is Borel bi-reducible to the restriction of $\ePL$ to $HypPL_{gt}(G)$, and these equivalence relations are either simultaneously smooth or not classifiable by countable structures. 
\end{thm}

\begin{proof}
Suppose first that $G$ is isotropic. Theorem \ref{Thm:LoxRigAct} and Lemmas \ref{Lem:smooth} and \ref{Lem:TauBorel} imply that the restriction of $\ePL$ to $HypPL_{gt}(G)$ is smooth.  Hence, it is Borel bi-reducible to one of the relations $=_n$ ($n\in \NN$), $=_\NN$, $=_\RR$. Further, Proposition \ref{Prop:PLGen} (c) provides us with a Borel reduction of $\sim $ to $\ePL$, which sends $Hyp_{gt}(G)$ to $HypPL_{gt}(G)$ by the definition of a hyperbolic length function of general type. Thus, the restriction of $\sim$ to $Hyp_{gt}(G)$ is Borel reducible to the restriction of $\ePL$ to $HypPL_{gt}(G)$; in particular, the former relation is also smooth. Conversely, let $\ell\in HypPL_{gt}(G)$ and let $G\act S$ be a $G$-action on a hyperbolic space such that, for some $s\in S$, we have $\ell(g)=\d_S(s, gs)$ for all $g\in G$. Combining the equivalence of conditions (a) and (g) in Theorem \ref{Thm:Iso} and Proposition~\ref{Prop:IsoEqDef}, we obtain $X_\ell\in Hyp_{gt}(G)$ such that $G\act S \eA G\act Cay(G,X_\ell)$. Clearly, for any $\ell_1, \ell_2\in HypPL_{gt}(G)$, we have $\ell_1 \ePL \ell_2$ if and only if $X_{\ell_1}\sim X_{\ell_2}$. This implies that $|Hyp_{gt}(G)/\sim| \ge |HypPL_{gt}(G)/\ePL|$. Since both the restriction of $\sim$ to $Hyp_{gt}(G)$ and the restriction of $\ePL$ to $HypPL_{gt}(G)$ are smooth, we obtain that the latter relation is Borel reducible to the former. This completes the proof of the theorem in the case when $G$ is isotropic. 

Suppose now that $G$ is anisotropic. Corollary \ref{Cor:PiEmb} guarantees that $E_{K_\sigma}$ is Borel reducible to the restriction of $\sim$ to $Hyp_{gt}(G)$, which in turn is Borel reducible to the restriction of $\ePL$ to $HypPL_{gt}(G)$ as explained in the previous paragraph. By Lemma \ref{Lem:red}, the equivalence relation $\ePL$ (and hence its restriction to $HypPL_{gt}(G)$) is Borel reducible to $E_{K_\sigma}$; therefore, all these relations are Borel bi-reducible with each other. 

To complete the proof of the theorem, it suffices to note that the claim about the realization of all complexity levels follows immediately from the discussion in Section \hyperref[Sec:Ex]{6.4}.
\end{proof}

\paragraph{7.3. An application to hyperbolic structures.}\label{Sec:AtoHS} Let $(R, \preceq)$ denote the poset $\Pi/E_{K_\sigma}$ endowed with the order $\preceq$ induced by $Q_{K_\sigma}$. To prove Corollary \ref{Cor:HS}, we will need the following.

\begin{lem}\label{Lem:PosetEmb}
The poset $(R, \preccurlyeq)$ has height and width $2^{\aleph_0}$ and contains every poset of cardinality at most $\aleph_1$.
\end{lem}

\begin{proof}
Recall that every Boolean algebra can be thought of as a poset (or, more precisely, a complemented distributive lattice). Let $(X, \le)$ be a poset of cardinality at most $\aleph_1$. The map $x\mapsto \{ y\in X \mid y\le x\}$ embeds $(X, \preceq)$ in the Boolean algebra $\mathcal P(X)$. Let $B(X)$ be the subalgebra of $\mathcal P(X)$ generated by the image of $X$ in $\mathcal P(X)$. Clearly, $|B(X)|\le \aleph_0\aleph_1=\aleph_1$. By the Parovichenko theorem \cite{Par}, every Boolean algebra of cardinality at most $\aleph_1$ embeds in the Boolean algebra $\mathcal P(\NN)/fin$, where $fin$ is the ideal consisting of finite sets. It is also well-known and easy to see that $\mathcal P(\NN)/fin$ has height and width $2^{\aleph_0}$.

Thus, to prove the lemma, it suffices to construct an embedding of $\mathcal P(\NN)/fin$ in $(R, \preccurlyeq)$. The existence of such an embedding is well-known and follows, for example, from the Glimm-Effros dychotomy. For the convenience of the reader, we provide a brief hands-on argument. For every $A\subseteq \NN$, let $f_A\in \Pi$ be the function defined by
$$
f_A(n)=\left\{ \begin{array}{l}
                 n,\; {\rm if}\; n\in A, \\
                 1, \; {\rm otherwise}.
               \end{array}
\right.
$$
Clearly, $f_A\, Q_{K_\sigma} f_B$ for some $A,B\subseteq \NN$ if and only if $|A\vartriangle B|<\infty $. Therefore, the map $A\mapsto f_A$ induces an embedding of posets $\mathcal P(\NN)/fin \to R$.
\end{proof}

We are now ready to prove our main result about hyperbolic structures.

\begin{proof}[Proof of Corollary \ref{Cor:HS}]
   If $G$ is isotropic, every $X\in Hyp_{gt}(G)$ represents a minimal element in the poset $\H_{gt} (G)$ by Theorem \ref{Thm:Iso} and Proposition~\ref{Prop:IsoEqDef}. Therefore, $\H_{gt} (G)$ is an antichain. Since $\sim$ is Borel, $\H_{gt} (G)=Hyp_{gt}(G)/\sim$ has cardinality $1, 2, \ldots, \aleph_0$, or $2^{\aleph_0}$ by the Silver dichotomy (see \cite[Theorem 35.20]{Kec}).  As shown in Section \hyperref[Sec:Ex]{6.4}, all possibilities can be realized.      
    
    Further, if $G$ is anisotropic, then $\H_{gt}(G)$ contains the poset $(R,\preceq)$ by Corollary \ref{Cor:PiEmb}, and the required claim follows from Lemma \ref{Lem:PosetEmb}.
\end{proof}

\bigskip

\section{Appendix. Bounded generation of $SL_2$ over finitely generated rings of algebraic numbers.}

\medskip

\begin{center}
 D. Osin, A. Rapinchuk
 \medskip
 
\end{center}

By $\overline{\QQ}$ we denote the algebraic closure of $\QQ$ in $\CC$. For a square-free integer $d\ne 1$, let $\mathcal{O}_d$ be the ring of algebraic integers in the quadratic field $\QQ(\sqrt{d})$. 
Our goal is to prove the following generalization of the main result of \cite{MRS} (see also \cite{KMR}). 
\begin{thm}\label{Thm:MSR}
For any finitely generated subring $R$ of $\overline \QQ$, the following conditions are equivalent: 
\begin{enumerate}
\item[(a)] the group $SL_2(R)$ is boundedly generated;
\item[(b)] the unit group $R^{\times}$ is infinite;
\item[(c)] there is no negative integer $d$ such that $R\subseteq {\mathcal O}_d$.
\end{enumerate}
\end{thm}

Our main reference for definitions and basic results in commutative algebra used below is \cite{L}. For consistency, we accept the definitions of places and valuations given in Sections VII.3 and XII.4 there. (Note that Lang's terminology is different from the one used in \cite{MRS}.) We will need the following.

\begin{lem}[{\cite[Proposition 3.6, Ch. VII]{L}}]\label{Lem:IC}
Let $A$ be a subring of a field $K$. An element $x\in K$ is integral over $A$ if and only if it belongs to every valuation ring of $K$ containing $A$. 
\end{lem}

Here is the key lemma that will enable us to derive Theorem \ref{Thm:MSR} from the main result of \cite{MRS}. 

\begin{lem}\label{Lem:main}
Let $R$ be a finitely generated subring of $\overline \QQ$ with the field of fractions $K$, and let $O$ be the integral closure of $R$ in $K$. 
\begin{enumerate}
\item[(a)]  $R$ is contained in all but finitely many valuation rings of $K$.
\item[(b)] There exists a finite set $S$ of (equivalence classes of) places of $K$ such that $O$ is the ring of $S$-integers of $K$, i.e.,  $O=\{ a\in K \mid \sigma(a) <\infty \; \forall\, \sigma \notin S\}$.     
\item[(c)] There exists a nonzero ideal $I\lhd O$ such that 
%$O/I$ is finite and 
$I\subseteq R$. 
%In particular, the abelian group $(R,+)$ has finite index in $(O,+)$.
\end{enumerate}
\end{lem}
\begin{proof}
Let $X \subseteq R$ be a finite generating set. Then $K = \QQ(X)$, hence a finite extension of $\QQ$. For any $a \in K$, one can find $n(a) \in \NN$ such that $n(a) \cdot a$ lies in the ring $O_K$ of algebraic integers in $K$. Set $n = \prod_{x \in X} n(x)$. Then $n \cdot X \subseteq O_K$. It follows from Ostrowski's Theorem that every place of $K$ lies above one of the $p$-adic places of $\QQ$, and moreover, since the extension $K/\QQ$ is finite, for each prime $p$, the field $K$ has only finitely many places lying above the $p$-adic place (cf. \cite[Corollary 4.9, Ch. XII]{L}). Let $S_0$ be the finite set of places of $K$ lying above the $p$-adic place for some $p \mid n$. Then for any place $\sigma$ of $K$ which is not in $S_0$, we have $\sigma(n) \neq 0$, implying that for every $x \in X$, 
$$
\sigma(x) = \sigma(nx)/\sigma(n) < \infty
$$ 
as $\sigma$ is finite on $O_K$ by Lemma \ref{Lem:IC} applied to $A = \ZZ$. It follows that $X$, and hence $R$, is contained in the valuation ring associated to every place $\sigma \notin S_0$, proving (a). 

To prove (b), we let $S$ denote the subset of $S_0$ consisting of places that are {\it not} finite on $R$, and then let 
$$
O_S = \{ a \in K \, \vert \, \sigma(a) < \infty \ \forall \sigma \notin S \}
$$ 
denote the corresponding ring of $S$-integers. It follows from Lemma \ref{Lem:IC} that $O_S$ is integral over $R$. On the other hand, the fact that a valuation ring is integrally closed implies that $O_S$, 
being an intersection of valuation rings, is also integrally closed. Thus, $O_S$ coincides with the integral closure $O$ of $R$ in $K$. 

Finally, it is well-known that $O = O_S$ is a finitely generated ring (in fact, it is the localization of $O_K$, which is a finitely generated $\ZZ$-module, with respect to finitely many elements corresponding 
to the places in $S$). Since O is finitely generated as an $R$-algebra and integral over $R$, it is a finitely generated $R$-module. We let $a_1, \ldots , a_{\ell}$ denote an arbitrary finite generating set. Since $O$ lies in the field of fractions $K$, for every $a \in O$ one can find a non-zero $r(a) \in R$ such that $r(a) \cdot a \in R$. Letting $r = \prod_{i = 1}^{\ell} r(a_i)$ we obtain that the ideal $I= rO$ is a subset of $R$, as required.  
\end{proof}

Before proceeding to the proof of Theorem \ref{Thm:MSR}, we recall the well-known fact that for any nonzero ideal $I \subseteq O = O_S$ the quotient $O/I$ is finite. One way to prove this is to observe first that a nonzero ideal $I \subseteq O$ has nonzero intersection with the ring of integers $O_K$, hence with $\ZZ$, and if $m > 0$ is an integer contained in $I$ then it is sufficient to prove the finiteness of $O/mO$. Let $\gamma \colon O \to O/mO$ be the canonical homomorphism. Then for any finitely generated additive subgroup $H \subseteq O/mO$ there exists a finitely generated subgroup $\widetilde{H} \subseteq O$ such that $\gamma(\widetilde{H}) = H$. Being a finitely generated torsion-free abelian group, $\widetilde{H}$ is free of rank at most $[K : \QQ]$, which implies that $|H| \leq m^{[K : \QQ]}$. Since this is true for every finitely generated subgroup of $O/mO$, we conclude that $O/mO$ is finite of order $\leq m^{[K : \QQ]}$.

\begin{proof}[Proof of Theorem \ref{Thm:MSR}]
We begin by establishing two facts. As in Lemma \ref{Lem:main}, we let $K$ denote the field of fractions of $R$, and $O$ the integral closure of $R$ in $K$. Then 

\begin{equation}\label{Eq:A1}
[O^{\times} : R^{\times}] < \infty
\end{equation}
and
\begin{equation}\label{Eq:A2}
 [SL_2(O) : SL_2(R)] < \infty
\end{equation}

\noindent Both statements critically depend on the fact that there exists a nonzero ideal $I \subseteq O$ that is contained in $R$ (see Lemma \ref{Lem:main} (c)) and that the quotient $O/I$ is finite. The canonical homomorphism $O \to O/I$ induces group homomorphisms 
$$
\phi \colon O^{\times} \to (O/I)^{\times} \ \ \text{and} \ \ \psi \colon SL_2(O) \to SL_2(O/I). 
$$
Since the quotient $O/I$ is finite, the congruence subgroups $U(I) = \ker \phi$ and $SL_2(O , I) = \ker \psi$ are of finite index in $O^{\times}$ and $SL_2(O)$, respectively. For the proof of (\ref{Eq:A1}), we observe that $U(I) = O^{\times} \cap \bigl( 1 + I \bigr)$ is contained in $R^{\times}$. Indeed, we clearly have $U(I) \subseteq R$; since $U(I) = U(I)^{-1}$ we actually get the inclusion $U(I) \subset R^{\times}$. So, (\ref{Eq:A1}) follows. Similarly, 
one proves the inclusion $SL_2(O , I) \subseteq SL_2(R)$, yielding (\ref{Eq:A2}). 

It follows from (\ref{Eq:A1}) that condition (b) in the theorem is equivalent to the condition that the unit group $O^{\times}$ is infinite. Likewise, it follows from (\ref{Eq:A2}) and \cite[Proposition 7]{T} that condition (a) in the theorem is equivalent to the fact that the group $SL_2(O)$ is boundedly generated. Since it was proved in \cite{MRS,KMR} that $SL_2(O)$ is boundedly generated if and only if the unit group $O^{\times}$ is infinite, we obtain the equivalence of conditions (a) and (b) in the theorem. Finally, it follows from the Dirichlet Theorem for $S$-units (cf., for example, \cite[Theorem 3.12]{Nar}) that the only rings of algebraic $S$-integers with finitely many units are $\ZZ$ and $\mathcal{O}_d$ for $d < 0$. This yields the equivalence of conditions (b) and (c) and completes the proof of the theorem.  
\end{proof} 

We mention one immediate corollary. Recall that the property of being acylindrically hyperbolic is a very strong negation of bounded generation. For example, an acylindrically hyperbolic group $G$ cannot be boundedly generated by any amenable subgroups and has an infinite-dimensional space of quasi-characters and second bounded cohomology with coefficients in a broad range of $G$-modules, including $\RR$ and $\ell^p(G)$ for any $p\ge 1$ (see \cite{Osi16, HO}).

\begin{cor}
For any finitely generated subring $R$ of $\overline \QQ$, the group $SL_2(R)$ is either boundedly generated or acylindrically hyperbolic.
\end{cor}

\begin{proof}
If $SL_2(R)$ is not boundedly generated, then $R\subseteq \mathcal O_d$ for some integer $d<0$. It is well-known that the standard action $SL_2(\mathcal O_d)\act \mathbb H^3$ is proper; in particular, it satisfies the Bestvina--Fujiwara weak proper discontinuity condition \cite{BF}. Since $SL_2(\ZZ)$ is contained in $SL_2(R)$, the action $SL_2(R)\act \mathbb H^3$ is also non-elementary. Therefore, $SL_2(R)$ is acylindrically hyperbolic by \cite[Theorem 1.2]{Osi16}. 
\end{proof}

\vspace{5mm}

\noindent \textbf{Denis Osin: } Stevenson Center 1326, Department of Mathematics,  Vanderbilt University, Nashville, TN 37240, USA.

\noindent E-mail: \emph{denis.v.osin@vanderbilt.edu}

\bigskip

\noindent \textbf{Koichi Oyakawa: } 970 Evans Hall, Department of Mathematics, University of California, Berkeley, CA 94720, USA.

\noindent E-mail: \emph{koichi.oyakawa@berkeley.edu}

\bigskip

\noindent \textbf{Andrei Rapinchuk: } 307 Kerchof Hall, Department of Mathematics, University of Virginia, Charlottesville, VA 22903, USA. 

\noindent E-mail: \emph{asr3x@virginia.edu}

\end{document}